\documentclass[10pt,english]{amsart}
\usepackage{amssymb}
\usepackage{lipsum}
\usepackage{amsthm}

\newtheorem{theorem}{Theorem}[section]
\newtheorem{remark}[theorem]{Remark}
\newtheorem{question}[theorem]{Question}
\newtheorem{lemma}[theorem]{Lemma}
\newtheorem{proposition}[theorem]{Proposition}
\newtheorem{corollary}[theorem]{Corollary}
\newtheorem{Remark}[theorem]{Remark}
\newtheorem{Definition}[theorem]{Definition}

\newcommand{\undt}[1]{\underline{t}_{#1}}
\newcommand{\ring}[1]{\mathcal{R}_{#1}}
\newcommand{\calu}{\mathcal{U}}
\newcommand{\calv}{\mathcal{V}}
\newcommand{\calh}{\mathcal{H}}
\newcommand{\nalpha}{N_\alpha}

\newcommand{\ZZ}{\mathbb{Z}}
\newcommand{\FF}{k}
\newcommand{\CC}{\mathbb{C}}

\newcommand{\RR}{\mathbb{R}}

\newcommand{\NN}{\mathbb{N}}

\newcommand{\TT}{\mathbb{T}}

\newcommand{\KPhat}{\mathfrak K_P}
\newcommand{\APhat}{\mathfrak A_P}

\makeatletter
\g@addto@macro{\endabstract}{\@setabstract}
\newcommand{\authorfootnotes}{\renewcommand\thefootnote{\@fnsymbol\c@footnote}}
\makeatother
\begin{document}

\title[$L$-series values in Tate algebras]{Arithmetic of positive characteristic \\ $L$-series values in Tate algebras\textsuperscript{1}}
\author[B. Angl\`es \and F. Pellarin \and F. Tavares Ribeiro]{B. Angl\`es\textsuperscript{2} \and F. Pellarin\textsuperscript{3,4}   \and F. Tavares Ribeiro\textsuperscript{2}}

\address{
Universit\'e de Caen Basse-Normandie, 
Laboratoire de Math\'ematiques Nicolas Oresme, UMR 6139, 
Campus II, Boulevard Mar\'echal Juin, 
B.P. 5186, 
14032 Caen Cedex, France.
}
\email{bruno.angles@unicaen.fr, floric.tavares-ribeiro@unicaen.fr}\email[\emph{author of the appendix}]{florent.demeslay@unicaen.fr}

\address{ 
Institut Camille Jordan, UMR 5208
Site de Saint-Etienne,
23 rue du Dr. P. Michelon,
42023 Saint-Etienne, France
}
\email{federico.pellarin@univ-st-etienne.fr}

\maketitle

\begin{center}{\sc With an appendix by F. Demeslay\textsuperscript{2}.}\end{center}  

\footnotetext[1]{MSC Classification 11F52, 14G25, 14L05. Keywords: $L$-values in positive characteristic, log-algebraic theorem, class modules, Bernoulli-Carlitz fractions.}
\footnotetext[2]{Universit\'e de Caen Basse-Normandie, 
Laboratoire de Math\'ematiques Nicolas Oresme, UMR 6139, 
Campus II, Boulevard Mar\'echal Juin, 
B.P. 5186, 
14032 Caen Cedex, France.}
\footnotetext[3]{Institut Camille Jordan, UMR 5208
Site de Saint-Etienne,
23 rue du Dr. P. Michelon,
42023 Saint-Etienne, France.}
\footnotetext[4]{Supported by the ANR HAMOT}

\begin{abstract} The second author has recently introduced a new class of $L$-series in the arithmetic theory of function fields over finite fields. We show that the values at one of these  $L$-series encode arithmetic information of a generalization of Drinfeld modules defined over Tate algebras that we introduce (the coefficients can be chosen in a Tate algebra). This enables us to generalize Anderson's log-algebraicity Theorem and an analogue of 
Herbrand-Ribet Theorem recently obtained by Taelman.\end{abstract}
  
\tableofcontents

\section{Introduction}
We fix a finite field $\FF$ with $q$ elements; we denote by $p$ its characteristic. We further set $A=\FF[\theta]$ (polynomial ring with coefficients in $\FF$ in an indeterminate $\theta$) and $K=\FF(\theta)$  
(the field of fractions of $A$). We also consider the field $K_\infty=\FF((\frac{1}{\theta}))$, the completion of $K$ with respect to the place at infinity;
we write $|\cdot|$ for the absolute value of $K_\infty$ normalized by setting $|\theta|=q$. 
We denote by $\CC_\infty$ 
the completion 
of a fixed algebraic closure of $K_\infty$ and we denote by $K^{ac}$
the algebraic closure  of $K$ in $\CC_\infty$. More generally, for any subfield $L$ of $\CC_\infty$, we denote by $L^{ac}$
the algebraic closure of $L$ in $\CC_\infty$.

Carlitz \cite{CAR} introduced the so-called {\em Carlitz zeta values}
$$\zeta_C(n):=\sum_{a\in A_+}a^{-n}\in K_\infty,\quad n>0, \quad \text{$n$ integer,}$$
as some analogues, up to a certain extent, of the classical zeta values $$\zeta(n)=\sum_{k>0}k^{-n}\in\RR$$ ($n>1$ integer). In the definition of $\zeta_C(n)$, $A_+$ denotes the set of monic polynomials in $A$
and provides a kind of substitute of the set of positive integers. The Carlitz zeta values offer interesting analogies with the classical zeta values. Let us look at the archimedean example of the divergent series 
\begin{equation}\label{zetaone}
\zeta(1)=\sum_{k\geq 1}k^{-1}=\prod_{p}\left(1-\frac{1}{p}\right)^{-1}=\infty\end{equation}
which we have developed as a divergent eulerian product (running over the prime numbers $p$). 
For a commutative ring $R$ and a functor $G$ from 
$R$-algebras to $R$-modules,  we denote by $\operatorname{Lie}(G)$ the functor from $R$-algebras
to $R$-modules defined, for $B$ an $R$-algebra, by:
$$\operatorname{Lie}(G)(B)=\operatorname{Ker}(G(B[\epsilon]/(\epsilon^2))\rightarrow G(B)).$$

The local factor at $p$ in (\ref{zetaone}) is
$$\left(1-\frac{1}{p}\right)^{-1}=\frac{p}{p-1}=\frac{|\operatorname{Lie}(\mathbb{G}_m)(\ZZ/p\ZZ)|}{|\mathbb{G}_m(\ZZ/p\ZZ)|},$$ where $|\cdot|$ denotes here the cardinality of a set. The above 
cardinalities can also be seen as positive generators of {\em Fitting ideals} of finite $\ZZ$-modules (the background on Fitting ideals is recalled in \S \ref{localfactors}). 

In parallel, let $C$ be the {\em Carlitz functor}
from $A$-algebras to $A$-modules (see \S \ref{carlitzexponential} for 
the background on the Carlitz module). Then, for $P$ a {\em prime} of $A$ (\footnote{A prime of $A$ is an irreducible monic polynomial of $A$.}) the module $C(A/PA)$ is
a finite $A$-module and one can easily prove (in different ways; read Goss, \cite[Theorem 3.6.3]{GOS}, Taelman, \cite[Proposition 1]{TAE1}, see also Anderson and Thakur's paper \cite[Proposition 1.2.1]{AND&THA}) that $P-1$ is the monic generator of the Fitting ideal of $M$. For a finitely generated and torsion $A$-module $M$, $[M]_A$ denotes the monic generator of its
Fitting ideal. Then, $$[C(A/PA)]_A=P-1$$ and 
\begin{equation}
\label{localfactor}
\zeta_C(1)=\prod_{P}\left(1-\frac{1}{P}\right)^{-1}=\prod_{P}\frac{P}{P-1}=\prod_P\frac{[\operatorname{Lie}(C)(A/PA)]_A}{[C(A/PA)]_A}.
\end{equation}

The tensor powers of the Carlitz module functor introduced by Anderson and Thakur \cite{AND&THA} provide a way to interpret the values $\zeta_C(n)$ as well, and this can be viewed as
one of the main sources of analogies between the theory of the Carlitz zeta values and the values of the Riemann zeta function at integers $n\geq 2$.

Carlitz proves that 
for all $n>0$ divisible by $q-1$, $\zeta_C(n)$ is, up to a scalar factor of $K^\times$ (the multiplicative group of $K$), proportional to $\widetilde{\pi}^n$, where the quantity 
$\widetilde{\pi}$ is defined \cite[definition 3.2.7]{GOS} by
\begin{equation}\label{productpi}
\widetilde{\pi} =\sqrt[q-1]{\theta-\theta^q}\prod_{ i\geq 1} \left(1-\frac{\theta^{q^i}-\theta}{\theta^{q^{i+1}}-\theta}\right)\in
\theta\sqrt[q-1]{-\theta}\left(1+\frac{1}{\theta} \FF\left[\left[\frac{1}{\theta}\right]\right]\right),\end{equation}
unique up to multiplication by an element of $\FF^\times$, see Goss, \cite[Chapter 3]{GOS}.  We consider the {\em Carlitz exponential} $\exp_C$ (see \S \ref{carlitzexponential} for the background about this function). Carlitz proved the formula
\begin{equation}\label{formulaofcarlitz}
\exp_C(\zeta_C(1))=1.
\end{equation} 
Knowing that $1$ belongs to the domain of convergence of the Carlitz logarithm $\log_C$, composition inverse of $\exp_C$ (see \S \ref{carlitzlogdef}) and comparing 
the absolute values of $\log_C(1),\zeta_C(1),\widetilde{\pi},$ we see that 
the above formula is equivalent to
\begin{equation}\label{formulaofcarlitz2}\zeta_C(1)=\log_C(1).\end{equation}

Taelman \cite{TAE1} recently exhibited an appropriate setting to interpret the above formula as an instance of the {\em class number formula}. His approach, involving a new type of trace formula for certain variants of 
bounded continuous operators,
also relies on the formula (\ref{localfactor}).
He did this in the broader framework of {\em Drinfeld modules} defined over the ring of integers $R$ of a finite extension $L$ of $K$. Taelman associated, to such a Drinfeld 
module $\phi$, a finite $A$-module called the {\em class module} (of $\phi$ over $L$), and a finitely generated $A$-module called the {\em unit module} (of $\phi$ over $L$).
An {\em $L$-series value} $L(\phi/R)$ that he also defines is then equal to the product of the monic generator of the Fitting ideal of the class module times the regulator of the unit module (see Theorem 1 of loc. cit.).

In the case of $\phi=C$ the Carlitz module, and $L=K$, the $L$-series value is equal to $\zeta_C(1)$, the class module is trivial, 
and the regulator of the unit module is $\log_C(1)$, the Carlitz logarithm of $1$, yielding (\ref{formulaofcarlitz}).

\subsection{Drinfeld modules} We introduce and discuss a generalization of Drinfeld modules.

Classically, a Drinfeld module $\phi$ of rank $r$ is the datum of an injective $\FF$-algebra homomorphism 
$$\phi:A\rightarrow\operatorname{End}_{\FF-\text{lin}}(\CC_\infty),$$ uniquely defined by the image of $\theta$, that is, the value $\phi_\theta$ of $\phi$ at $\theta$, which is of the form
\begin{equation}\label{formdrinfeld}
\phi_\theta=\theta+\alpha_1\tau+\cdots+\alpha_r\tau^r,\end{equation}
where the {\em parameters} $\alpha_1,\ldots,\alpha_r$ are elements of $\CC_\infty$ with $\alpha_r\neq0$.

The {\em Tate algebra} $\mathbb{T}_s$ of dimension $s$ is the completion of the polynomial algebra $\CC_\infty[t_1,\ldots,t_s]$
for the {\em Gauss norm} (see \S \ref{tatealgebras}) and we have $\mathbb{T}_0=\CC_\infty$.

Observe that the automorphism $\CC_\infty\rightarrow\CC_\infty$ defined by $x\mapsto x^q$ extends to a continuous homomorphism of $\FF[t_1,\ldots,t_s]$-algebras $\tau:\mathbb T_s\rightarrow \mathbb T_s$ and  that $\FF[t_1,\ldots,t_s]=\{ f\in \mathbb T_s, \tau(f)=f\}.$

We use the $\FF[t_1,\ldots,t_s]$-linear automorphism $\tau$ of
$\mathbb{T}_s$ to define {\em Drinfeld $A[t_1, \ldots, t_s]$-modules of rank $r$ over $\mathbb{T}_s$}; a Drinfeld $A[t_1, \ldots, t_s]$-module $\phi$ of rank $r$ over $\mathbb{T}_s$ is an injective 
$\FF[t_1,\ldots,t_s]$-algebra homomorphism
$$\phi:A[t_1,\ldots,t_s]\rightarrow\operatorname{End}_{\FF[t_1,\ldots,t_s]-\text{lin.}}(\mathbb{T}_s),$$ with $\phi_\theta$ as in (\ref{formdrinfeld})  but where the parameters $\alpha_1,\ldots,\alpha_r$ are now allowed to be chosen in $\mathbb{T}_s$ and,
of course, $\alpha_r\neq0$; see \S \ref{DrinfeldmodulesoverTs}.

\subsubsection{$L$-series values in the case of rank one}

$L$-series values at one can be associated to Drinfeld $A[t_1,\ldots,t_s]$-modules 
defined over $A[t_1,\ldots,t_s]$ by extending the construction of Taelman in \cite{TAE2} by using 
Fitting ideals on the model of the eulerian product (\ref{localfactor}). 
Typically, the $L$-series values of Taelman are elements of $\CC_\infty$ while our $L$-series values are elements of Tate algebras of any dimension.

The construction of these values is explained in \S \ref{localfactors} in the case of a Drinfeld $A[t_1,\ldots,t_s]$-module $\phi$ of rank one; 
we momentarily denote by $\mathcal{L}(\phi)$ the Taelman $L$-series value of $\phi$ at one (a slightly different notation will be used later in the text).
In this case, and with the unique parameter $\alpha=\alpha_1$ in $A[t_1,\ldots,t_s]$, the construction of $L$-series values becomes in fact very explicit.
We consider the maps
$$\rho_\alpha:A\rightarrow\FF[t_1,\ldots,t_s]$$
defined by $$\rho_\alpha(b)=\operatorname{Res}_\theta(b,\alpha),$$ where $\alpha$ is a polynomial of $A[t_1,\ldots,t_s]\setminus\{0\}$ and where $\operatorname{Res}_\theta(P,Q)$ denotes  the resultant of two polynomials $P,Q$ in $\theta$ (\cite[\S IV.8]{LAN}). 
We recall that if $F$ is a field and $X$ is an indeterminate over $F,$ if $P(X) =\sum_{i=1}^n p_{n-i}X^i,Q(X)=\sum_{j=1}^m q_{m-j}X^j$ are elements in $ F[X]$ then $\operatorname{Res}_X(P,Q)=(-1)^{nm}q_0^n\prod_{j=1}^mP(\beta_j),$ where $ \beta_1, \ldots, \beta_m$ are the roots of $Q(X)$ in some algebraic closure of $F$ (\cite[\S IV.8, Proposition 8.3]{LAN}). For example,  if $\alpha =(t_1-\theta)\cdots (t_s-\theta),$ for $b\in A$ we get $\rho_\alpha(b)=b(t_1)\cdots b(t_s).$ We view the map $\rho_\alpha$
as a kind of higher dimensional generalization of a Dirichlet character.

Then, we set:
\begin{equation}\label{ourlseries}
L(n,\phi)=\sum_{b\in A_+}\operatorname{Res}_\theta(b,\alpha)b^{-n}=\prod_P\left(1-\frac{\rho_\alpha(P)}{P^n}\right)^{-1}\end{equation}
(the product is taken over the primes $P$ of $A$). It is easy to see that this series converges in the Tate algebra $\mathbb{T}_s$; see \S\ref{Lseriesvalues}. 
We show in  Proposition \ref{prodexpansion} that
$$L(1,\phi)=\mathcal{L}(\phi).$$
If $\alpha=(t_1-\theta)\cdots (t_s-\theta),$ then we recover Goss abelian $L$-functions by specializing the variables $t_i$ to elements in $k^{\text{ac}}.$ 
With our definition of $L$-series values, we will cover already many 
$L$-series values studied by Goss, as well as in \cite{PEL2,GOS2,PER1} and \cite{ANG&PEL}. 
In \S \ref{connection} we have included some further remarks on the link 
existing between our $L$-series values and the global $L$-functions
of Goss, Taguchi-Wan and B\"ockle-Pink (see \cite{GOS,BOE, BOE&PIN, TAG&WAN}). These remarks may be of help for the reader to 
orientate in the literature.

\subsection{The main results }
The $L$-series values that we study being elements of the Tate algebras $\mathbb{T}_s$, they have the double status of ``numbers" and ``functions". As numbers, the indeterminates
$t_1,\ldots,t_s$ are unspecified and the series $L(n,\phi)$ are handled as 
elements of the ring $\mathbb{T}_s$.  As functions, the variables $t_1,\ldots,t_s$ can be specialized and the analytic 
properties of the functions $L(n,\phi)$ can be used to obtain arithmetic information
e.g. on Carlitz zeta values.
\subsubsection{$L$-series values as ``numbers"} 
Let $\phi$ be a Drinfeld $A[t_1, \ldots, t_s]$-module of rank one over $\mathbb{T}_s$ of parameter $\alpha$
(that is, $\phi_\theta=\theta+\alpha\tau$) in $A[t_1,\ldots,t_s]$. The module $\phi$ is a global object defined over $A[t_1,\ldots,t_s]$. Considering it over $\mathbb{T}_s$ means looking at its realization at the place $\infty$. 
We introduce, in \S\ref{classunitmodules},
the {\em class $A[t_1,\ldots,t_s]$-module} $H_\phi$ and the {\em unit $A[t_1,\ldots,t_s]$-module} $U_\phi$ 
associated to $\phi$. This provides, in a certain way,
a generalization of the constructions of Taelman paper \cite{TAE2}. Note that, while Taelman's class modules are vector spaces of finite dimension over $\FF$, the class module  $H_\phi$ is of finite rank over $\FF[t_1,\ldots,t_s]$. 
We said ``in a certain way" because we only deal here
with the case of rank $r=1$. The general case, however, can be handled 
as well, as Demeslay does, in a forthcoming work.
In  the case of $\alpha=(t_1-\theta)\cdots(t_s-\theta)$, the class modules are ``generic" in the sense that they can be used to interpolate 
Taelman's class modules (over the cyclotomic extensions).

In the following, we denote the set of variables $\{t_1,\ldots,t_s\}$ by $\undt{s}$ (or by $\undt{}$ 
when the number of variables is understood in the context and when there is a real need to simplify the notation).

The $\FF(\undt{s})$-vector 
space $$\calh_\phi=\FF(\undt{s})\otimes_{\FF[\undt{s}]} H_\phi$$ is of finite dimension
and endowed with the structure of module over the ring $$\FF(\undt{s})[\theta]=\FF(\undt{s})\otimes_\FF A$$ that we denote by $\mathcal{R}$ to simplify certain formulas (see Corollary \ref{corollaryHszerobis}). Let $$[\calh_\phi]_{\ring{}}\in \ring{}$$ be 
the monic generator of its Fitting ideal.
We will see (Proposition \ref{proposition2}) that the vector space $$\calu_\phi=\FF(\undt{s})\otimes_{\FF[\undt{s}]} U_\phi$$ is a free $\ring{}$-module of rank one,
to which we can associate a {\em regulator} $$[\ring{}: \calu_\phi]_{\ring{}}.$$
Then, the class number formula for the $L$-series value $L(1,\phi)$ (Theorem \ref{theorem2})
can be obtained (the notation will be made more precise
later in this text):

\medskip
 {\em 

$$L(1, \phi) =[\calh_{\phi}]_{\ring{}} [\ring{}: \calu_\phi]_{\ring{}}.$$
}

This result  is deduced from 
Theorem \ref{theoremA1} of the Appendix, by F. Demeslay. 
The originality of our approach is the use we make of the above class number formula.

The properties of the {\em exponential function} $\exp_\phi$ associated to $\phi$ (\S\ref{explog}) strongly 
influence the properties of $L(1,\phi)$, $H_\phi$ and $U_\phi$. By the results obtained in \S\ref{sectioncnf}, \ref{abeliandrinfeld}, and  \ref{definedoverA}, for $\phi$ a Drinfeld $A[\undt{s}]$-module 
of rank one defined over  $A[\undt{s}]$, with $\phi_\theta=\theta+\alpha\tau$ with $\alpha \in A[\undt{s}]\setminus \{0\},$ we have that  $\exp_\phi(L(1,\phi))\in\mathbb{T}_s$ belongs in fact to $A[\undt{s}]$. It is a  torsion point for the structure of $A[\undt{s}]$-module induced by $\phi$
if and only if the opposite of the parameter $-\alpha\in A[\undt{s}]$ is a monic polynomial in $\theta$ of degree $r\equiv1\pmod{q-1}$. We will see that this latter  condition  is equivalent to the fact that the  function $\exp_\phi$ is surjective, and its kernel has non-trivial intersection with  $\mathbb{T}_s(K_\infty)$, the completion of $K_\infty[\undt{s}]$ for the 
Gauss norm. Furthermore by the results of \S\ref{subsectionb} and Theorem \ref{theorem4}, we have:
 \medskip

 {\em 
If $-\alpha$ is monic as a polynomial in $\theta$, of degree $r\equiv1\pmod{q-1}$
with $r\geq 2$, then $\exp_\phi(L(1,\phi))=0$. In the special case of $\phi$ 
of parameter $\alpha=(t_1-\theta)\cdots(t_s-\theta)$,
 for $s\geq 2, s\equiv 1\pmod{q-1},$ we have the formula:
$$L(1,\phi)=\frac{\widetilde{\pi}\mathbb{B}_s}{\omega(t_1)\cdots\omega(t_s)},$$
where $\mathbb{B}_s\in A[\undt{s}]$ is the monic generator of the Fitting ideal of the $A[\undt{s}]$-module
$H_\phi$ and $\omega$ is the Anderson-Thakur function introduced in \cite{AND&THA}.}

\medskip

We have few explicit examples of polynomial $\mathbb{B}_s$. 
Here are some:
\begin{eqnarray*}
\mathbb{B}_q&=&1,\\
\mathbb{B}_{2q-1}&=&\theta-\sum_{1\leq i_1<\cdots<i_{q}\leq 2q-1}\prod_{k=1}^{q}t_{i_k},\quad (q>2).
\end{eqnarray*}
 Let $\phi$ be the Drinfeld $A[t_1]$-module of parameter $\alpha = t_1-\theta,$
if we set $\mathbb{B}_1=\frac{1}{\theta-t_1},$ by Lemma \ref{corr=beta=1}, we have the formula:
$$L(1,\phi)=\frac{\widetilde{\pi}\mathbb{B}_1}{\omega(t_1)}.$$

As a consequence of the class number formula, we shall also mention the 
{\em log-algebraicity Theorem of Anderson}, in the case of the Carlitz module, see \cite[Theorem 3 and Proposition 8 (I)]{AND2} (so, surprisingly, the class number formula implies Anderson's log-algebraic theorem for the Carlitz module).
In fact, we prove in \S\ref{anderson} a result which can be interpreted as an {\em operator theoretic version}, thus a refinement of Anderson Proposition 8 (I) loc. cit. 
 We introduce a class of formal series in infinitely many  indeterminates $X_i,\tau(X_i),\ldots,Z,\tau(Z),\ldots$ ($i=1,\ldots,r$) by setting:
$$\mathcal{L}_r(X_1,\ldots,X_r;Z)=\sum_{d\geq 0}\left(\sum_{a\in A_{+,d}}C_a(X_1)\cdots C_a(X_r)a^{-1}\right)\tau^d(Z),$$
where $A_{+,d}$ denotes the set of monic polynomials of degree $d$ and $C_a(X_i)$ denotes a certain
polynomial in $X_i,\tau(X_i),\ldots$ obtained from
the action of the Carlitz module evaluated at $a$ on the indeterminate $X_i$; for example, $C_\theta(X_1)=
\theta X_1+\tau(X_1)$.
We have (Theorem \ref{theoanderson}):

\medskip

{\em  ${\exp}_C(\mathcal{L}_r(X_1,\ldots,X_r;Z))\in A[X_i,\tau(X_i),\ldots, Z,\tau(Z),\ldots,\, i=1,\ldots,r].$ }

\medskip

If we substitute, in the above result, $X_1=\cdots=X_r=X$ and $\tau^n(X)=X^{q^n},\tau^n(Z)=Z^{q^n}$ for all $n\geq 0$, we recover Anderson's original result asserting that 
$$\exp_C\left(\sum_{a\in A_+}Z^{q^{\deg_\theta(a)}}a^{-1}C_a(X)^r\right)\in A[X,Z].$$ 
\medskip 

\subsubsection{ $L$-series values as "functions"} The evaluation of $L$-series values is 
the necessary step to deduce from the above results, arithmetic results on the  values of Goss abelian $L$-series. One of the main novelties of our work is that we are able to study the isotypic components of Taelman's class modules in families.
Let $\chi$ be a Dirichlet character of {\em type} $s$ such that 
$s\equiv1\pmod{q-1}$ and conductor $a\in A_+$ (see \S \ref{somesettings}). Let us denote by $\FF_a$ the subfield of $\CC_\infty$ obtained adjoining to $\FF$ the roots of
$a$, by $K_a$ the $a$-th {\em cyclotomic field} and by $\Delta_a$ the Galois group 
of $K_a$ over $K$. We denote by $H_a$ the {\em Taelman class $A$-module}
associated to the Carlitz module and relative to the extension $K_a/K$. This
is a finite $\FF[\Delta_a]$-module. 
Let $e_\chi$ be the {\em idempotent element} of $\FF_a[\Delta_a]$
associated to $\chi$. Then, the $\chi$-isotypic component 
$$H_\chi=e_\chi( H_a\otimes_{\FF}\FF_a)$$
is a finite $A[\FF_a]$-module endowed with a suitable structure of
$\FF_a[\Delta_a]$-module. The evaluation map $\operatorname{ev}_\chi$  is described 
in \S\ref{evaluationdirichletcharacters}; it is obtained by substituting 
the variables $t_i$ (for $i=1,\ldots,s$) by appropriate roots of unity
chosen among the roots of the conductor $a$ in $\FF^{ac}.$ By Corollary  \ref{theorem6}, the Fitting ideal of the $A[\FF_a]$-module $H_\chi$ is generated by   $\operatorname{ev}_\chi({\mathbb{B}_s}).$ A similar result is obtain when the type $s$ satisfies $s\not \equiv 1 \mod q-1$, see Theorem \ref{theorem7}. In \S\ref{bernoullicarlitzgen} we associate to 
our character $\chi$ certain {\em generalized Bernoulli-Carlitz numbers}
denoted by $\operatorname{BC}_{i,\chi^{-1}}$. These are elements
of the compositum $K(\FF_a)$ of $\FF_a$ and $K$ in $\CC_\infty$.

Let us write $$\chi=\vartheta_P^N\widetilde{\chi},$$ where $P$ is a 
prime dividing the conductor $a$ of $\chi$ (so that $a=Pb$ with 
$P$ not dividing $b$), $\widetilde{\chi}$ is a Dirichlet character
of conductor $b$, $\vartheta_P$ is the Teichm\"uller character
associated to $P$, and $N$ is an integer between $0$ and $q^d-2$,
with $d$ the degree of $P$. 

Let $P$ be a prime and $\KPhat$ the completion of $K$ at $P$.
If $\APhat$ is the valuation ring of $\KPhat$, 
the valuation ring of the 
field $\KPhat(\FF_a)$ is $\APhat[\FF_a]$. 
We obtain  a generalization of 
Herbrand-Ribet-Taelman Theorem of \cite{TAE2} (Theorem \ref{theorem8}).
Here, we suppose that $N\geq 2$ if $\widetilde{\chi}=1$:

\medskip

{\em  The $\APhat[\FF_a][\Delta_a]$-module $e_\chi(H_a\otimes_A\APhat[\FF_a])$ is  non-trivial if and only 
if $$\operatorname{BC}_{q^d-N,\widetilde{\chi}^{-1}}\equiv0\pmod{P}.$$}

\medskip

The original result of Taelman \cite[Theorem 1]{TAE2} corresponds to the case in which $\widetilde{\chi}$ 
is the trivial character. Our demonstration of Theorem \ref{theorem8} is inspired by  the alternative proof of  Herbrand-Ribet Theorem for function fields  given in   \cite{ANG&TAE}.

\section{Notation and background}

The basic list of notation of this paper is the following:

\bigskip

\begin{itemize}
\item $\NN$: the set of non-negative integers.
\item $k$: a fixed finite field with $q$ elements.
\item $p$: the characteristic of $\FF$.
\item $\theta$: an indeterminate over $\FF$.
\item $A$: the polynomial ring $\FF[\theta]$.
\item $A_+$: the set of monic elements in $A.$
\item For $n\in \mathbb N,$ $A_{+,n}$ denotes the set of monic elements in $A$ of degree $n.$
\item $K=\FF(\theta)$: the fraction field of $A$.
\item $\infty$: the unique place of $K$ which is a pole of $\theta$ and $v_{\infty}$ is the discrete valuation on $K$ corresponding to the place $\infty.$ The valuation $v_{\infty}$ is normalized such that $v_{\infty}(\theta)=-1.$
\item $K_{\infty}=\FF((\frac{1}{\theta}))$: the completion of $K$ at $\infty.$
\item $\CC_{\infty}$: a fixed  completion of an algebraic closure of $K_{\infty}$. The unique valuation on $\CC_{\infty}$ that extends  $v_{\infty}$ will still be denoted by $v_{\infty}.$
\item $|\cdot|$: the absolute value of $\CC_\infty$ defined by 
$|\alpha| =q^{-v_{\infty}(\alpha)}$ for $\alpha \in \mathbb C_{\infty}.$ 
\item $L^{\text{ac}}$: the algebraic closure in $\CC_\infty$ of a field $L\subset\mathbb C_{\infty}.$
\item $R^\times$: the group of invertible elements of a ring $R$.
\item $\operatorname{Frac}(R)$: the fraction field of a domain $R$.
\end{itemize}

\bigskip

In all the following we keep using a set of variables $\{t_1,\ldots,t_s\}$ for various choices of $s\geq 0$. We recall that this set is denoted by $\undt{s}$ or by $\undt{}$ if the value of $s$ is understood. For example,
$\FF(\undt{s})$ denotes the field $\FF(t_1,\ldots,t_s)$. In particular, we have $\undt{0}=\emptyset$ and 
$\FF(\undt{0})=\FF$. If $s=1$, we will often write $t$ instead of $t_1$ and $\undt{1}$. We will also use the following notation, where $R$ is a $\FF$-algebra:

\bigskip

\begin{itemize}
\item $R[\undt{s}]$: the ring $\FF[\undt{s}]\otimes_\FF R.$ We observe that $R[\undt{0}]=R$.
\item $\mathbb T_s:$  the Tate algebra in the variables $t_1, \ldots, t_s$ with coefficients in $\mathbb C_{\infty}$. We observe that $\TT_0=\CC_\infty$.
\item $\mathbb T_s(K_{\infty})$: the ring $\mathbb T_s \cap K_{\infty}[[\undt{s}]].$
\item $K(\undt{s})_{\infty}$: the field $\FF(\undt{s})((\frac{1}{\theta}))$. 
\item $\ring{}$: the ring $\FF(\undt{s})[\theta]=\FF(\undt{s})\otimes_\FF A$ (this notation will not be used systematically).
\end{itemize}

\bigskip

We observe that $$K_{\infty}[\undt{s}]\subsetneq \mathbb T_s(K_{\infty})\subsetneq \operatorname{Frac}(\mathbb T_s(K_{\infty}))\subsetneq 
K(\undt{s})_{\infty}.$$

In fact, $K(\undt{s})_{\infty}$ is the completion of the fraction field of $\TT_s(K_\infty)$.

\subsection{The Carlitz exponential.}\label{carlitzexponential} The Carlitz exponential is the function
$$\exp_C:\CC_\infty\rightarrow\CC_\infty$$
defined by 
$$\exp_C(X)=\sum_{i\geq 0} \frac{X^{q^i}}{D_i},\quad X\in\CC_\infty,$$
where $(D_i)_{i\geq 0}$ is the sequence of $A$ defined by $D_0=1$ and, for $i\geq 1$, $$D_i=(\theta^{q^i}-\theta)D_{i-1}^q.$$
This function, $\FF$-linear, is entire from the identity $|D_i|=q^{iq^i}.$  In particular,  $\exp_C$ is surjective. 

The kernel of $\exp_C$
is the $A$-module $\widetilde{\pi}A,$ where $\widetilde{\pi}$ is defined by the infinite product 
(\ref{productpi}), see \cite[Corollary 3.2.9]{GOS}. We have
$$|\widetilde{\pi}|=q^{\frac{q}{q-1}}.$$
We can expand the function $\exp_C$ in a convergent infinite product:
$$\exp_C(X)=X\prod_{a\in A\setminus\{0\}}\left(1-\frac{X}{\widetilde{\pi}a}\right),\quad X\in\CC_\infty.$$ We observe  that $\exp_C$ induces an isometric $\FF$-automorphism of the disk $$D_{\CC_\infty}(0,q^{\frac{q}{q-1}})=\{z\in\CC_\infty;|z|<q^{\frac{q}{q-1}}\}.$$

\subsubsection{The Carlitz module} 
The $\CC_\infty$-algebra of the $\FF$-linear algebraic endomorphisms of 
$\mathbb{G}_{a,\CC_\infty}$
$$\operatorname{End}_{\FF-\text{lin.}}(\mathbb{G}_{a, \CC_\infty})$$ can be identified with the skew polynomial ring $\CC_\infty[\tau]$ whose elements 
are the finite sums $\sum_{i\geq 0}c_i\tau^i$ with the $c_i$'s in $\CC_\infty$, subject to the product rule defined 
by $\tau x=x^q\tau$ for all $x\in\CC_\infty$. If $X\in\CC_\infty$ and $P=\sum_{i=0}^dP_i\tau^i$ is an element of $\CC_\infty[\tau]$, the {\em evaluation} of $P$ at $X$ is defined by setting $$P(X)=\sum_{i=0}^dP_iX^{q^i}.$$ For example, 
the evaluation of $\tau$ at $X$ is $\tau(X)=X^q$.

The {\em Carlitz module}
is the unique $\FF$-algebra homomorphism $$C:A\rightarrow\operatorname{End}_{\FF}(\mathbb{G}_{a,\mathbb{C}_{\infty}})$$
determined by $$C_\theta=\theta+\tau.$$  
If $a\in A_{+,d}$, we denote by $C_a$ the image of $a$ via $C$. We have $$C_a=a_0\tau^0+a_1\tau^1+\cdots+a_{d-1}\tau^{d-1}+\tau^d$$ with $a_0=a$, and
if $X\in\CC_\infty$, we will write, in particular,
$$C_a(X)=a_0X+a_1X^q+\cdots+a_{d-1}X^{q^{d-1}}+X^{q^d}.$$ This endows $\CC_\infty$ with a structure of $A$-module
that will be denoted by $C(\CC_\infty)$ and we have $$C_a(\exp_C(X))=\exp_C(aX)$$ for all $a\in A$ and $X\in\CC_\infty$.
The Carlitz module $C$ allows to make the exact sequence of $\FF$-vector spaces $0\rightarrow\widetilde{\pi}A\rightarrow\CC_\infty\rightarrow\CC_\infty\rightarrow0$ induced by $\exp_C$ into an 
exact sequence of $A$-modules $$0\rightarrow\widetilde{\pi}A\rightarrow\CC_\infty\rightarrow C(\CC_\infty)\rightarrow0.$$

\subsubsection{The Carlitz logarithm}\label{carlitzlogdef} The Carlitz logarithm is the rigid analytic function defined, for $X\in \CC_\infty$ such that $|X|<q^{\frac{q}{q-1}}$,
by the convergent series
$$\log_C(X)=\sum_{i\geq 0}\frac{X^{q^i}}{l_i},$$
where $(l_i)_{i\geq 0}$ is the sequence defined by $l_0=1$ and, for $i\geq 1$, $$l_i=(\theta - \theta^{q ^i})l_{i-1}.$$

The convergence property is due to the fact that $|l_i|=q^{q\frac{q^i-1}{q-1}}$.
We then have, for all $X\in \CC_\infty$ such that $|X|<q^{\frac{q}{q-1}}$, 
\begin{equation}\label{isometrycarlitz1}|X|=|\exp_C(X)|=|\log_C(X)|\end{equation}
and
\begin{equation}\label{isometrycarlitz2}\log_C(\exp_C(X))= \exp_C(\log_C(X))=X.\end{equation}

\subsubsection{The Carlitz torsion.}\label{Carlitztorsion}
For $a\in A_{+},$ we set $$\lambda_a=\exp_C\left(\frac{\widetilde{\pi}}{a}\right)\in\CC_\infty.$$ 
The subfield of $\CC_\infty$
$$K_a=K(\lambda_a),$$ a finite extension of $K$, will be called the $a$-th {\em cyclotomic function field}. A reference for the basic theory of these 
fields is \cite[Chapter 12]{ROS}. Here, we recall that $K_a/K$ is a finite abelian extension unramified outside $a$ and $\infty$.
Its Galois group $$\Delta_a=\operatorname{Gal}(K_a/K)$$ is isomorphic to the unit group $$\left(\frac{A}{aA}\right)^\times.$$ 
Then, the isomorphism is explicitly given in the following way: if $b\in A$  is relatively prime with $a,$ there exists a unique $\sigma_b \in \Delta_a$ such that
\begin{equation}\label{sigma}
\sigma_b(\lambda_a) =C_b(\lambda_a).
\end{equation}

\subsection{Tate algebras}\label{tatealgebras}
We use the above conventions and notation. Let $s$ be in $\NN$. Let  $L$ be an extension of $K_{\infty}$ in $\CC_\infty$ such that $L$ is complete with respect to $v_{\infty}|_L.$ Let us consider a polynomial $f\in L[\undt{s}]$, expanded as a finite sum $$f=\sum_{i_1, \ldots,i_s\in \mathbb N}x_{i_1, \ldots, i_s}t_1^{i_1}\cdots t_s^{i_s},\quad x_{i_1, \ldots, i_s}\in L.$$ We set 
$$v_{\infty}(f) =\inf\{ v_{\infty}(x_{i_1, \ldots, i_s}), \,  i_1, \ldots, i_s\in \mathbb N\}.$$
We then have, for $f,g\in L[\undt{s}]$:
$$v_{\infty} (f+g)\geq\inf(v_{\infty}(f), v_{\infty}(g)).$$ Furthermore, we have $$v_{\infty}(fg)=v_{\infty}(f)+v_{\infty} (g)$$ so that $v_\infty$ is a valuation, called the {\em Gauss valuation}. 

Let us set, for $f\in L[\undt{s}]$, $\|f\|=q^{-v_\infty(f)}$ if $f\neq 0$ and $\| 0\| =0$. We have $\|f+g\|\leq\max\{\|f\|,\|g\|\}$, $\|fg\|=\|f\|\|g\|$ and $\|f\|=0$ if and only if $f=0$; the function $\|\cdot\|$ is an $L$-algebra norm on $L[\undt{s}]$ and an absolute value, called the {\em Gauss absolute value}.  

We denote by $\mathbb T_s(L)$ the completion of $L[\undt{s}]$ with respect to $\|.\|.$
When $s=1$, we also write $\mathbb{T}(L)$ for $\mathbb{T}_1(L)$ and we observe that $\mathbb{T}_0(L)=L$. 
Equipped with the Gauss norm, $\mathbb{T}_s(L)$ is an $L$-Banach algebra  that can be identified with the set of formal series of $f\in L[[\undt{s}]]$ such that, writing $$f=\sum_{i_1, \ldots,i_s\in \mathbb N}x_{i_1, \ldots, i_s}t_1^{i_1}\cdots t_s^{i_s},\quad x_{i_1, \ldots, i_s}\in L,$$ we have $$\lim_{i_1+\cdots +i_s\rightarrow +\infty}x_{i_1, \ldots, i_s}=0.$$ 
The Gauss norm of $f$ as above is then given by
$$\|f\|=\sup\{|x_{i_1, \ldots, i_s}|,(i_1,\ldots,i_s)\in\NN^s\},$$ and the supremum is a maximum.
When $L=\CC_\infty$, we shall write $\mathbb{T}_s,\mathbb{T}$ instead of $\mathbb{T}_s(\CC_\infty),\mathbb{T}_1(\CC_\infty)$.
We refer the reader to \cite[Chapter 3]{PUT} for the basic properties of Tate algebras.

We denote by $\mathfrak{o}_L$ the valuation ring of $L$ (whose elements 
$x$ are characterized by the fact that $|x|\leq 1$). We denote by $\mathfrak{m}_{L}$ 
the maximal ideal of $\mathfrak{o}_L$ whose elements $x$ are such that $|x|<1.$ Then, the field $\overline{L}= L\cap k^{ac}$ satisfies $\overline{L}\simeq \frac{\mathfrak{o}_L}{\mathfrak{m}_L}.$ We further denote by $\mathfrak{o}_{\mathbb{T}_s(L)}$ the subring of elements $f\in\mathbb{T}_s(L)$ such that $\|f\|\leq 1$ and by $\mathfrak{m}_{\mathbb{T}_s(L)}$
the prime ideal of $\mathfrak{o}_{\mathbb{T}_s(L)}$ whose elements are the $f$ such $\|f\|<1$. Then, we have that
$$\overline{\mathbb{T}_s(L)}:=\frac{\mathfrak{o}_{\mathbb{T}_s(L)}}{\mathfrak{m}_{\mathbb{T}_s(L)}}\simeq\overline{L}[\undt{s}].$$

If $L/K_{\infty}$ is a finite extension which is complete, let $\pi_L$ be a uniformizer of $L$. 
Then, we have that $L=\overline{L}((\pi_L))$, $\mathfrak{o}_L=\overline{L}[[\pi_L]]$. 
In particular:
$$\mathbb T_s(L)=\overline{L}[\undt{s}]((\pi_L)).$$

\subsubsection{$\FF[\undt{s}]$-linear endomorphisms of $\mathbb{T}_s$.}  We denote by $\tau$ the unique
$\FF[\undt{s}]$-linear automorphism $\TT_s\rightarrow\TT_s$ such that the restriction $\tau|_{\CC_\infty}$ is the automorphism of $\CC_\infty$ defined by $x\mapsto x^q$.
Explicitly, it can be computed as follows. For $f\in \mathbb T_s$ with 
$$f=\sum_{i_1, \ldots,i_s\in \mathbb N}x_{i_1, \ldots, i_s}t_1^{i_1}\cdots t_s^{i_s},\quad x_{i_1, \ldots, i_s}\in \mathbb C_{\infty},$$  we set
$$\tau(f)=\sum_{i_1, \ldots,i_s\in \mathbb N}x_{i_1, \ldots, i_s}^qt_1^{i_1}\cdots t_s^{i_s}.$$ This is a $\FF[\undt{s}]$-linear
automorphism of $\mathbb{T}_s$. In fact, $\tau$ is also an automorphism for the structure of $\FF[\undt{s}]$-algebra of $\mathbb{T}_s$.  If we set $$\mathbb{T}_s^{\tau=1}= \{ f\in \mathbb T_s, \tau (f)=f\},$$ we have  $\mathbb{T}_s^{\tau=1}=\FF[\undt{s}].$ Observe that:
$$\|\tau^n(f)\|=\|f\|^{q^n},\quad n\geq 0,\quad f\in\mathbb{T}_s.$$

With the action of $\tau$ on $\mathbb{T}_s$, we have the non-commutative skew polynomial rings $\mathbb{T}_s[\tau]$  and $\mathbb{T}_s[[\tau]]$. The latter is, as a set, constituted of the formal series $\sum_{i\geq 0}f_i\tau^i$ with $f_i\in\mathbb{T}_s$ for all $i$, and 
the elements of the former are the formal series whose sequences of coefficients are eventually zero. The 
commutation rule defining the product is given by $$\tau f=\tau(f)\tau,$$ for $f\in\mathbb{T}_s$. Moreover, the ring $\mathbb{T}_s[\tau]$ acts on $\mathbb{T}_s$: if $P=\sum_{i=0}^dP_i\tau^i\in \mathbb{T}_s[\tau]$ and $f\in\mathbb{T}_s$, then
we set
$$P(f)=\sum_{i=0}^dP_i\tau^i(f)\in\mathbb{T}_s.$$


\section{Drinfeld $A[t_1, \ldots,t_s]$-modules over $\mathbb{T}_s$.}\label{DrinfeldmodulesoverTs}

 \begin{Definition}\label{drinfelddef}{\em Let $r\geq 1$ be an integer. A {\em Drinfeld $A[\undt{s}]$-module  over $\mathbb T_s$} is a homomorphism of 
 $\FF[\undt{s}]$-algebras $$\phi:A[\undt{s}]\rightarrow \mathbb{T}_s[\tau]$$ defined by $$\phi_{\theta}= \theta+\alpha_1\tau+\cdots+\alpha_r\tau^r,$$ for an integer $r>0$ and elements $\alpha_1,\ldots,\alpha_r \in \mathbb T_s$ with $\alpha_r$ non-zero. The integer $r$ is the {\em rank} of $\phi$.
 The vector  $$\underline{\alpha}=(\alpha_1,\ldots,\alpha_r)\in\mathbb{T}_s^r$$ is the {\em parameter} of $\phi.$
 If $r=1$, we identify the parameter with its unique entry $\alpha_1$.
}\end{Definition}

Given a Drinfeld $A[\undt{s}]$-module $\phi$ of rank $r$ over $\mathbb{T}_s$,
if $M$ is a sub-$\FF[\undt{s}]$-module of $\mathbb T_s$ such that $\phi_{\theta}(M)\subset M,$  we denote by $\phi(M)$ the $\FF[\undt{s}]$-module $M$ equipped with the $A[\undt{s}]$-module structure induced by $\phi.$ In particular, we will often work
in the module $\phi(\mathbb{T}_s)$. This notation should not lead to confusion since $\phi(A[\undt{s}])$ will \emph{always} denote $A[\undt{s}]$ equipped with the new module structure, and \emph{not} the image of $\phi$ in $\mathbb{T}_s[\tau]$.

If $s\leq s'$ then we have the embedding $\mathbb{T}_s\subset\mathbb{T}_{s'}$ induced by the inclusion
$$\FF[\undt{s}]\subset\FF[t_1,\ldots,t_s,t_{s+1},\ldots,t_{s'}]=\FF[\undt{s'}].$$
Every Drinfeld $A[\undt{s}]$-module over $\mathbb T_s$
can be extended in a natural way to a Drinfeld $A[\undt{s'}]$-module over $\mathbb T_{s'}$ of the same rank, which will be denoted again by $\phi$
for the sake of simplicity.

\begin{Definition}\label{defdrinfeldmodule}
{\em Let $\phi, \phi'$ be two Drinfeld $A[\undt{s}]$-modules over $\mathbb T_s.$ We 
say that $\phi$ is {\em isomorphic} to $\phi'$ if there exists $u\in \mathbb T_s^\times$ ($\mathbb{T}_s^\times$ denotes the multiplicative group of the units of $\mathbb{T}_s$) such that, in $\mathbb T_s[\tau]$:
$$\phi_{\theta} u= u\phi'_{\theta}.$$ If $\phi$ and $\phi'$ are isomorphic Drinfeld modules, they must have the same rank and
we shall also write $\phi\cong\phi'$.}\end{Definition} 
\begin{Remark}{\em When two Drinfeld $A[\undt{s}]$-modules are isomorphic, it is understood that they are isomorphic {\em over $\mathbb{T}_s$}.}\end{Remark}

Let $\phi,\phi'$ be Drinfeld modules of rank $r>0$ over $\mathbb{T}_s$ of respective parameters $$\underline{\alpha}=(\alpha_1,\ldots,\alpha_r),\quad\underline{\alpha'}=(\alpha_1',\ldots,\alpha_r')\in\mathbb{T}_s.$$Then, the condition
$\phi\cong\phi'$ amounts to the existence of $u\in\mathbb{T}_s^\times$ such that
$$\alpha_i\tau^i(u)=\alpha_i'u,\quad i=1,\ldots,r.$$

\begin{Remark}{\em 
If $s=0$, all the Drinfeld $A$-modules of rank one are isomorphic over $\CC_\infty$ to the Carlitz module $C$.
This is no longer true for Drinfeld $A[\undt{s}]$-modules of rank one if $s\geq 1$; for example, the Drinfeld modules of rank $1$ of parameters $\alpha=1$ (Carlitz module) and $\alpha=t$ (both defined over $\TT_1=\TT$) are not
isomorphic. }

\end{Remark}

From now on, we will be focused on Drinfeld modules of rank $1$.

\begin{Definition}\label{genericmodule}{\em 
We will denote by $C_s$ the Drinfeld module of rank one over $\mathbb{T}_s$ with parameter $$\alpha=(t_1-\theta)\cdots(t_s-\theta).$$}\end{Definition}

Of course, if $s=0$, we get $C_0=C$, the Carlitz module. 
\subsection{Exponential and logarithm.}\label{explog} 
Let $\phi$ be a Drinfeld $A[\undt{s}]$-module of rank one defined over $\mathbb{T}_s$ 
with parameter $\alpha\in\mathbb{T}_s$. We also set $$\tau_{\alpha }= \alpha \tau\in \mathbb T_s[\tau].$$ Explicitly, for any $n\geq 0$, we have
$$\tau_{\alpha }^n= \alpha \tau(\alpha)\cdots\tau^{n-1}(\alpha)\tau^n.$$ 

We will be particularly interested in the formal series of $\mathbb{T}_s[[\tau]]$
\begin{eqnarray*}
\exp_{\phi} &=&\sum_{n\geq 0} \frac{1}{D_n}{\tau_{\alpha}^n},\\
\log_{\phi}&= &\sum_{n\geq 0}\frac{1}{l_n}{\tau_{\alpha}^n},
\end{eqnarray*}
respectively called the {\em exponential series} and the {\em logarithm series}
associated to $\phi$. 

It is easy to show that, in $\mathbb T_s[[\tau]],$ we have:
$$\exp_{\phi} \log_{\phi}=\log_{\phi} \exp_{\phi}=1,\quad \exp_{\phi} \theta = \phi_{\theta} \exp_{\phi}.$$
A routine computation also shows the identities in $\mathbb T_s[[\tau]]$:
$$\phi_a\exp_\phi=\exp_\phi a,\quad \log_\phi\phi_a=a\log_\phi,\quad\text{ for all }a\in A[\undt{s}].$$

We observe that
$$\|D_n^{-1}\alpha \tau(\alpha)\cdots\tau^{n-1}(\alpha)\|=\|\alpha\|^{\frac{q^n-1}{q-1}}q^{-nq^n}$$
so that for all $f\in\mathbb{T}_s$, the series
$$\exp_\phi(f):=\sum_{n\geq0}\frac{\tau_\alpha^n(f)}{D_n}=\sum_{n\geq 0}\frac{\alpha\tau(\alpha)\cdots\tau^{n-1}(\alpha)}{D_n}\tau^n(f)$$
converges in $\mathbb{T}_s$ (\footnote{The reader is warned that we are using the same symbols to denote completely different entities. Indeed, at once, $\exp_\phi$ denotes a formal series of $\mathbb{T}_s[[\tau]]$
and a continuous endomorphism of $\mathbb{T}_s$. The same remark can be made for $\log_\phi$. This should not lead to 
confusion and contributes to easily manageable notation.}).
The $\FF[\undt{s}]$-linear map $$\exp_\phi:\mathbb{T}_s\rightarrow\mathbb{T}_s$$
defined by $f\mapsto\exp_\phi(f)$
is called the {\em exponential function} of $\phi$. It is open and continuous, as 
the reader can easily check.  Also, if $\alpha\in\mathbb{T}_s(K_\infty)$, then $\exp_\phi$ induces a $\FF[\undt{s}]$-linear 
 map $\mathbb{T}_s(K_\infty)\rightarrow\mathbb{T}_s(K_\infty)$.

If $B$ is a normed ring with ultrametric norm $\|\cdot\|$, and if $r\geq 0$, we shall denote by $D_B(0,r)$
(resp. $\overline{D}_B(0,r)$) the set $\{z\in B;\|z\|<r\}$ (resp.  $\{z\in B;\|z\|\leq r\}$). We notice that, 
for all $r\geq0$, the sets $D_{\mathbb{T}_s}(0,r)$ and  $\overline{D}_{\mathbb{T}_s}(0,r)$ are
$\FF[\undt{s}]$-submodules of $\mathbb{T}_s$. 
We observe that
$$\|l_n^{-1}\alpha \tau(\alpha)\cdots\tau^{n-1}(\alpha)\|=\|\alpha\|^{\frac{q^n-1}{q-1}}q^{-q\frac{q^n-1}{q-1}}.$$ Let us set $r=-v_\infty(\alpha)$.
For all $f\in\mathbb{T}_s$ such that $v_{\infty}(f)> \frac{r-q}{q-1}$ (that is, $f\in D_{\mathbb{T}_s}(0,q^{\frac{q-r}{q-1}})$), the series
$$\log_\phi(f):=\sum_{n\geq0}\frac{\tau_\alpha^n(f)}{l_n}$$ also
converges in $\mathbb{T}_s$. The $\FF[\undt{s}]$-linear map $$\log_\phi:D_{\mathbb{T}_s}(0,q^{\frac{q-r}{q-1}})\rightarrow\mathbb{T}_s$$
defined by $f\mapsto\log_\phi(f)$
is called the {\em logarithm function} of $\phi$. As a consequence of the above discussion we have the next Lemma.
\begin{lemma}\label{Nalpha}
The functions $\exp_\phi,\log_\phi$ induce isometric automorphisms of $$D_{\mathbb{T}_s}(0,q^{\frac{q-r}{q-1}}),$$ inverse of each other.
\end{lemma}

\subsubsection{The modules $\nalpha $}\label{sectionNalpha} 
We denote by $u(\alpha)$ the maximum of the lower integer part of $\frac{r-q}{q-1}$ and zero:
 $$u(\alpha)=\max\left\{0,\lfloor\frac{r-q}{q-1}\rfloor\right\}.$$
Here we assume that $\alpha \in \mathbb T_s(K_{\infty}).$  Because this will be needed in the computations of 
\S \ref{sectioncnf}, we give  some elementary properties of the $\FF[\undt{s}]$-module
 $$\nalpha =\{f\in\TT_s(K_\infty);v_\infty(f)\geq u(\alpha)+1\}.$$ 
In particular, we note that $\nalpha =\mathfrak{m}_{\mathbb{T}_s(K_\infty)}$ if $r<2q-1$.

We observe that we have a direct sum of $\FF[\undt{s}]$-modules:
 $$\mathbb T_s(K_{\infty})= A[\undt{s}]\oplus \mathfrak{m}_{\mathbb{T}_s(K_\infty)}.$$

We notice that $u(\alpha)>0$ if and only if $r\geq2q-1$.  The proof of the next Lemma is easy and left to the reader.
  \begin{lemma} 
 \label{lemma1}
 For all $\alpha$ such that $r=-v_\infty(\alpha)\geq 1$, we have a direct sum of $\FF[\undt{s}]$-modules:
 $${\mathfrak{m}_{\mathbb{T}_s(K_\infty)}}=\nalpha \oplus\theta^{-u(\alpha)}\langle 1,\ldots,\theta^{u(\alpha)-1}\rangle_{\FF[\undt{s}]},$$
 where $\langle\cdots\rangle_{\FF[\undt{s}]}$ denotes 
 the $\FF[\undt{s}]$-span of a set of elements of $\mathbb{T}_s$.
 \end{lemma}

 We denote by $M_\alpha$ the module  
\begin{equation}\label{moduleM}
M_\alpha=\theta^{-u(\alpha)}\langle1,\ldots,\theta^{u(\alpha)-1}\rangle_{\FF[\undt{s}]},
\end{equation} (if $r<2q-1$, we set $M_\alpha=\{0\}$).
Then, for all $\alpha$ such that $r\geq 1$,
\begin{equation}\label{decompositionTs}
\mathbb{T}_s(K_\infty)=A[\undt{s}]\oplus M_\alpha\oplus \nalpha .
\end{equation}

\subsection{An example of $\exp_\phi$ injective and not surjective}\label{examplet} We shall 
consider here the case of 
$\alpha=t\in\mathbb{T}$ and describe  some properties 
of the associated exponential function $\exp_\phi$, given by
$$\exp_\phi=\sum_{i\geq 0}\frac{t^i}{D_i}\tau^i.$$
This map $\exp_{\phi}: \mathbb T\rightarrow \mathbb T$  is obviously injective. Moreover, it is not surjective. To see this, 
let us extend $\exp_\phi$ to $\CC_\infty[[t]].$ For $z=\sum_{n\geq 0} c_n t^n, c_n\in \mathbb C_{\infty},$ we set:
$$\exp_{\phi} (z)=\sum_{n\geq 0} \exp_{\phi}(c_n)t^n \in \mathbb C_{\infty}[[t]].$$
Then $\exp_{\phi}: \CC_\infty[[t]]\rightarrow \CC_\infty[[t]]$ is $\FF[[t]]$-linear. We  have:

\begin{lemma}
Let $y$ be an element of $\CC_\infty$. 
There exists a unique formal series $x=\sum_{i\geq 0}x_nt^n\in\CC_\infty[[t]]$ such that
$\exp_\phi(x)=y$. Furthermore, 
let $\epsilon\in\RR$ be such that 
$$|y|=q^{\frac{q-\epsilon}{q-1}}.$$
Then, for all $n\geq0$,
$$|x_n|=q^{\frac{q-q^n\epsilon}{q-1}}.$$
In particular, $x\in\mathbb{T}$ if and only if $|y|<q^{\frac{q}{q-1}}$. 
\end{lemma}
\begin{proof}
Let $x=\sum_{i\geq 0}x_it^i\in\CC_\infty[[t]]$ such that $\exp_\phi(x)=y.$ Then:$$x_0=y,$$
and for $n\geq 1:$
$$x_n=-(x_0^{q^n}D_n^{-1}+x_1^{q^{n-1}}D_{n-1}^{-1}+\cdots+x_{n-1}^qD_1^{-1}).$$
One can then prove that $|x_n|=q^{\frac{q-q^n\epsilon}{q-1}}$ by induction on $n$. 
\end{proof}

\subsection{Entire operators}\label{entirefunctions} Let $$f=\sum_{i_1,\ldots,i_s}f_{i_1,\ldots,i_s}t_1^{i_1}\cdots t_s^{i_s}$$ be an element of 
$\mathbb{T}_s$ (the coefficients $f_{i_1,\ldots,i_s}$ lie in $\CC_\infty$).  
We say that $f$ is an {\em entire function} if $$\lim_{i_1+\cdots +i_s \rightarrow +\infty}\, \frac{v_\infty(f_{i_1,\ldots,i_s})}{i_1+\cdots+i_s}=+\infty.$$
The subset $\mathbb{E}_s$ of entire functions of $\mathbb{T}_s$ is a subring containing the subring of polynomials $\CC_\infty[\undt{s}]$. Observe that $\tau (\mathbb E_s) \subset \mathbb E_s.$ 

Let us  consider a sequence of entire functions $(F_n)_{n\geq 0}$ and an 
operator
$$F=\sum_{n\geq 0}F_n\tau^n\in\mathbb{E}_s[[\tau]].$$
We say that $F$ is an {\em entire operator} if $\lim_{n\rightarrow\infty}v_\infty(F_n)q^{-n}=+\infty$. In particular, for all  $f\in \mathbb{T}_s,$  $F(f)=\sum_{n\geq 0}F_n\tau^n(f)$ converges in $\mathbb{T}_s.$
 
\begin{lemma}\label{lemmaentire}
Let $F=\sum_{n\geq 0}F_n\tau^n$ be an entire operator. Then: $F(\mathbb{E}_s)\subset\mathbb{E}_s$.
\end{lemma}
\begin{proof}
With $\underline{i}$ we shall denote here a multi-index $(i_1,\ldots,i_s)$
whose entries are non-negative integers. We denote by $|\underline{i}|$ the integer $i_1+\cdots+i_s$, 
and if $\underline{i},\underline{j}$ are such multi-indices, then $\underline{i}+\underline{j}$ denotes 
their component-wise sum.
We also write $\underline{t}^{\underline{i}}$ for the monomial $t_1^{i_1}\cdots t_s^{i_s}$. Hence,
we have $f=\sum_{\underline{j}}f_{\underline{j}}\underline{t}^{\underline{j}}$.
We expand each entire function $F_n$ in series 
$$F_n=\sum_{\underline{i}}F_{n,\underline{j}}\underline{t}^{\underline{j}},$$
where by hypothesis, $\lim_{|\underline{i}|\rightarrow +\infty}\frac{v_\infty(F_{n,\underline{i}})}{|\underline{i}|}=+\infty.$ 
Now, we verify easily that $F(f)=\sum_{\underline{k}}c_{\underline{k}}\underline{t}^{\underline{k}}\in\mathbb{T}_s$
where $$c_{\underline{k}}=\sum_{\underline{i}+\underline{j}=\underline{k}}\sum_{n\geq 0}F_{n,\underline{i}}f_{\underline{j}}^{q^n}.$$
Since
$$v_\infty(c_{\underline{k}})\geq\inf_{\underline{i}+\underline{j}=\underline{k},n\geq 0}(v_{\infty}(F_{n,\underline{i}})+q^nv_\infty(f_{\underline{j}})),$$
and since 
\begin{eqnarray*}
\lim_{n\rightarrow\infty}v_\infty(F_n)q^{-n}&=&\lim_{|\underline{i}|\rightarrow +\infty}v_\infty(f_{\underline{i}})|\underline{i}|^{-1}\\
&=&\lim_{|\underline{j}|\rightarrow +\infty}v_\infty(F_{n_0,\underline{j}})|\underline{j}|^{-1}\\
&=& +\infty\end{eqnarray*} for all $n_0\in \mathbb N$, we get $\lim_{|\underline{k}|\rightarrow +\infty}v_{\infty}(c_{\underline{k}})|\underline{k}|^{-1}=+\infty$ and  thus $F(f)\in\mathbb{E}_s$.
\end{proof}
Let $\alpha$ be an element of $\mathbb{E}_s$. Then, $\alpha\tau(\alpha)\cdots\tau^{i-1}(\alpha)$ is also entire for all $i$ and
$$\lim_{i\rightarrow\infty}v_\infty(\alpha\tau(\alpha)\cdots\tau^{i-1}(\alpha)D_i^{-1})q^{-i}=+\infty.$$ Therefore we deduce from Lemma \ref{lemmaentire} the following  Proposition which will be of some help later on in this paper:
\begin{proposition}\label{entireness}
Let $\phi$ be a Drinfeld $A[\undt{s}]$-module of rank one over $\mathbb{T}_s$ and let $\alpha$ be its parameter. Let us assume that $\alpha\in \mathbb E_s.$ Then:
$$\exp_{\phi}(\mathbb E_s)\subset \mathbb E_s.$$\end{proposition}

\section{$L$-series values}\label{Lseriesvalues}

In this section we consider a Drinfeld $A[\undt{s}]$-module $\phi$ of rank one over $\mathbb{T}_s$ with parameter $\alpha\in A[\undt{s}]\setminus\{ 0\}.$ We are going to associate to such a parameter $\alpha$  an {\em $L$-series value}.

\subsection{Definition of $L$-series values}

Let $\alpha$ be an element of $A[\undt{s}]\setminus\{0\}$.
In a fixed algebraic closure $\FF(\undt{s})^{ac}$ of $\FF(\undt{s})$ we can find elements $x_1,\ldots,x_r$ and $\beta\in\FF[\undt{s}]\setminus\{ 0\}$ so that, in $\FF(\undt{s})^{ac}[\theta]$,
\begin{equation}\label{alpha2}
\alpha=\beta(x_1-\theta)\cdots(x_r-\theta).
\end{equation}

We
define $$\rho_{\alpha}:A\rightarrow \FF(\undt{s})^{ac}$$  by $\rho_\alpha(0)=0$ and
$$\rho_{\alpha}(a)= \beta^{\deg_\theta(a)}a(x_1)\cdots a(x_r),\quad a\in A\setminus\{ 0\}.$$
An alternative way to write it is:
$$\rho_\alpha(a)=\operatorname{Res}_\theta(a,\alpha)\in \FF[\undt{s}],$$ where $\operatorname{Res}_\theta(P,Q)$ denotes the resultant of 
two polynomials in the indeterminate $\theta$ (\footnote{We recall that if $P=P_0\theta^d+P_1\theta^{d-1}+\cdots+P_d$ and $Q=Q_0\theta^r+Q_1\theta^{r-1}+\cdots+Q_r$
are polynomials with roots respectively $\zeta_i$ and $x_j$, then, for the resultant $\operatorname{Res}_\theta(P,Q)$, we have the identity 
$\operatorname{Res}_\theta(P,Q)=P_0^r\prod_iQ(\zeta_i)=(-1)^{dr}Q_0^d\prod_jP(x_j)$.}).
In particular, with $P$ a prime of $A$ (we recall that a {\em prime} of $A$ is monic irreducible element in $A_+$), $\rho_{\alpha}(P)=0$ if and only if $P$ divides $\alpha$ in $A[\undt{s}]$.

If $a,b\in A$, then $\rho_{\alpha}(ab)=\rho_\alpha(a)\rho_\alpha(b)$ and if $\alpha_1,\alpha_2$ are polynomials of $A[\undt{s}]$,
then
$$\rho_{\alpha_1\alpha_2}(a)=\rho_{\alpha_1}(a)\rho_{\alpha_2}(a),\quad a\in A.$$

\begin{Definition}{\em 
Let $\phi$ be the Drinfeld $A[\undt{s}]$-module of rank one of parameter $\alpha\in A[\undt{s}]$, let $n\geq 1$ be an integer. 
The {\em $L$-series value at $n$ associated to $\phi$} is the unit of norm one of $\mathbb{T}_s(K_\infty)$
defined by (\ref{ourlseries}).}
\end{Definition}

By \cite[Lemma 4]{ANG&PEL}, we also have that $L(-n,\phi):=\sum_{d\geq 0}\sum_{a\in A_{+,d}}\rho_\alpha(a)a^n$ converges in $\mathbb T_s$ for $n\geq 0$  and is in fact  in $A[\undt{s}].$
Note that the above definition is \emph{not} invariant under isomorphism of Drinfeld $A[\undt{s}]$-modules.

\begin{Remark}{\em We can also associate $L$-series values $L(n,\phi) \in \FF(\undt{s})((\frac 1\theta))$ to 
Drinfeld $\FF(\undt{s})[\theta]$-modules of rank one defined over $\FF(\undt{s})[\theta]$. 
 In the sequel, we will also work with such modules and $L$-series values, but the 
most interesting examples discussed here will arise from the case of $A[\undt{s}]$-modules defined over $A[\undt{s}]$.
}\end{Remark}
The value $L(1, \phi)$ will be one of the main objects of interest of the present paper.
\subsection{Examples.}
\subsubsection{Examples with $s=0$.} If $s=0$ and $\alpha=1$, we have $\phi=C$ and
$$L(n,\phi)=L(n,C)=\zeta_C(n),$$
where $\zeta_C(n)$ is, for $n>0$, the {\em Carlitz zeta value} 
$$\zeta_C(n)=\sum_{a\in A^+}a^{-n}\in 1+\theta^{-1}\FF[[\theta^{-1}]].$$
If $\alpha\in A\setminus\{0\}$, then we can write $\alpha=\beta\mathfrak{p}_1^{\nu_1}\cdots\mathfrak{p}_m^{\nu_m}$
with $\beta\in\FF^\times$,
for primes $\mathfrak{p}_1,\ldots,\mathfrak{p}_m$ of respective degrees $d_1,\ldots,d_m$
so that $\sum_id_i\nu_i=r=\deg_\theta(\alpha)$, we have, for $a\in A^+$,
$$\rho_\alpha(a)=\beta^{\deg_\theta(a)}\prod_{i=1}^m\prod_{j=1}^{d_i}a(\zeta_{i,j})^{\nu_i},$$
where $\zeta_{i,1},\ldots,\zeta_{i,d_i}$ are the zeros of $\mathfrak{p}_i$ in $\FF^{ac}$ for all $i$.
This implies, in the case $\beta=1$ (that is, $\alpha\in A^+$), that the series $L(n,\phi)$ is the special value of a 
Dirichlet $L$-series:
$$L(n,\phi)=\sum_{a\in A_+}a^{-n}\prod_{i=1}^m\prod_{j=1}^{d_i}a(\zeta_{i,j})^{\nu_i}\in K_\infty.$$

\subsubsection{Case of $\alpha=t$} It is understood here that $s=1$ so that we are in $\mathbb{T}=\mathbb{T}_1$. 
This case directly refers to the example of Drinfeld module $\phi$ treated in \ref{examplet}. We have then
$$L(n,\phi)=\sum_{d\geq 0}t^d\sum_{a\in A_{+,d}}a^{-n}\in\mathbb{T}\cap K[[t]]$$
if $n>0$. It is easy to see that
$$L(1,\phi)=\sum_{i\geq0}t^i\ell_i^{-1}=\log_\phi(1)\in\TT.$$
We have that the series $$L(-j,\phi):=\sum_{d\geq 0}t^d\sum_{a\in A_{+,d}}a^{j}$$
defines an element of $A[t]$ for $j\geq 0,$
and 
$$L(-j,\phi)=z(t^{-1},-j)$$
where the function $z$ is defined as in Goss' book \cite[Remark 8.12.1]{GOS}. In loc. cit. Goss computes recursively the polynomial
$z(t^{-1},-j)\in A[t]$ for all $j\geq 0$.

\subsubsection{Case in which $\alpha=(t_1-\theta)\cdots(t_s-\theta)$}
\begin{Definition}{\em We will denote by $C_s$ the Drinfeld module of rank one over $\TT_s$ with parameter
$$\alpha=(t_1-\theta)\cdots(t_s-\theta).$$ We notice that $C_0=0$, the Carlitz module.}\end{Definition}
We have:
$$L(n,C_s)=L(\chi_{t_1}\cdots\chi_{t_s},n)=\sum_{a\in A_+}\frac{\chi_{t_1}(a)\cdots\chi_{t_s}(a)}{a^n}\in\mathbb{T}_s(K_\infty)^\times$$ with $L(\chi_{t_1}\cdots\chi_{t_s},n)$ the functions studied in 
\cite{ANG&PEL} and where for all $a\in A, \chi_{t_i}(a)=a(t_i).$ The case $s=1$ and $\alpha=t-\theta$ yields the 
functions $L(\chi_t,n)$ of \cite{PEL2}.

\subsubsection{A further example}\label{analyticity}
We shall also trace a connection with the Goss zeta functions, especially the functions considered by Goss in 
\cite{GOS2}, see also \cite[Section 2.1]{ANG&PEL}. We recall, from Section 2.1 of loc. cit. the definition of the $L$-series $L(\chi_{t_1}\cdots\chi_{t_s},x,y)$, with $(x,y)$ in 
the topological group $\mathbb{C}_\infty^\times\times\mathbb{Z}_p$ 
denoted by $\mathbb{S}_\infty$ therein: 
$$L(\chi_{t_1}\cdots\chi_{t_s},x,y)=\sum_{k\geq 0}x^{-k}\sum_{a\in A_{k,+}}\chi_{t_1}(a)\cdots\chi_{t_s}(a)\langle a \rangle^{-y},$$
where $\langle a\rangle$ is the $1$-unit $a/\theta^{\deg_\theta(a)}$, and its $p$-adic exponentiation by $-y$ is well defined.
For fixed $(x,y)\in\mathbb{S}_\infty$, the above series is a well defined unit element of $\mathbb{T}_s$.
Thanks to \cite[Proposition 6]{ANG&PEL}, we know that the above series in $\mathbb{T}_s$
also defines an entire function $\CC_\infty^s\rightarrow\CC_\infty$.

We have, for $\beta\in\FF^\times$:
\begin{eqnarray*}
L(\chi_{t_1}\cdots\chi_{t_s},\beta^{-1}\theta^n,n)&=&\sum_{k\geq 0}\beta^k\theta^{-kn}\sum_{a\in A_{k,+}}\chi_{t_1}(a)\cdots\chi_{t_s}(a)\theta^{kn}a^{-n}\\
&=&\sum_{k\geq 0}\beta^k\sum_{a\in A_{k,+}}\chi_{t_1}(a)\cdots\chi_{t_s}(a)a^{-n}.
\end{eqnarray*}
This equals $L(n,C_s)$ if $\beta=1$.

\section{The class number formula}\label{sectioncnf}

In this section, the integer $s\geq 0$ is fixed. Hence, we more simply write $\undt{s}=\undt{}=\{t_1,\ldots,t_s\}$. 

We introduce the {\em class module} and the {\em unit module} associated to a given Drinfeld $A[\undt{}]$-module of rank one 
of parameter $\alpha\in A[\undt{}]\setminus\{0\} $. We then give a {\em class number formula} which relates $L(1, \phi)$ to these objects.

\subsection{Class and unit modules}\label{classunitmodules} Let $\phi$ be a Drinfeld $A[\undt{}]$-module of rank one with parameter $\alpha\in A[\undt{}]\setminus \{ 0\}$. Recall that $r=-v_{\infty}(\alpha)\in \mathbb N.$  The definitions below are inspired by 
Taelman's work \cite{TAE2,TAE3}.

\subsubsection{The class module.} We define the class module $H_{\phi}$ as the quotient of $A[\undt{}]$-modules:
 $$H_{\phi}:= \frac{{\phi}(\mathbb T_s(K_{\infty}))}{\exp_{\phi}(\mathbb T_s(K_{\infty}))+{\phi}(A[\undt{}])}$$
 where we recall that $\phi(A[\undt{}])$ is the $k[\undt{}]$-module $A[\undt{}]$ equipped with the $A[\undt{}]$-module structure induced by $\phi$.

\subsubsection{The unit module} It is the $A[\undt{}]$-submodule of $\mathbb T_s(K_{\infty})$  defined by:
$$U_{\phi}:=\{ f\in  \mathbb T_s(K_{\infty}); \exp_{\phi}(f)\in A[\undt{}]\}.$$

\begin{lemma}\label{lemmanu} For all $r\geq 1$, the exponential function $\exp_\phi:\mathbb{T}_s(K_\infty)\rightarrow\mathbb{T}_s(K_\infty)$ induces an injective 
homomorphism of $\FF[\undt{}]$-modules
$$\frac{\mathbb{T}_s(K_\infty)}{U_\phi\oplus \nalpha }\rightarrow\frac{\mathbb{T}_s(K_\infty)}{A[\undt{}]\oplus \nalpha }$$
whose cokernel is $H_\phi$.
\end{lemma}
\begin{proof} This is plain by the fact that $\exp_\phi$ restricted to $D_{\mathbb{T}_s}(0,q^{\frac{q-r}{q-1}})$ is an isometric isomorphism (Lemma \ref{Nalpha}) and the fact that $\exp_\phi^{-1}(A[\undt{}])\cap \mathbb T_s(K_{\infty})=U_\phi$.
\end{proof}

If $M$ is a finitely generated $\FF[\undt{}]$-module, we will use the term \emph{rank of $M$} for its generic rank, that is, $\dim_{\FF(\undt{})} M\otimes_{\FF[\undt{}]}\FF(\undt{})$.

\begin{corollary}\label{corollaryHszero}
For all Drinfeld modules $\phi$ as above, $H_\phi$ is a finitely generated $\FF[\undt{}]$-module of rank $\leq u(\alpha)$.
\end{corollary}

\begin{Remark}\label{exactseq1}{\em We have constructed a 
short exact sequence of $A[\undt{}]$-modules
\begin{equation}\label{shortexactsequence}
0\rightarrow\frac{\mathbb{T}_s(K_\infty)}{U_\phi\oplus \nalpha }\rightarrow\frac{\phi (\mathbb{T}_s(K_\infty))}{\phi (A[\undt{}])\oplus \nalpha }\rightarrow H_\phi\rightarrow 0.
\end{equation}  On the other hand, there is an isomorphism of $\FF[\undt{}]$-modules between $M_\alpha$ 
(the module defined in (\ref{moduleM}))
and $$\frac{\mathbb{T}_s(K_\infty)}{A[\undt{}]\oplus \nalpha }.$$ Therefore, the  $\FF[\undt{}]$-modules $\frac{\mathbb{T}_s(K_\infty)}{U_\phi \oplus \nalpha }$ and $H_\phi$ are finitely generated, and their ranks add up to
$u(\alpha)$, which is the rank of $M_\alpha$. This tells us in particular that $U_\phi$ is non-zero. If $r<2q-1$, we have that $H_\phi=\{0\}$.
}\end{Remark}

\subsection{Modules over $\ring{}$}\label{aroundcalr} We
observe that $\mathbb  T_s(K_{\infty}) \subset K(\undt{})_{\infty}$  (\footnote{We recall that $K(\undt{})_{\infty}=\FF(\undt{})((\frac{1}{\theta}))$.}) and that $\tau$ extends to a continuous homomorphism of $k(\undt{})$-algebras again denoted by $\tau: K(\undt{})_{\infty}\rightarrow K(\undt{})_{\infty}.$  If $M\subset \mathbb T_s(K_{\infty})$ is a $\FF[\undt{}]$-module, we denote by $\FF(\undt{})M\subset K(\undt{})_{\infty}$ the $\FF(\undt{})$-module  generated by $M.$ We also set, for $\phi$ as above,
$$\calu_\phi=\FF(\undt{})U_\phi,\quad \calh_\phi=\FF(\undt{})\otimes_{\FF[\undt{}]}H_\phi.$$
The vector spaces $\calu_\phi,\calh_\phi$ are $\FF(\undt{})[\theta]$-modules. From here to the end of this section, we are going to make extensive study of modules over $\FF(\undt{})[\theta]$ so that we denote this ring by $\ring{}$.
\begin{proposition}
 \label{proposition2} The following properties hold.
 \begin{enumerate}
\item The $\ring{}$-module $\calu_\phi$ is free of rank one.
\item We have:
 $$\calu_\phi=\{ f\in K(\undt{})_{\infty},\, \exp_{\phi}(f) \in \ring{}\}.$$
\end{enumerate}
 \end{proposition}
 \begin{proof} (1). Since  $U_\phi$ is non-trivial,
there exists an element $f\in \calu_\phi$ with $\|f\|>0$ minimal. 
Indeed, by the fact that $\exp_\phi$ induces an isometric isomorphism of $D_{\mathbb{T}_s}(0,q^{\frac{q-r}{q-1}})$, $\calu_\phi$ is discrete, that is,
$\calu_\phi\cap\mathfrak{m}_{K(\undt{})_{\infty}}^n=\{0\}$ for $n$ big enough.

Let $g$ be another element of $\calu_\phi$.
Then, since $ K(\undt{})_{\infty}=\ring{}\oplus \mathfrak{m}_{K(\undt{})_{\infty}}$, there exists a polynomial $h$ of $\ring{}$ such that
$g=hf+b$ where $b\in K(\undt{})_{\infty}$ is such that $\|b\|<\|f\|$. Since $\calu_\phi$ is an $\ring{}$-module, we get $b\in\calu_\phi$ so that $b=0$.
This means that $\calu_\phi$ is free of rank one.  
\medskip

\noindent (2). 
Observe that $\FF(\undt{})\mathbb T_s(K_{\infty})$ is dense in $K(\undt{})_{\infty},$ thus:
 $$K(\undt{})_{\infty}=\FF(\undt{})\mathbb T_s(K_{\infty})+ \mathfrak{m}_{K(\undt{})_{\infty}}^{u(\alpha)+1}.$$

It is clear that $\calu_\phi\subset \{ f\in K(\undt{})_{\infty},\, \exp_{\phi}(f) \in \ring{}\}.$ Now, let $f\in K(\undt{})_{\infty}$ be such that $\exp_{\phi} (f) \in\ring{}.$ 
 We can write $f$ as a sum $g+h,$ where $g\in \FF(\undt{})\mathbb T_s(K_{\infty})$ and $h\in  \mathfrak{m}_{K(\undt{})_{\infty}}^{u(\alpha)+1}.$ We get:
 $$\exp_{\phi} (h) = \exp_{\phi}(f)-\exp_{\phi}(g) \in \FF(\undt{})\mathbb T_s(K_{\infty}).$$
 This implies that:
 $$\exp_{\phi} (h) \in \FF(\undt{})\mathfrak{m}_{\mathbb{T}_s(K_\infty)}^{u(\alpha)+1}=  \mathfrak{m}_{K(\undt{})_{\infty}}^{u(\alpha)+1}\cap \FF(\undt{})\mathbb T_s(K_{\infty}).$$ 
 Therefore $h\in \FF(\undt{})\mathfrak{m}_{\mathbb{T}_s(K_\infty)}^{u(\alpha)+1}$ and thus $f\in \FF(\undt{})\mathbb T_s(K_{\infty}).$ We conclude that $f\in \calu_\phi.$
\end{proof}

\begin{corollary}\label{remarkfloric}
The $A[\undt{}]$-module $U_\phi$ is free of rank one.
\end{corollary}
\begin{proof} By Proposition \ref{proposition2}, we have  
$\calu_\phi=f\ring{}$, with $f\in\mathbb{T}_s(K_\infty)$. Without 
loss of generality, we can also suppose that if
$h$ divides $f$ in $\mathbb{T}_s(K_\infty)$, with $h\in\FF[\undt{}]$, then $h\in\FF^\times$. Clearly, $U_\phi\supset fA[\undt{}]$. Let us consider now
an element $g$ of $U_\phi$. We have that $g\in f\ring{},$ and  we can write $g=af/\delta$ where $a\in A[\undt{}]$
and $\delta\in\FF[\undt{}]\setminus\{ 0\}$. This means that $\delta$ divides
$af$ in $\mathbb{T}_s(K_\infty)$ which is a unique factorization domain.  So $\delta$ must divide $a$ in $A[\undt{}]$ and we get $g\in f A[\undt{}].$
\end{proof}

 \subsection{Local factors of the $L$-series values}\label{localfactors}
 
Let $R$ be a unitary  commutative ring. Let $M$ be a finitely generated $R$-module. 
As {\em Fitting ideal} of $M$ we mean the {\em initial Fitting ideal} as defined in \cite[Chapter XIX]{LAN}.
By Chapter XIX, Corollary 2.9 of loc. cit., if $M$ is a finite direct sum of cyclic modules, $$M=\bigoplus_{i=1}^n\frac{R}{\mathfrak{a}_i},\quad \mathfrak{a}_i\text{ ideal of }R,$$
then $$\operatorname{Fitt}_R(M)=\mathfrak{a}_1\cdots\mathfrak{a}_n.$$

Let $\theta$ be an indeterminate over a field $F$. We write $R=F[\theta]$, and we consider an $R$-module $M$ which also is an $F$-vector space of finite dimension.
Let $e_\theta$ be the endomorphism of $M$ induced by the multiplication by $\theta$.
Then, we write 
$$[M]_R=\det{}_{R}(Z-e_\theta|M)|_{Z=\theta}\in R$$ for the characteristic polynomial of $e_\theta$, where the indeterminate 
$Z$ is replaced with $\theta$. This is a monic polynomial in $\theta$ of $R=F[\theta]$ and it is the monic generator of $\operatorname{Fitt}_{R}(M)$.

Let $\alpha$ be an element of $\ring{}\setminus\{ 0\}$ (we recall that $\ring{}=\FF(\undt{})[\theta]$) and let us consider the Drinfeld $\ring{}$-module of 
rank one and parameter $\alpha$, that is, 
the injective homomorphism of $\FF(\undt{})$-algebras
$$\phi: \ring{}\rightarrow \operatorname{End}_{\FF(\undt{})-\text{lin.}}(K(\undt{})_{\infty})$$ given by $\phi_{\theta}= \theta +\alpha\tau$. For all $a\in A$, the resultant $\rho_\alpha(a)=\operatorname{Res}_\theta(a,\alpha)$ is a well defined element of $\FF(\undt{})$ making the series (and the corresponding eulerian product)
$$L(n,\phi)=\sum_{a\in A_+}\rho_\alpha(a)a^{-n}=\prod_P\left(1-\frac{\rho_\alpha(P)}{P^n}\right)^{-1},\quad n>0$$
 convergent in $K(\undt{})_{\infty}.$

 \begin{lemma}
 \label{lemmaA3} Let $P$ be a prime of $A$  of degree $d.$ Then, the following congruence holds in $\ring{}[\tau]:$
 $$\phi_P\equiv \rho_\alpha(P) \tau^d\pmod{P\ring{}[\tau]}.$$
 \end{lemma}
 \begin{proof}We recall that $\tau_\alpha=\alpha\tau$.
 Let $a\in A_+$ be of degree $d$. We expand in $\ring{}[\tau_\alpha]$:
$$\phi_a= \sum_{i=0}^d(a)_i\tau_\alpha^i,$$ where it is easy to see that
$(a)_0=a$.
From the relation $\phi_a\phi_{\theta}=\phi_{\theta}\phi_a$ we get, by induction on $i=1,\ldots,d-1$,

$$(a)_i= \frac{\tau((a)_{i-1})-(a)_{i-1}}{\theta^{q^i}-\theta}$$
and $(a)_i\in A$ for $i=0, \ldots, d$ with $(a)_d=1$ (so these are the 
coefficients of the classical Carlitz multiplication by $a$). 
Since a prime $P$ of degree $d$ does not divide $\theta^{q^i}-\theta$ if $i<d$,  we get for $i=0, \ldots,d-1$,
$(P)_i\equiv0\pmod{P\ring{}}.$ This implies that 
$$\phi_P\equiv\tau_\alpha^d\equiv\alpha\tau(\alpha)\cdots\tau^{d-1}(\alpha)\tau^d\pmod{P\ring{}[\tau]}.$$
Now, we observe that, if $\zeta_1,\ldots,\zeta_d$ are the roots of $P$ in $\FF^{ac}$ and if $\zeta$ is one of these roots,
\begin{eqnarray*}
\rho_\alpha(P)&=&\operatorname{Res}_\theta(P,\alpha)\\
&=&\prod_{j=1}^d\alpha|_{\theta=\zeta_j}\\
&=&\alpha\tau(\alpha)\cdots\tau^{d-1}(\alpha)|_{\theta=\zeta}\\
&\equiv&\alpha\tau(\alpha)\cdots\tau^{d-1}(\alpha)\pmod{P\ring{}}.
\end{eqnarray*}
\end{proof}
If $L$ is a ring with an endomorphism $\sigma$, we denote by $L^{\sigma=1}$ the subring $\{ x\in L, \sigma (x)=x\}$.
\begin{lemma}
\label{lemmaA5} Let $L$ be a  field and let $\sigma \in {\rm Aut} (L).$ We set $F=L^{\sigma=1}.$ Let $r\geq 0$ be an integer strictly less than the order of $\sigma$, let $a_0,\ldots, a_{r-1}$ be $r$ elements in $L$. Then  
$$V=\left\{ x\in L; \sigma^r(x)+ \sum_{i=0}^{r-1} a_i \sigma^i (x) =0\right\}$$  is an $F$-vector space of dimension not exceeding $r$.
\end{lemma}
\begin{proof} A sketch of proof will be enough as this is essentially well known, see \cite[\S 1.2]{PUT2}.
Let $n\geq 1$ be an integer and let  $\mathcal{A}$ be a matrix with entries in $L$. Let $v_1, \ldots, v_r$ be vectors of $L^n$ such that $\sigma (v_i) =\mathcal{A}v_i, i=1, \ldots, r.$ Then by the proof of \cite[Lemma 1.7]{PUT2}, if the vectors $v_1, \ldots, v_r$ are linearly dependent over $L$, they are also linearly dependent over $F.$ This implies that the $F$-vector space $W=\{ v\in L^n, \sigma (v) = \mathcal{A}v\}$ satisfies:
$$\dim_{F}(W)\leq n.$$
\noindent Let $\mathcal{A}$ be  the  companion matrix  of the equation $$\sigma^r(x)+ \sum_{i=0}^{r-1} a_i \sigma^i (x) =0$$ (see \cite[p. 8]{PUT2}).  Let $W=\{ v\in L^r, \sigma (v) =\mathcal{A}v\}.$ Then the map $V\mapsto W,$ $$x\mapsto{}^t (x,\sigma(x), \cdots, \sigma^{r-1}(x))$$ (the sign $\cdot{}^t$ means transposition) is an isomorphism of $F$-vector spaces. 
\end{proof}

\begin{lemma}
 \label{propositionA1}
 Let  $\phi$ be  a Drinfeld $\ring{}$-module of rank one over $K(\undt{})_{\infty}$ with parameter 
$\alpha\in \ring{}\setminus\{ 0\}$. Let  $P$ be a prime of $A$ of degree $d$. Then, we have an isomorphism of $\ring{}$-modules:
 $$ \phi\left(\frac{\ring{}}{P\ring{}}\right)\cong \frac{\ring{}}{(P-\rho_{\alpha} ( P ))\ring{}}.$$
 \end{lemma}
 \begin{proof} By Lemma \ref{lemmaA3}, we have:
$$(P-\rho_{\alpha} ( P )).\phi\left(\frac{\ring{}}{ P\ring{}}\right)=\{0\}.$$ 
We set $L=\ring{}/P\ring{}.$  Then $\tau$ induces an automorphism of $L$ and $L^{\tau=1}=\FF(\undt{})$, so that $\tau \in {\rm Gal}(L/\FF(\undt{}))$ is  of order  $d$. Also, $\phi$ induces a morphism of $\FF(\undt{})$-algebras $\phi: \ring{} \rightarrow L[\tau].$ For $a$ in $\ring{}$ we set $\operatorname{Ker}(\phi_a)=\{ x\in L;\phi_a(x)=0\}$. We write $b=P-\rho_{\alpha}(P)$. We notice that $d=\dim_{\FF(\undt{})}(L)= \deg_{\theta}(b)$ and $L=\operatorname{Ker}(\phi_b)$. We have, by Lemma \ref{lemmaA5}, for all $a\in \ring{}$, $\dim_{\FF(\undt{})}(\operatorname{Ker}(\phi_a))\leq \deg_\theta(a)$. This implies (see for example \cite[proof of Theorem 6.3.2]{GOS}) that we have an isomorphism of $\ring{}$-modules
$\operatorname{Ker}(\phi_b)\cong \ring{}/b\ring{}$. 
 \end{proof}
 
Let $P$ be a prime of $A$. By Lemma \ref{propositionA1}, we have $$\left[\phi\left(\frac{\ring{}}{P\ring{}}\right)\right]_{\ring{}}=P-\rho_\alpha(P).$$
The $L$-series attached to $\phi/\ring{},$ denoted by $\mathcal{L}(\phi/\ring{}),$ is the  infinite product running over the primes $P$ of $A$:
\begin{equation}\label{defprodL}
\prod_{P}\, \left[\frac{\ring{}}{P\ring{}}\right]_{\ring{}}\left[\phi\left(\frac{\ring{}}{P\ring{}}\right)\right]_{\ring{}}^{-1}.\end{equation}
 
 \begin{proposition}\label{prodexpansion} Let $\phi$ be a Drinfeld $A[\undt{}]$-module of rank one of parameter $\alpha\in A[\undt{}]\setminus\{ 0\}$.
 The product $\mathcal{L}(\phi/\ring{})$ in (\ref{defprodL}) converges in $K(\undt{})_{\infty}$ and we have
 $$\mathcal{L}(\phi/\ring{})=L(1,\phi).$$
 \end{proposition}
 \begin{proof}
 We have, by Lemma \ref{propositionA1}, that the quotient 
 $$\frac{\left[\frac{\ring{}}{P\ring{}}\right]_{\ring{}}}{\left[\phi\left(\frac{\ring{}}{P\ring{}}\right)\right]_{\ring{}}}$$ is equal to:
 $$\frac{P}{P-\rho_\alpha(P)}=\left(1-\frac{\rho_\alpha(P)}{P}\right)^{-1}.$$
 The factors of the infinite product defining $\mathcal{L}(\phi/\ring{})$ agree with the eulerian factors of $L(1,\phi)$.
 Since the product $L(1,\phi)$ converges in $K(\undt{})_{\infty}$, this implies that the product $\mathcal{L}(\phi/\ring{})$ converges to $L(1,\phi)$ in $K(\undt{})_{\infty}$.
 \end{proof}

\subsection{The class number formula}\label{classnumberformula}
 An element $$f=\sum_{i\geq i_0}f_i\theta^{-i}\in K(\undt{})_{\infty}\setminus\{ 0\},\quad f_i\in \FF(\undt{}), f_{i_0}\not =0,$$
 is {\em monic} if the leading coefficient $f_{i_0}$ is equal to one.
 We shall write $$[\ring{}:\calu_\phi]_{\ring{}}$$ ($\ring{}=\FF(\undt{})[\theta]$) for the unique monic element $f$ in $K(\undt{})_{\infty}$ such that $\calu_\phi =f\ring{}$,
the existence of which is guaranteed by the first part of Proposition \ref{proposition2}.

\subsubsection{Estimating the dimension of $\calv_\phi$}\label{calv}
Let us consider the following $\ring{}$-module:
$$\calv_{\phi}=\frac{\phi(K(\undt{})_{\infty})}{\phi(\ring{})+\exp_{\phi}(K(\undt{})_{\infty})}.$$

Just as in the proof of the second part of Proposition \ref{proposition2}, we see in fact that for all $n\geq 1$, $$K(\undt{})_{\infty}=\FF(\undt{})\mathbb{T}_s(K_\infty)+\mathfrak{m}_{K(\undt{})_{\infty}}^n.$$
For all $n$ big enough, $\exp_\phi$ induces an isometric automorphism of $\mathfrak{m}_{K(\undt{})_{\infty}}^n$ (for instance, it suffices 
to take $n\geq u(\alpha)+1$). Therefore, for such a choice of $n$, we have the isomorphism of $\FF(\undt{})$-vector spaces
$$
\calv_{\phi}=
\frac{\FF(\undt{})\mathbb{T}_s(K_\infty)+\mathfrak{m}_{K(\undt{})_{\infty}}^n}{\ring{}+\exp_{\phi}(\FF(\undt{})\mathbb{T}_s(K_\infty))+\mathfrak{m}_{K(\undt{})_{\infty}}^n}\cong
\frac{\FF(\undt{})\mathbb{T}_s(K_\infty)}{\ring{}+\exp_{\phi}(\FF(\undt{})\mathbb{T}_s(K_\infty))}.$$
This  implies that we have an isomorphism of $\ring{}$-modules:
\begin{equation}\label{secondisom}
\calv_{\phi}\simeq H_\phi\otimes_{\FF[\undt{}]}\FF(\undt{})=\calh_\phi.
\end{equation}
 Note that $\dim_{\FF(\undt{})}(\calh_\phi)\leq u(\alpha)$ by Corollary \ref{corollaryHszero}.
This yields:
\begin{corollary}\label{corollaryHszerobis} The $\ring{}$-module $\calv_{\phi}$ is a finite dimensional $\FF(\undt{})$-vector space
of dimension at most $u(\alpha)$.\end{corollary}

\subsubsection{The formula}
The next Theorem directly follows from Theorem \ref{theoremA1}, proved in the appendix by Florent Demeslay,
by means of Proposition  \ref{prodexpansion} and the isomorphism (\ref{secondisom}).

\begin{theorem}[The class number formula]
\label{theorem2}
\noindent Let $\phi$ be a Drinfeld $A[\undt{}]$-module of rank one of parameter $\alpha\in A[\undt{}]\setminus\{ 0\}$.
The following equality holds in $K(\undt{})_{\infty}$:
$$L(1, \phi) =[\calh_{\phi}]_{\ring{}} [\ring{}: \calu_\phi]_{\ring{}}.$$
\end{theorem}

Let $\phi$ be a Drinfeld $A[\undt{}]$-module of rank one over $\mathbb{T}_s$ with parameter $\alpha\in A[\undt{}]\setminus\{ 0\}$.
The following Corollary to the class number formula will be crucial:
\begin{corollary} 
\label{lemma2}
Let $\phi$ be a Drinfeld $A[\undt{}]$-module of rank one of parameter $\alpha\in A[\undt{}]\setminus\{ 0\}$.
We have:
$$\exp_{\phi}(L(1, \phi))\in A[\undt{}].$$
\end{corollary} 
 \begin{proof} By definition, $[\ring{}:\calu_\phi]_{\ring{}}\ring{}= \calu_\phi$ and obviously,
 $[\calh_\phi]_{\ring{}}\in \ring{}$. Thus:
 $$L(1,\phi)\in \calu_\phi.$$ By Part (2) of Proposition \ref{proposition2}, 
 $\exp_\phi(L(1,\phi))\in \ring{}$. At once, by construction, $L(1,\phi)\in\mathbb{T}_s(K_\infty)$
 so that $\exp_\phi(L(1,\phi))\in \mathbb{T}_s(K_\infty)$. But then,
 $$\exp_\phi(L(1,\phi))\in \ring{}\cap\mathbb{T}_s(K_{\infty})=A[\undt{}].$$
 \end{proof}

\begin{Remark}\label{caserinfq-1}{\em If $\alpha$ is as in (\ref{alpha2}) and $0\leq r\leq q-1$ ($r=-v_{\infty}(\alpha)$) we have that $L(1,\phi)-1\in\mathfrak{m}_{\mathbb{T}_s(K_\infty)}$ (see \S \ref{sectionNalpha}).
 Since $\exp_\phi$ is an isometric automorphism of $D_{\mathbb{T}_s}(0,q^{\frac{q-r}{q-1}})$ we also get 
 $\exp_\phi(L(1,\phi))-1\in\mathfrak{m}_{\mathbb{T}_s(K_\infty)}$ but $\ring{}\cap \mathfrak{m}_{\mathbb{T}_s(K_\infty)}=\{0\}$.
 We have obtained the identity \begin{equation}\label{identityrinfquminusone}\exp_\phi(L(1,\phi))=1.\end{equation}
 This can be rewritten as
 \begin{equation}\label{identityrinfquminusone2}L(1,\phi)=\log_\phi(1)\end{equation}
 because $1=\|1\|<q^{\frac{q-r}{q-1}}$ thanks again to the hypothesis on $r$. Similar formulas have been observed
 by Perkins in \cite{PER1}.}\end{Remark}

\subsubsection{The circular unit module} This is the sub-$A[\undt{}]$-module
 $$U^c_\phi= L(1, \phi)A[\undt{}]\subset U_\phi.$$ 
 \begin{proposition}\label{lemma3} 
 Let $\phi$ be a Drinfeld $A[\undt{}]$-module of rank one of parameter $\alpha\in A[\undt{}]\setminus\{ 0\}$.
The modules 
 $\frac{ U_{\phi}}{U^c_\phi}$ and $H_{\phi}$ are  finitely generated $\FF[\undt{}]$-modules  of equal rank at most $u(\alpha).$
 \end{proposition}
 \begin{proof} By Lemma \ref{lemma1}, we have a direct sum of $\FF[\undt{}]$-modules:
  $$\mathbb T_s(K_{\infty}) = \theta^{-u(\alpha)}U^c_\phi \oplus \nalpha .$$
  Then, we have an exact sequence  of $\FF[\undt{}]$-modules:
  \begin{equation}\label{anotherusefulexactsequence}0\rightarrow \frac{U_{\phi}}{U^c_\phi}\rightarrow \frac{\theta^{-u(\alpha)}U^c_\phi \oplus \nalpha }{U^c_\phi \oplus \nalpha }\rightarrow \frac{\mathbb T_s(K_{\infty})}{ U_{\phi}\oplus \nalpha } \rightarrow 0.
\end{equation} 
 Observe that the $\FF[\undt{}]$-module in the middle is free of rank $u(\alpha).$ Thus $\frac{U_{\phi}}{U^c_\phi}$ is a finitely generated torsion-free $\FF[\undt{}]$-module. of rank not bigger than $u(\alpha).$
 
Recall from Remark \ref{exactseq1} that $\exp_{\phi}$ induces an exact sequence of finitely generated $\FF[\undt{}]$-modules:
  $$0\rightarrow \frac{\mathbb T_s(K_{\infty})}{U_{\phi}\oplus \nalpha }\rightarrow \frac{\mathbb T_s(K_{\infty})}{A[\undt{}]\oplus \nalpha }\rightarrow H_{\phi} \rightarrow 0.$$ Since there is an isomorphism of $\FF[\undt{}]$-modules
  $$\frac{\theta^{-u(\alpha)}U^c_\phi \oplus \nalpha }{U^c_\phi \oplus \nalpha }\cong \frac{\mathbb T_s(K_{\infty})}{A[\undt{}]\oplus \nalpha }$$
the modules $\frac{U_{\phi}}{U^c_\phi}$ and $H_{\phi}$ have the same rank over $\FF[\undt{}]$. 
\end{proof}
\begin{Remark}{\em In particular, we obtain $U_\phi=U_\phi^c$ if $r<2q-1$.}\end{Remark}

We deduce, from Theorem \ref{theorem2} (with $\ring{}=\FF(\undt{})[\theta]$, $\calh_\phi=\FF(\undt{})\otimes_{\FF[\undt{}]}H_\phi$ and $\calu_\phi^c=\FF(\undt{})U_\phi^c$), the following Corollary.

\begin{corollary}
 \label{theorem41} Let $\phi$ be a Drinfeld $A[\undt{}]$-module of rank one of parameter $\alpha\in A[\undt{}]\setminus\{ 0\}$.
We have
 $$[\calh_{\phi}]_{\ring{}}= [\calu_\phi:\calu^c_\phi]_{\ring{}}.$$
  \end{corollary}
 \begin{proof} We have:
  $$[\ring{}: \calu_\phi]_{\ring{}}=[\ring{}: \ring{}L(1, \phi)]_{\ring{}} \left[\frac{\calu_\phi}{\ring{}L(1, \phi)}\right]_{\ring{}}^{-1}.$$
 Then, by Theorem \ref{theorem2}, we obtain:
 $$\left[\frac{\calu_{\phi}}{\ring{}L(1, \phi)}\right]_{\ring{}}=[\calh_{\phi}]_{\ring{}}.$$
 \end{proof}

\begin{Remark}\label{floric}
{\em Observe that, by Corollary  \ref{remarkfloric} and since $L(1, \phi)\in \mathbb T_s(K_{\infty})^{\times},$ the $\FF[\undt{}]$-module $\frac{ U_{\phi}}{U^c_\phi}$ is free. Therefore, by Corollary \ref{theorem41}, we have:
$$[\calh_{\phi}]_{\ring{}}\in  A[\undt{}]\cap \mathbb T_s(K_{\infty})^{\times}.$$}
\end{Remark}
To proceed further, we need a precise characterization of the Drinfeld $A[\undt{}]$-modules of rank one whose exponential function is injective. This is in fact closely related to the 
non-surjectivity of $\exp_\phi$ and will be investigated in the next Section.

\section{Uniformizable Drinfeld modules of rank one.}\label{abeliandrinfeld}
Again in this section, the integer $s$ is fixed so that we can write $\undt{}$ instead of $\undt{s}$.
In this Section we consider general Drinfeld $A[\undt{}]$-modules of rank one defined over $\mathbb{T}_s$.
If $\Phi\in \mathbb T_s[\tau],$ we set
$$\mathbb T_s^{\Phi=1}= \{g\in \mathbb T_s, \Phi(g)=g\}.$$ This is a $\FF[\undt{}]$-submodule of $\mathbb{T}_s$.
Observe that in the case where $\phi=C$ is the Carlitz module over $\mathbb{T}_s$ (this is equivalent to $\alpha =1$; note that this does not 
impose any constraint on $s$), then
$\exp_C: \mathbb T_s\rightarrow \mathbb T_s$ is a surjective homomorphism of $\FF[\undt{}]$-modules. Furthermore
$$\operatorname{Ker}(\exp_C) = \widetilde{\pi} A[\undt{}],$$
and 
$$\mathbb T_s^{\tau=1}=\FF[\undt{}].$$
We are going to study a class of Drinfeld $A[\undt{}]$-modules of rank one defined over $\mathbb{T}_s$ which 
have similar properties.

\begin{Definition}
{\em Let $\phi$ be a Drinfeld $A[\undt{}]$-module of rank one over $\mathbb{T}_s$. We say that $\phi$ is {\em uniformizable} if $\exp_{\phi}$ is surjective on $\mathbb T_s.$}\end{Definition}

\begin{proposition}
\label{proposition1}
\noindent  Let $\phi$ be a Drinfeld $A[\undt{s}]$-module of rank one over $\mathbb T_s$ and let $\alpha\in\mathbb{T}_s\setminus\{0\}$ be its parameter. Recall that $\tau_{\alpha}=\alpha \tau \in \mathbb T_s[\tau ].$ The following conditions are equivalent.
\begin{enumerate}
\item $\phi$ is uniformizable,
\item $\mathbb T_s^{\tau_{\alpha}=1}\neq\{0\}$,
\item $\alpha \in \mathbb T_s^\times$,
\item $\phi$ is isomorphic to the Carlitz module over $\TT_s$.
\end{enumerate}
\end{proposition}
\begin{proof} 
We begin by proving that (3) implies (4). Since $\alpha \in \mathbb T_s^\times,$ 
there exists $x\in \mathbb C_{\infty}^\times$ such that $v_{\infty}(\alpha-x)> v_{\infty} (\alpha).$  Observe that:
$$v_{\infty}\left(\frac{\tau^i(\alpha)}{x^{q^i}}-1\right)\geq q^i(v_{\infty}(\alpha-x)-v_{\infty}(\alpha)).$$
Thus the product $\prod_{i\geq 0}\left(\frac{x^{q^i}}{\tau^i (\alpha )}\right)$ converges in $\mathbb T_s^\times.$ Now let us choose 
an element $\gamma \in \mathbb C_{\infty}^\times$ such that:
$$\gamma^{q-1}= x.$$
We set:
\begin{equation}\label{omegaalpha}
\omega_{\alpha}= \gamma \prod_{i\geq 0}\left(\frac{x^{q ^i }}{\tau^i(\alpha)}\right) \in \mathbb T_s^\times.\end{equation}
At first sight, these functions depend on the choice of $x$ but it is easy to 
show that they are defined up to a scalar factor of $\FF^\times$. 
We also notice that, for $\alpha_1,\alpha_2\in\mathbb{T}_s^\times$,
$$\omega_{\alpha_1\alpha_2}\in\FF^\times\omega_{\alpha_1}\omega_{\alpha_2},\quad \alpha_1,\alpha_2\in \mathbb T_s^\times.$$
Then:
\begin{equation}\label{differenceequationomalpha}
\tau (\omega_{\alpha} )= \alpha \omega_{\alpha}.\end{equation}
This implies that, in $\mathbb T_s[\tau],$ we have:
$$C_{\theta} \omega_{\alpha}= \omega_{\alpha}\phi_{\theta},$$
that is, $\phi$ and $C$ are isomorphic.

\medskip

In fact, it is also easy to show that (4) implies (3). Indeed, assuming that the Drinfeld module  of rank one $\phi$ is isomorphic to $C$, we see directly that
the parameter $\alpha$ of $\phi$ must satisfy $\tau(u)/u=\alpha$ for a unit $u$ of $\mathbb{T}_s$, but this implies 
that $\alpha$ is a unit as well.

\medskip

Next, we prove that (4) implies (1). By hypothesis, there exists $\vartheta \in \mathbb T_s^\times$ such that, in $\mathbb T_s[[ \tau]]:$
$$\vartheta \tau_{\alpha} = \tau \vartheta$$
We get in $\mathbb T_s[[\tau]]:$
$$\exp_{\phi} =\vartheta^{-1} \exp_{C}\vartheta.$$
Since $\exp_{C}$ is surjective on $\mathbb T_s,$ we obtain that $\exp_\phi$ is also surjective.

\medskip

We prove that (1) implies (2).
Let us then suppose that $\exp_{\phi}$ is surjective. The  map $\mathbb T_s\rightarrow \mathbb T_s$ defined by $f\mapsto \phi_{\theta}(f)$ is surjective. Explicitly,
for all $f\in \mathbb T_s$ there exists $g\in \mathbb T_s$ such that:
$$\alpha \tau(g)+\theta g=f.$$
Recall that we have set $\lambda_\theta=\exp_C(\widetilde{\pi}/\theta).$ Since $\lambda_{\theta}\not =0$ and $C_{\theta}(\lambda_{\theta})=\exp_C(\widetilde{\pi})=0,$ we have $\lambda_{\theta}^{q-1}=-\theta.$ Therefore:
$$\alpha \tau\left(\frac{g}{\lambda_\theta}\right)- \frac{g}{\lambda_\theta}= \frac{f}{\lambda_\theta^q}=-\frac{f}{\lambda_\theta\theta}$$
This implies that the map $\tau_{\alpha}-1$ is surjective on $\mathbb T_s.$

Furthermore, there exists $x\in \mathbb C_{\infty}^*$ such that $\|\alpha x^{q-1}\|=1.$ Observe that $$x^{-1}(\tau_{\alpha}-1)x=\tau_{\alpha x^{q-1}}-1.$$    Hence,  we can assume, without loss of generality, that $\|\alpha\|=1.$
Let us suppose, by contradiction, that $\mathbb{T}_s^{\tau_\alpha =1}=\{0\}$. Then, the map $f\mapsto \tau_\alpha(f)-f$ is 
an isomorphism of $\FF[\undt{}]$-modules which satisfies $\|\alpha\tau(f)-f\|<1$ if and only if
$\|f\|<1$ and $\|\alpha\tau(f)-f\|>1$ if and only if
$\|f\|>1$. In particular, this map induces an automorphism of the $\FF[\undt{}]$-module $\{f\in\mathbb{T}_s:\|f\|=1\}$.

Reducing modulo $\mathfrak{m}_{\mathbb{T}_s}$, the above map induces the $\FF[\undt{}]$-linear endomorphism of $\FF^{ac}[\undt{}]$
given by $\overline{f}\mapsto \overline{\alpha}\tau(\overline{f})-\overline{f}$ where $\overline{\alpha}\neq 0$ is the image of $\alpha$ by the reduction map in 
$\FF^{ac}[\undt{}]$ and $\overline{f}\in\FF^{ac}[\undt{}]$. One can easily verify that this endomorphism is not 
an automorphism. 

This constitutes a contradiction with the assumption that $\mathbb{T}_s^{\tau_\alpha =1}=\{0\}$.

\medskip

We finally prove that (2) implies (3). Let $g$ be a non-zero element of  $\mathbb{T}_s^{\tau_\alpha =1}$.
By $\alpha\tau(g)=g$ we deduce that $\alpha\tau(\alpha)\cdots\tau^{n-1}(\alpha)\tau^n(g)=g$ for all $n$.
If $\alpha$ were not a unit, this would contradict the finiteness of the number of irreducible factors of $g$.\end{proof}

\begin{Remark}\label{remarkabelian}{\em 
The following observation will be extensively used in the rest of this paper.
Let $\phi$ be a uniformizable Drinfeld $A[\undt{}]$-module of rank one over $\mathbb{T}_s$ of parameter $\alpha\in\mathbb{T}_s^\times$. Then the function $\exp_{\phi}$ induces an exact sequence of $A[\undt{}]$-modules:
$$0\rightarrow \frac{\widetilde{\pi}}{\omega_{\alpha}}A[\undt{}]\rightarrow \mathbb T_s \rightarrow \phi( \mathbb T_s)\rightarrow 0,$$
where $\omega_\alpha$ is defined as in (\ref{omegaalpha}).
In this case, the module $\mathbb{T}_s^{\tau_\alpha=1}$ is obviously given by:
 $$\mathbb T_s^{\tau_{\alpha}=1}=\frac{1}{\omega_{\alpha}}\FF[\undt{}].$$}
 
  \end{Remark}

\begin{Definition}
{\em Let $\phi$ be a Drinfeld $A[\undt{}]$-module of rank one over $\mathbb T_s.$ Then $\mathbb T_s$ is an $A[\undt{}]$-module via $\phi.$ Thus if $f\in A[\undt{}]\setminus\{0\},$ we define the {\em $A[\undt{}]$-module of $f$-torsion}  $\phi[f]$ by:
$$\phi[f]=\{g\in \mathbb T_s, \phi_{f}(g)=0\}.$$}
\end{Definition}

\begin{corollary} Let $\phi$ be a Drinfeld $A[\undt{}]$-module of rank one over $\mathbb{T}_s$ of parameter $\alpha\in\mathbb{T}_s\setminus\{ 0\}$.
\label{corolllary1}
\noindent The following assertions are equivalent:
\begin{enumerate}
\item $\phi$ is uniformizable,
\item $\mathbb{T}_s^{\tau_\alpha=1}$ is a $\FF[\undt{}]$-module of rank one,
\item for all $f\in A[\undt{}]\cap \mathbb T_s^\times$ we have an isomorphism of $A[\undt{}]$-modules:
$$\phi[f]\simeq \frac{A[\undt{}]}{fA[\undt{}]},$$
\item there exists $f\in A[\undt{}]\cap \mathbb T_s^\times$ such that $\phi[f]\not =\{0\}.$
\end{enumerate}
\end{corollary}

\begin{proof}
\noindent The equivalence of the properties (1) and (2) is already covered by the proof of Proposition \ref{proposition1} and by the Remark \ref{remarkabelian}.

\medskip

We show that (1) implies (3).
By Remark \ref{remarkabelian} we have that  
$$\operatorname{Ker}(\exp_\phi) =\frac{\widetilde{\pi}}{\omega_\alpha} A[\undt{}].$$ Notice also that $\exp_{\phi}$ is surjective so that if $f\in A[\undt{}]\cap  \mathbb T_s^\times,$
we also have that
$$\exp_{\phi}^{-1} (\phi[f])= \frac{\widetilde{\pi}}{f\omega_\alpha} A[\undt{}].$$
It is obvious that (3) implies (4); it remains to show that (4) implies (1). Let $\alpha$ be the parameter of $\phi$ and let us assume that
for some $f\in A[\undt{}]\cap \mathbb T_s^\times$, we have $\phi[f]\not =\{0\}$;
let $g\in \mathbb T_s\setminus\{0\}$ be such that  $\phi_f(g) =0.$ 

We can write, in $\mathbb T_s[\tau]$:
$$\phi_{f} =\sum_{i=0}^dc_i\tau_\alpha^i=\sum_{i\geq 0}^d \alpha\cdots \tau^{i-1}(\alpha) c_i\tau  ^i,$$
where for $i=0, \ldots, d,$ $c_i \in A[\undt{}],$ and $c_0=f, c_d\in \FF^\times.$ We get:
$$\sum_{i=1}^d \alpha\cdots \tau^{i-1}(\alpha)c_i \tau^i(g)= -fg.$$
Since $f\in \mathbb T_s^\times,$ we get $g=\alpha g_1,$ $g_1\in \mathbb T_s\setminus\{0\}.$ Thus (we recall that $\mathbb{T}_s$ is a 
unique factorization domain):
$$\sum_{i=1}^d \tau (\alpha)\cdots \tau^i (\alpha) c_i \tau^i(g_1)= -fg_1.$$
Therefore $g_1=\tau (\alpha) g_2,$ and $\alpha \tau (\alpha)$ divides $g$ in $\mathbb T_s.$  Thus for any $n\geq 1,$ $\alpha \cdots \tau^n(\alpha)$ divides $g$ in $\mathbb T_s.$ Therefore $\alpha \in \mathbb T_s^\times$ and $\phi$ is uniformizable by Proposition \ref{proposition1}. 
\end{proof}

\begin{Remark}{\em The definition of uniformizable $A[\undt{}]$-module is motivated by Anderson's result \cite[Theorem 4]{AND1}. It is an interesting question to 
characterize higher rank ``uniformizable" Drinfeld modules over $\mathbb{T}_s$, that is, Drinfeld modules over $\mathbb{T}_s$ which have surjective associated exponential function.}\end{Remark}

\subsection {The elements $\omega_\alpha$} \label{omega}
Let $\alpha\in \mathbb{T}_s(K_\infty)^\times$. Then, there exists $\gamma\in\mathbb{T}_s(K_\infty)^\times$ monic (as a power series in $\theta^{-1}$)
such that $\alpha=\rho \gamma$ for some $\rho \in \FF^\times$.
 
The function $\omega_{\alpha}$ defined in (\ref{omegaalpha}) is determined up to a factor in $\FF^\times$. 
Let  $x=\rho\theta^r$ with $r=-v_\infty(\alpha).$ Then $\|\alpha-x\|<\|\alpha\|.$ Therefore:
\begin{equation}\label{factorizationomega}
\omega_{\alpha}=\widetilde{\rho}\lambda_\theta^r\prod_{i\geq 0} \left(\frac{\tau^i(\alpha)}{\rho \theta^{rq^i}}\right)^{-1},
\end{equation}
where $\widetilde{\rho}\in \FF^{ac}$ is such that $\widetilde{\rho}^{q-1}=(-1)^r\rho.$ From this it is apparent that
$$\omega_\alpha\in\widetilde{\rho}\lambda_\theta^r\mathbb{T}_s(K_\infty)^\times$$ and that
\begin{equation}\label{normomega}
\|\omega_\alpha\|=q^{\frac{r}{q-1}}.\end{equation}

\begin{Remark}\label{uniquenessofomega}{\em 
Observe that $\omega_{\alpha}$ is defined up to the multiplication by an element in $\FF^\times.$ When $\rho =(-1)^r,$ we choose $\widetilde{\rho}=1.$ From now on, we will 
always use this normalization.}\end{Remark}

The proof of the next Lemma is easy and left to the reader.
\begin{lemma}\label{lemmainjectivityKinfty}
Let $\alpha$ be in $\mathbb{T}_s(K_\infty)^\times$. The following conditions are equivalent:
\begin{enumerate}
\item $\frac{\widetilde{\pi}}{\omega_\alpha}\in\mathbb{T}_s(K_\infty),$
\item If $r=-v_\infty(\alpha)$, $r\equiv1\pmod{q-1}$ and $-\alpha$ is monic. 
\end{enumerate}
\end{lemma}

\subsection{Examples}\label{examplesalphabeta}
If $s=1$ and $\alpha=t-\theta$, we have an important example:
\begin{equation}\label{omegaandersonthakur}
\omega_\alpha=\omega=\lambda_\theta\prod_{i\geq 0}\left(1-\frac{t}{\theta^{q^i}}\right)^{-1}\in\lambda_\theta \mathbb T(K_\infty).
\end{equation}
This function, introduced in Anderson and Thakur paper \cite[Proof of Lemma 2.5.4 p. 177]{AND&THA}, was also used 
extensively in \cite{PEL2,ANG&PEL}.
For general $s$, it is important to also consider the function $\omega_\alpha$ associated with the choice of 
$\alpha=\beta(t_1-\theta)\cdots(t_s-\theta)$, $\beta \in \FF^\times$. In this case,
\begin{equation}\label{omegaalphast}\omega_\alpha=\widetilde{\beta}\omega(t_1)\cdots\omega(t_s),\end{equation}
where $\widetilde{\beta}^{q-1}= \beta.$

\section{Uniformizable Drinfeld modules of rank one defined over $A[t_1, \ldots,t_s]$}\label{definedoverA}

In this section, we fix  $\phi$ a uniformizable Drinfeld $A[\undt{s}]$-module of rank one  {\em defined over $A[\undt{s}].$} By Proposition \ref{proposition1}, its parameter $\alpha$ lies in $A[\undt{s}]\cap\mathbb{T}_s^\times$ and we have a factorization $\alpha =\beta (x_1-\theta)\cdots (x_r-\theta),$  with $x_1, \cdots,x_r\in \FF(\undt{s})^{ac},$ $\beta\in\FF^\times$ and  $r=-v_{\infty}(\alpha)\in\ZZ_{\geq 0}.$

\subsection{The  torsion case}

In this subsection, we assume that $\beta=1$ and $r\equiv 1\pmod{q-1}$. If $r=0$ then $q=2$ and $L(1, \phi)=\zeta_C(1);$ in this case $\exp_C(\zeta_C(1))=1$ is a torsion point for the Carlitz module.
 We begin with  the case $r=1$. We  then have $\alpha=x-\theta$ with $x\in\FF[\undt{s}]$, so that $\rho_\alpha(a)=a(x)$ for $a\in A$, and in particular, $\rho_\alpha(\theta)=x$.
\begin{lemma}\label{corr=beta=1}
If $\beta=1$, $r=1$ and $\alpha =x-\theta,$ $x\in \FF[\undt{s}],$ we have the identity:
$$L(1, \phi)\, \omega_{\alpha}=\frac{\widetilde{\pi}}{\theta-x}.$$
 \end{lemma}
\begin{proof}  We recall that $\lambda_{\theta}=\exp_C(\frac{\widetilde{\pi}}{\theta})$ and that we have the infinite product (\ref{productpi}) which converges to $\widetilde{\pi}$. This shows:
$$\frac{\theta \lambda_{\theta}}{\widetilde{\pi}}=\sum_{i\geq 0} \frac{\widetilde{\pi}^{q^ i-1}}{D_i\theta^{q^i-1}}\in K_{\infty}.$$
Observe that for $i\geq 1:$
$$v_{\infty}\left(\frac{\widetilde{\pi}^{q^ i-1}}{D_i\theta^{q^i-1}}\right)=q^i\left(i-\frac{1}{q-1}\right)+\frac{1}{q-1}>0.$$
Therefore:
$$\frac{\theta\lambda_{\theta}}{\widetilde{\pi}} \equiv 1 \pmod{\frac{1}{\theta}\FF\left[\left[\frac{1}{\theta}\right]\right]}.$$
We recall that $\omega_{\alpha}= \lambda_{\theta}\prod_{i\geq 0}(1-\frac{x}{\theta^{q^i}})^{-1}.$ We get:
$$(x-\theta)\omega_\alpha\widetilde{\pi}^{-1}=\frac{-\theta\lambda_\theta}{\widetilde{\pi}}\prod_{i\geq 1}\left(1-\frac{x}{\theta^{q^i}}\right)^{-1} \equiv\frac{-\theta\lambda_\theta}{\widetilde{\pi}}\equiv -1 \pmod{ {\mathfrak{m}_{\mathbb{T}_s(K_\infty)}}}.$$ Hence, if we write
$$F=(x-\theta) L(1,\phi) \omega_{\alpha}\widetilde {\pi}^{-1},$$
 where $L(1, \phi)=\sum_{a\in A_+}\frac{a(x)}{a},$ we have:
 $$F\equiv -1 \pmod{\mathfrak{m}_{\mathbb{T}_s(K_\infty)}}.$$
Now, we have that $\exp_{\phi}(L(1, \phi))=1$ (Remark \ref{caserinfq-1}), thus:
$$\exp_{\phi}((x-\theta)L(1, \phi))= \phi_{x-\theta}(\exp_{\phi}(L(1, \phi)))= \phi_{x-\theta}(1)= x-\phi_{\theta}(1)=0.$$
Therefore  $(x-\theta)L(1, \phi)\in \operatorname{Ker}(\exp_{\phi})=\frac{\widetilde{\pi}}{\omega_{\alpha}}A[\undt{s}] $ so that $F\in A[\undt{s}].$ But $v_{\infty} (F+1)>0$  which implies  $F=-1$.
\end{proof}
The above Lemma implies \cite[Theorem 1]{PEL2}.
\begin{proposition}\label{lemma3bis}
Let $\phi$ be the Drinfeld $A[\undt{s}]$-module of rank one of parameter $\alpha= (x_1-\theta)\cdots (x_r-\theta)\in A[\undt{s}],$  with $x_1, \ldots,x_r\in \FF(\undt{s})^{ac},$   $r\geq 1, r\equiv 1\pmod{q-1}.$ The following properties hold.
 \begin{enumerate}
\item If $r\geq q$, $U_\phi=\frac{\widetilde{\pi}}{\omega_{\alpha}}A[\undt{s}]$ and $\exp_\phi(U_\phi)=0$.
\item The module $H_{\phi}$ is a free  $\FF[\undt{s}]$-module of rank $u(\alpha).$ 
 Moreover  $U_\phi/U^c_\phi$ is isomorphic to $H_{\phi}$ as a $\FF[\undt{s}]$-module.
 \end{enumerate}
 \end{proposition}
 \begin{proof} 
 \noindent (1). By Lemma \ref{lemmainjectivityKinfty} and the identity (\ref{normomega})
we see that $\|\widetilde{\pi}/\omega_\alpha\|=q^{u(\alpha)}$ so that 
\begin{equation}\label{usefuldecomposition}
\mathbb T_s(K_{\infty})=\frac{\widetilde{\pi}}{\omega_{\alpha}}A[\undt{s}]\oplus \nalpha .\end{equation}
Let $f$ be in $U_\phi$ and let us write $f=f_1+ f_2$ with $f_1\in \frac{\widetilde{\pi}}{\omega_{\alpha}}A[\undt{s}]$ and $f_2\in \nalpha $. Since 
$f_1$ is in the kernel of $\exp_\phi$,
we have $\exp_\phi(f)=\exp_\phi(f_2)\in \nalpha $. Since $\exp_\phi$
induces an isometric automorphism of $\nalpha $, the condition $\exp_\phi(f)\in A[\undt{s}]$ yields $f_2=0$.
This means that $U_\phi=\frac{\widetilde{\pi}}{\omega_\alpha}A[\undt{s}]$ as expected.

\medskip 

\noindent (2).  By (\ref{usefuldecomposition}),  $\nalpha  =\exp_{\phi}(\mathbb T_s(K_{\infty}))$ and therefore is 
an $A[\undt{s}]$-module via $\phi$.  
We have:
$$H_\phi=\frac{\phi(\mathbb{T}_s(K_\infty))}{\phi(A[\undt{s}])\oplus\phi(\nalpha )}.$$
In particular, $H_\phi$, as a $\FF[\undt{s}]$-module, is isomorphic to $A[\undt{s}]\theta^{-u(\alpha)}/A[\undt{s}],$ and hence is free of rank  $u(\alpha).$ Finally, in (\ref{anotherusefulexactsequence}), the third arrow maps to zero so that $U_\phi/U^c_\phi\cong H_\phi$ as a $\FF[\undt{s}]$-module. \end{proof}
We deduce the next Corollaries.
\begin{corollary}\label{corllaryr=q} Let $\alpha$ be as in (\ref{alpha2}) with $\beta=1$ and $r=q$. Let $\phi$ be the Drinfeld $A[\undt{s}]$-module of 
rank one with parameter $\alpha$. Then, 
the following formula holds:
$$L(1,\phi)=-\frac{\widetilde{\pi}}{\omega_\alpha}.$$
\end{corollary}

 \begin{corollary}\label{evencorollary}
 If $\beta=1$, $r\equiv1\pmod{q-1}$ and $r\geq q,$ then
 $$L(1,\phi)\in A[\undt{s}]\frac{\widetilde{\pi}}{\omega_\alpha}.$$
 \end{corollary}

\begin{Remark}\label{remar72}{\em The results of Lemma  \ref{corr=beta=1} and Corollary \ref{evencorollary} also justify the terminology {\em  torsion case} because  $\exp_{\phi}(L(1,\phi))$ is a torsion point for $\phi.$ And by Lemma \ref{lemmainjectivityKinfty},  for $\phi$ a uniformizable Drinfeld $A[\undt{s}]$-module of rank one  defined over $A[\undt{s}],$ $\exp_{\phi}(L(1, \phi))$ is a torsion point for $\phi$ if and only if $r\equiv 1\pmod{q-1}$ and $\beta=1.$

The terminology is also suggested by the behavior of the higher Carlitz zeta values $\zeta_C(n)=
\sum_{a\in A_+}a^{-n}$. In \cite{AND&THA}, Anderson and Thakur constructed a point $z_n\in\operatorname{Lie}(C^{\otimes n})(K_\infty)$ with last entry $\Pi(n-1)\zeta_C(n)$ such that $\operatorname{Exp}_n(z_n)=Z_n$,
where $\Pi$ denotes the {\em Carlitz factorial} (see \S \ref{evaluationpolynomials}), $\operatorname{Exp}_n$ denotes the exponential function of $C^{\otimes n}$ and $Z_n$
is a certain $A$-valued {\em special point} of $C^{\otimes n}$ explicitly constructed in loc. cit. We have that $Z_n$ is a torsion point for
$C^{\otimes n}$ if and only if $q-1$ divides $n$ (see Anderson and Thakur, \cite[Corollary 3.8.4]{AND&THA} and J. Yu, \cite[Corollary 2.6]{YU}).

The methods of \cite[Theorem 4]{ANG&PEL} can probably be used to show that, more generally, 
$\widetilde{\pi}^{-n}L(n,\phi)\omega_\alpha$ is rational if and only if $r\equiv n\pmod{q-1}$ and $-\alpha$ is monic. 
It would be nice to see if these are also 
related to torsion points for the tensor powers of the modules $\phi$ as in \cite{AND&THA} in the 
case $s=0$.
}\end{Remark}

  \subsubsection{The polynomials $\mathbb{B}_\phi$}\label{subsectionb}
If $\alpha$ is as in (\ref{alpha2}) with $r=-v_\infty(\alpha)$ such that 
$r\equiv 1\pmod{q-1}, $ $r\geq q,$ by  Corollary \ref{evencorollary}, we have that 
\begin{equation}\label{Bphi}
 \mathbb{B}_{\phi}:=(-1)^{\frac{r-1}{q-1}}L(1, \phi)\omega_{\alpha}\widetilde{\pi}^{-1}\in A[\undt{s}].
\end{equation}
We also set for  $r=1:$ 
$$\mathbb{B}_\phi=\frac{1}{\theta-x},$$ 
where $x\in \FF[\undt{s}]$ is the unique root of $\alpha$ as a polynomial in $\theta$.

The polynomials $\mathbb{B}_\phi\in\FF[\undt{s}][\theta]$ have already been studied in \cite{ANG&PEL} in the case of
 $\alpha=(t_1-\theta)\cdots(t_r-\theta)$ with $r=s$.
 If $r=q$, we can even deduce the exact value of $\mathbb{B}_\phi$ (see Corollary \ref{corllaryr=q}): $\mathbb{B}_\phi=1$.
More generally, we have the following, for $r>1$:
 \begin{lemma}
 \label{lemma4} The polynomial $\mathbb{B}_\phi\in \FF[\undt{s}][\theta]$ is a monic polynomial of degree $u(\alpha)=\frac{r-q}{q-1}$ in the indeterminate $\theta$.
 \end{lemma}
 \begin{proof} Let us write:
$$\mathbb{B}_\phi=\sum_{i=0}^m a_i \theta ^i,$$
where $a_i \in \FF[\undt{s}],$ and $a_m\not =0.$  
We have that $v_\infty(\widetilde{\pi}^{-1}L(1,\phi)\omega_\alpha)=v_\infty(\mathbb{B}_\phi)=\frac{q-r}{q-1}$, which implies that 
$$m=\frac{r-q}{q-1}$$ and $$a_m\in \FF^\times.$$

 To compute $a_m$ it suffices to compute the leading coefficient of the expansion of $\widetilde{\pi}^{-1}L(1,\phi)\omega_\alpha$
 as a series of $\FF(\undt{s})((\theta^{-1}))$. This computation is easy and left to the reader.
\end{proof}  

The importance of the polynomials $\mathbb{B}_\phi$ is dictated by the next Theorem.
  \begin{theorem}
 \label{theorem4}Let $r\geq q,$ then
 $$\operatorname{Fitt}_{A[\undt{s}]} (H_{\phi} )= \mathbb{B}_{\phi}A[\undt{s}].$$
   \end{theorem}
 \begin{proof} Since $H_\phi$ is free of rank $u(\alpha)$ (part (2) of Proposition \ref{lemma3bis}), we have that
 $\operatorname{Fitt}_{A[\undt{s}]}(H_\phi)=FA[\undt{s}]$, where 
 $$F= \det{}_{\FF[\undt{s}]} (Z - \phi_{\theta}| _{H_{\phi}})|_{Z=\theta},$$
 and $F$ has degree $u(\alpha)$ as a
 polynomial in $\theta$.  Again by (1) of Proposition \ref{lemma3bis}, we have $U_{\phi}=\frac{\widetilde {\pi}}{\omega_{\alpha}}A[\undt{s}]$
 and $$[\ring{}: \calu_\phi]_{\ring{}} = \left[\ring{}: \frac{\widetilde {\pi}}{\omega_{\alpha}} \ring{}\right]_{\ring{}} = (-1) ^{\frac{r-1}{q-1}}\frac{\widetilde {\pi}}{\omega_{\alpha}}$$ (in the notation on \S \ref{aroundcalr}),
 because $(-1)^{\frac{r-1}{q-1}}\frac{\widetilde {\pi}}{\omega_{\alpha}}$ is monic. It remains to apply Corollary \ref{theorem41}.\end{proof}

We presently do not know much about the irreducible factors of the polynomials $\mathbb{B}_\phi$. However,
if $\alpha=(t_1-\theta)\cdots(t_s-\theta)$ (that is, if $\phi=C_s$) with $s\equiv1\pmod{q-1}$, more can be said.
\begin{lemma}\label{varyingzeta}
If $s\equiv1\pmod{q-1}$, $s\geq q,$ then $\mathbb{B}_{C_s}$ has no non-trivial divisor in $A$.
\end{lemma}
\begin{proof}
In this case, we have $\omega_\alpha=\omega(t_1)\cdots\omega(t_s)$ and $L(1,C_s)=L(\chi_{t_1}\cdots\chi_{t_s},1)$
in the notation of \cite{ANG&PEL}. We can evaluate at $t_1=\cdots=t_s=\zeta\in\FF$ and, by the fact that $s\equiv1\pmod{q-1}$,
$L(1,C_s)|_{t_i=\zeta}=\sum_{a\in A_+}\frac{a(\zeta)}{a}$. By using Lemma \ref{corr=beta=1}, we obtain 
$$\sum_{a\in A_+}\frac{a(\zeta)}{a}=\frac{\widetilde{\pi}}{(\theta-\zeta)\omega(\zeta)}.$$
Therefore,
$$\mathbb{B}_{C_s}|_{t_i=\zeta}=\widetilde{\pi}^{-1}L(1,C_s)\omega_\alpha|_{t_i=\zeta}=\omega(\zeta)^{r-1}(\theta-\zeta)^{-1}=(\theta-\zeta)^{\frac{r-q}{q-1}}\in A.$$
Now, if $a\in A\setminus\{0\}$ divides $\mathbb{B}_{C_s}$ in $A[\undt{s}]$, then $a$ divides $(\theta-\zeta)^{\frac{r-1}{q-1}}$
for all $\zeta\in\FF$ so that $a\in\FF^\times$.
\end{proof}

\begin{proposition}\label{propo7puntodue} Let us suppose that $s\geq 2q-1$ and $s\equiv1\pmod{q-1}$. Then the $A$-module $H_{C_s}$ is torsion-free and not finitely generated.
\end{proposition}
\begin{proof} 
Proposition \ref{lemma3bis} asserts that $H_{C_s}$ is a $\FF[\undt{s}]$-module free of rank $\frac{s-q}{q-1}\geq 1$. Thus the assertion that the $A$-module $H_{C_s}$ is torsion-free is a consequence of Lemma \ref{varyingzeta}.
Now, it is a general fact that a non-trivial $\FF[\undt{s}, \theta]$-module $M$ cannot be simultaneously free of finite rank over $\FF[\undt{s}]$ and over $A$ if $s>1$. Let us suppose by contradiction that $M$ is a non-trivial
$\FF[\undt{s}, \theta]$-module which is free of finite rank as a module over
$\FF[\undt{s}]$ and over $A=\FF[\theta]$. Then $\operatorname{End}_{A}(M)$ would be isomorphic to $\operatorname{Mat}_{n\times n}(\FF)[\theta]$ as an $A$-module. For $1\leq i \leq s$ the actions on $M$ of $t_i$ and of $\theta$ commute, and, for all $i$, the multiplication by $t_i$ defines an element $T_i\in \operatorname{End}_{A}(M)$. Since $M$ is free over $\FF[\undt{s}]$, we deduce that if $i\neq j$, $T_i$ and $T_j$ are algebraically independent over $\FF$. This is not possible. Now, for $s$ as in our hypotheses, the module $H_{C_s}$ is finitely generated over $\FF[\undt{s}]$ and non-trivial. 
Hence, it is not finitely generated over $A$.
\end{proof}

 \subsection{The non-torsion case} In this subsection, we consider the Drinfeld module $C_s$ (recall that this is the Drinfeld $A[\undt{s}]$-module of rank one  with parameter $(t_1-\theta)\cdots (t_s-\theta)$) and we assume that   $s\not\equiv1\pmod{q-1}, s\geq 2q-1$. Let  $M$ be the $A$-torsion submodule  of $H_{C_s}$. This also is an $A[\undt{s}]$-module, and we know that it is a finitely generated $\FF[\undt{s}]$-module (Corollary \ref{corollaryHszero}).
Moreover:
\begin{proposition}\label{propM}
The  $A[\undt{s}]$-submodule $M$ is a torsion $\FF[\undt{s}]$-module.
\end{proposition}
\begin{proof}
We must show that $M\otimes_{\FF[\undt{s}]}\FF(\undt{s})=\{0\}$. By the isomorphism (\ref{secondisom}), it is enough to show that
$[\calh_{C_s}]_{\ring{}}$ has no divisors in $A$, where $\ring{}=\FF(\undt{s})[\theta]$.
By part (1) of Proposition \ref{proposition2}, we know that $\calu_{C_s}=\FF(\undt{s})U_{C_s}$ is an $\ring{}$-module free of rank one. 
The class number formula, Theorem \ref{theorem2}, yields that 
$$[\calh_{C_s}]_{\ring{}}\calu_{C_s}=\ring{}L(1,{C_s}).$$
Let $a\in A\setminus\{0\}$ be a divisor of $[\calh_{C_s}]_{\ring{}}$. Then, $a^{-1}L(1,{C_s})\in \calu_{C_s}$. By part (2) of Proposition \ref{proposition2},
we have that $\exp_{C_s}(a^{-1}L(1,{C_s}))\in \ring{}$. Since we also have, at once, $\exp_{C_s}(a^{-1}L(1,{C_s}))\in \mathbb{T}_s(K_\infty)$, we obtain
that $$\exp_{C_s}(a^{-1}L(1,{C_s}))\in A[\undt{s}].$$
We claim that this is impossible unless $a\in\FF^\times$. To see this, we appeal to Proposition \ref{entireness} which
says us that $\exp_{C_s}(a^{-1}L(1,{C_s}))$ extends to an entire function in $s$ variables. 

It is here that we use the particular shape of the parameter $\alpha$.
Indeed, $\alpha$ vanishes at $t_s=\theta$. The evaluation at $t_s=\theta$ in $\exp_{C_s}(a^{-1}L(1,{C_s}))$ yields an
entire function in $s-1$ variables $t_1,\ldots,t_{s-1}$. Since
\begin{equation}\label{polynomialsvs}
\exp_{C_s}(a^{-1}L(1,{C_s}))=\sum_{k\geq 0}\sum_{i+j=k}\frac{\alpha\tau(\alpha)\cdots\tau^{i-1}(\alpha)}{a^{q^{i}}D_i}\sum_{b\in A_{+,j}}\frac{\chi_{t_1}(b)\cdots\chi_{t_s}(b)}{b^{q^j}},\end{equation}
evaluating at $t_s=\theta$ gives,
$$\exp_{C_s}(a^{-1}L(1,{C_s}))|_{t_s=\theta}=a^{-1}L(0,C_{s-1})\in a^{-1}\FF[\undt{s-1}]\cap \ring{s-1}$$ where $\ring{s-1}=\FF(\undt{s-1})[\theta]$ and $C_{s-1}$ is the Drinfeld module of rank one of parameter 
$$\alpha'=(t_1-\theta)\cdots(t_{s-1}-\theta).$$ If by contradiction $a\not\in\FF^\times$, then $a^{-1}\FF[\undt{s-1}]\cap \ring{s-1}=\{0\}$
and $$\exp_{C_s}(a^{-1}L(1,{C_s}))|_{t_s=\theta}=L(0,C_{s-1})=0.$$
However, $L(0,C_{s-1})\neq0$. Indeed, by hypothesis, $s-1\not\equiv0\pmod{q-1}$ and, by \cite{GOS}, page 278, line 4, we have: $$L(0,C_{s-1})|_{t_1=\cdots=t_{s-1}=\theta}=\zeta_C(1-s)\in A\setminus\{0\}$$  which yields a contradiction. Therefore, $a\in\FF^\times$.
\end{proof}
\section{On the log-algebraicity Theorem of Anderson}\label{anderson}

We first recall Anderson's  {\em log-algebraicity Theorem} for the Carlitz module (cf. \cite[Theorem 3]{AND2}; see also loc. cit. Proposition 8). Let $Y,z$ be two indeterminates over $\mathbb C_{\infty}.$ Let $$\tau: \mathbb C_{\infty}[Y][[z]]\rightarrow \mathbb C_{\infty}[Y][[z]]$$ be the map $f\mapsto f^q.$ G. Anderson proved that for $n\in \mathbb N:$
$$\exp_C\left(\sum_{d\geq 0}\sum_{a\in A_{+,d}}\frac{C_a(Y)^n}{a} z^{q^d}\right)\in A[Y,z].$$
It turns out that the class number formula (Theorem \ref{theorem2}) implies a refined version of Anderson's log-algebraicity Theorem in the case of the Carlitz module as we will explain below.

We consider, for $r\in \mathbb N,$ and  for all $1 \le j \le r, i\in \mathbb N$, ``symbols" $X_1,\ldots,X_r,Z$, $\tau(X_1), \dots\tau(X_r), \dots, \tau^i(X_j), \dots$ Let us consider the polynomial ring in infinitely many indeterminates
$$\mathcal{B}_r=\mathbb C_{\infty}[X_1,\ldots,X_r,\tau(X_1),\ldots,\tau(X_r),\tau^2(X_1),\ldots,\tau^2(X_r),\ldots].$$
We extend the action of $\tau$ to $\mathcal{B}_r$ by setting $$\tau(\tau^m(X))=\tau^{m+1}(X)$$ with $X=X_1,\ldots,X_r$, and $\tau(c)=c^q$ for $c\in\CC_\infty$.
 We use the ring $\mathcal{B}_r$ (and the so far unused indeterminate $Z$) to construct yet another ring, non-commutative,
denoted by $\mathcal{A}_r$. This is the set of infinite series $$\sum_{i\geq 0} c_i\tau^i(Z)$$
with the coefficients $c_i\in\mathcal{B}_r$, the sum being the usual one  and the product is given by the following rule. For
$$F=\sum_{i\geq0}f_i\tau^i(Z),\quad G=\sum_{j\geq0}g_j\tau^j(Z),$$
we set
$$F\cdot G:=\sum_{k\geq0}\left(\sum_{i+j=k}f_i\tau^i(g_j)\right)\tau^k(Z).$$
Note that the action of $\tau$ on  $\mathcal{B}_r$ extends to an action of $\tau$  on  $\mathcal{A}_r$  by setting $$\tau(\tau^i(Z))=\tau^{i+1}(Z).$$
We identify $A[X_1,\ldots,X_r,Z,\tau(X_1),\ldots,\tau(X_r),\tau(Z),\ldots]$ with the subring of $\mathcal{A}_r$ consisting of elements $\sum_{i\geq 0} c_i\tau^i(Z)$ where the sequence of coefficients 
$c_i \in A[\tau^i(X_j) ; \linebreak[1] 1\le j\le r, i\ge 0] \subset \mathcal{B}_r$ is ultimately $0$.
The series 
$$\mathcal{L}_r(X_1,\ldots,X_r;Z)=\sum_{d\geq 0}\left(\sum_{a\in A_{+,d}}C_a(X_1)\cdots C_a(X_r)a^{-1}\right)\tau^d(Z)$$
defines an element of ${\mathcal{A}_r}$. 
Let ${\exp}_C=\sum_{i\geq 0}D_i^{-1}\tau^i$ be the operator associated to Carlitz's exponential. Obviously, $$S_r(X_1,\ldots,X_r;Z):=\exp_C(\mathcal{L}_r(X_1,\ldots,X_r;Z))$$ is an element of ${\mathcal{A}_r}$.
But more is true.

\begin{theorem}\label{theoanderson} We have that
$$S_r(X_1,\ldots,X_r;Z)\in A[X_1,\ldots,X_r,Z,\tau(X_1),\ldots,\tau(X_r),\tau(Z),\ldots].$$\label{logalg}\end{theorem}
\begin{proof}
Let $\phi$ be the Drinfeld $A[\undt{r+1}]$-module of rank one  whose parameter
is $$\alpha'=t_{r+1}(t_1-\theta)\cdots(t_r-\theta)\in A[\undt{r+1}].$$  For $i=1, \ldots, r+1,$ 
there is a unique homomorphism of $\CC_\infty$-algebras
$$ t_i: \mathcal{A}_r\rightarrow \mathcal{A}_r$$
defined by the following table of multiplication, for $m\in\NN$:
\bigskip
\begin{center}
\begin{tabular}{|l|l|}
\hline
$t_i.\tau^m(X_j)=\tau^m(X_j)$ & $i\neq j$ \\

$t_i.\tau^m(Z)=\tau^m(Z)$ & $i\neq r+1$ \\

$t_i.\tau^m(X_i)=\tau^m(C_\theta(X_i))$ & $i\leq r$ \\

\hline
$t_{r+1}.\tau^m(X_i)=\tau^m(X_i)$ & $i\leq r$ \\

$t_{r+1}.\tau^m(Z)=\tau^{m+1}(Z)$ & \\
\hline
\end{tabular}
\end{center}
\bigskip
(we notice on the way the identity $\tau^m(C_\theta(X_i))= \tau^{m+1}(X_i)+\theta^{q^m} \tau^m(X_i)$).
We can endow  $\mathcal{A}_r$ with the structure of a $K[\undt{r+1}]$-algebra. The underlying 
$K[\undt{r+1}]$-module structure can be described as follows. If $f\in\mathcal{A}_r$ and if
$g\in K[\undt{r+1}]$ has expansion
$$g=\sum_{i_1, \ldots i_{r+1} \in \mathbb N}g_{i_1, \ldots,i_{r+1}} t_1^{i_1}\cdots t_{r+1}^{i_{r+1}},\quad g_{i_1, \ldots,i_{r+1}}\in K,$$
then we have
$$g.f=\sum_{i_1, \ldots i_r \in \mathbb N}g_{i_1, \ldots,i_{r}} (t_1^{i_1}.f)\cdots (t_{r+1}^{i_{r+1}}.f).$$
We deduce, from the above multiplication table, the identity, for $1\leq i\leq r,$ $j\geq 0$ and $m_1, \ldots,m_{r}\in \mathbb N$:
\begin{eqnarray*}
\lefteqn{t_i^j.(\tau^{m_1}(X_1)\cdots \tau^{m_r}(X_r))=}\\ &=&\tau^{m_1}(X_1)\cdots \tau^{m_{i-1}}(X_{i-1}) \tau^{m_{i}}(C_{\theta^j}(X_{i}))\tau^{m_{i+1}}(X_{i+1})\cdots 
\tau^{m_r}(X_r).\end{eqnarray*}
Thus:
$$(a(t_1)\cdots a(t_r)).\tau^m(X_1\cdots X_r)=\tau^m(C_a(X_1)\cdots C_a(X_r)),\quad a\in A.$$
In fact, the action of $K[\undt{r+1}]$ extends to an action of $K[\undt{r}][[t_{r+1}]]$ in the following way.
If $F=\sum_{i\ge 0} F_it_{r+1}^i\in K[\undt{r}][[t_{r+1}]]$, we set:
$$F.(X_1\dots X_rZ) = \sum_{i\ge 0} F_it_{r+1}^i.(X_1\dots X_rZ) = \sum_{i\ge 0} F_i.(X_1\dots X_r\tau^i(Z))\in \mathcal A_r.$$
We observe that:
$$L(1, \phi)=\sum_{n\geq 0} \sum_{a\in A_{+,n}} \frac{a(t_1)\cdots a(t_r)}{a} t_{r+1}^n\in\TT_{r+1}\cap K[\undt{r}][[t_{r+1}]].$$
Therefore, the multiplication $L(1, \phi).(X_1\dots X_rZ)$ is well defined and we have:
$$L(1, \phi).(X_1\dots X_rZ) = \mathcal L_r(X_1, \dots, X_r ; Z).$$

We also recall that:
$$\exp_{\phi} = 1+\sum_{i\geq 1} \frac{\alpha '\cdots \tau^{i-1}(\alpha ')}{D_i} \tau ^i.$$
This defines an element of $\mathcal{A}_r$ again denoted by $\exp_\phi$ and, for all $F\in\mathcal{A}_r$, 
we have $\exp_\phi\cdot F=\exp_\phi(F)$, which justifies that we are using the same notation for a series of $\mathcal{A}_r$ and a series of $K[\undt{r+1}][[\tau]]$.
We claim that 
$$ \exp_C( \mathcal L_r(X_1, \dots, X_r ; Z))=(\exp_{\phi}(L(1, \phi))).(X_1\cdots X_rZ).$$
Indeed, if we choose $1\leq i\leq r,$ and integers $m_1, \ldots, m_{i-1}, m_{i+1},\ldots, m_r, n\in \mathbb N,$ we have, for $j\geq 1$, that the element of $\mathcal{A}_r$
$$((t_i-\theta)\cdots (t_i-\theta^{q^{j-1}})).(\tau^{m_1}(X_1)\cdots\tau^{m_{i-1}}(X_{i-1}) X_i\tau^{m_{i+1}}(X_{i+1})\cdots \tau^{m_r}(X_r) \tau^{n}( Z))$$ is equal to:
$$\tau^{m_1}(X_1)\cdots \tau^{m_{i-1}}(X_{i-1}) \tau^j(X_i)\tau^{m_{i+1}}(X_{i+1})\cdots \tau^{m_r}(X_r) \tau^{n}( Z).$$
This implies that for $i\geq 1$:
$$\alpha'\tau(\alpha')\cdots\tau^{i-1}(\alpha').(X_1\cdots X_rZ)=\tau^i(X_1\cdots X_rZ),$$
from which we deduce the claim.

But Corollary \ref{lemma2} implies that $$\exp_\phi(L(1,\phi))\in A[\undt{r+1}],$$ thus we can conclude that:
\begin{eqnarray*}
\lefteqn{S_r(X_1, \dots, X_r ; Z) =}\\ &=&\exp_C( \mathcal L_r(X_1, \dots, X_r ; Z))\in A[X_1,\ldots,X_r,Z,\tau(X_1),\ldots,\tau(X_r),\tau(Z),\ldots].\end{eqnarray*}

\end{proof}
The above Theorem implies a multivariable version of Anderson's log-algebraicity Theorem:
\begin{corollary}\label{corollaryAnderson}
Let $r\geq 0$ be an integer and let $Y_1, \cdots, Y_r, z$ be $r+1$ indeterminates over $\mathbb C_{\infty}.$ Let $\tau: \mathbb C_{\infty}[Y_1, \ldots, Y_r][[z]]\rightarrow \mathbb C_{\infty}[Y_1, \ldots, Y_r][[z]], f\mapsto f^q.$ Then:
$$\exp_C\left(\sum_{d\geq 0}\sum_{a\in A_{+,d}}\frac{C_a(Y_1)\cdots C_a(Y_r)}{a} z^{q^d}\right)\in A[Y_1,\ldots, Y_r,z].$$
\end{corollary}
\begin{proof} Let $\psi: \mathcal{A}_r \rightarrow \mathbb C_{\infty}[Y_1, \ldots, Y_r][[z]]$ be the morphism of $\mathbb C_{\infty}$-algebras given by: for $m\in \mathbb N, \psi (\tau^m(X_i))= Y_i^{q^m}, 1\leq i\leq r,$ and $\psi (\tau^m(Z))=z^{q^m}.$ Then,
$$\text{ for all } f\in \mathcal {A}_r, \psi(\tau(f))=\tau(\psi(f)).$$
The Corollary follows from Theorem \ref{logalg}.
\noindent
\end{proof}
\begin{Remark}{\em Even though it only applies to the Carlitz module, Theorem \ref{logalg} has an advantage if compared to Anderson original result \cite[Theorem 3]{AND2}, and this, even if we forget the occurrence of the distinct variables $X_1,\ldots,X_r$. Indeed, these variables can vary in the Tate algebra $\mathbb{T}_s$, while 
Anderson's result holds if the variable is chosen in $\CC_\infty$. Let us assume, for the sake of simplicity,
that $X_1=\cdots=X_r=X$. In \cite[\S 4.3]{AND2}, Anderson also provides a table of {\em special polynomials} of small order for small values of $q$. For example, if $q=3$ and $r=4$, we have the formula (cf. loc. cit. p. 191):
\begin{equation}\label{fromandersontables}
\exp_C\left(\sum_{k\geq 0}Z^{q^k}\sum_{a\in A_{+,k}}\frac{C_a(X)^4}{a}\right)=ZX^4-Z^3X^6.
\end{equation} This formula has to be understood {\em with the variables $X,Z$ varying in $\CC_\infty$ so that $|Z|$ is small enough to ensure convergence}. If the variables are chosen in $\mathbb{T}_s$, the formula no longer holds. 
It can be proved, with an explicit computation, that,
again for $q=3$,

\begin{eqnarray*}
\lefteqn{
S_4(X_1,\ldots,X_4;Z)}\\ &=&ZX_1\cdots X_4-\tau(Z)(X_1X_2X_3\tau(X_4)+\\
& &X_1X_2\tau(X_3)X_4+X_1\tau(X_2)X_3X_4+
\tau(X_1)X_2X_3X_4).\end{eqnarray*}
If we choose $X_1=\cdots=X_4=X$, then we get
$$S_4(X,\ldots,X;Z)=ZX^4-\tau(Z)X^3\tau(X),$$
so that, if $X,Z\in\CC_\infty$, we recover the original entry of Anderson table (\ref{fromandersontables}). Of course, further information about the polynomials $S(X_1,\ldots,X_r;Z)$ can be made explicit 
in the same spirit of \cite[Proposition 8]{AND2}; we refer the interested reader to a forthcoming work of the authors.
}\end{Remark}


\section{Evaluation at Dirichlet characters}\label{evaluationdirichletcharacters}
In this Section, more involved than the previous ones, we prove a generalization of Herbrand-Ribet-Taelman Theorem \cite{TAE1}. The difficulties to overcome are due to the evaluations at roots of unity that we have to 
control, in order to extract information about the  Taelman's class modules associated to cyclotomic function field extensions of $K$ (see  below) 
from the structure of the ``generic class modules" $H_{C_s}$ studied in \S
\ref{subsectionb}.

We recall that, for $a\in A^+$, we have set $\lambda_a=\exp_C(\frac{\widetilde{\pi}}{a})$ and we have denoted by $K_a=K(\lambda_a)$ the $a$-th cyclotomic field extension of $K$.
Let $P$ be a prime of $A.$ Then Taelman's class module associated to the extension $K_P/K$ and to the Carlitz module is the finite dimensional $\FF$-vector space:
$$H(C/A[\lambda_P])= \frac{K_P\otimes_KK_{\infty}}{\exp_C(K_P\otimes_KK_{\infty})+A[\lambda_P]}.$$
We notice that $A[\lambda_P]$ is the integral closure of $A$ in $K_P.$
 This $\FF$-vector space is equipped with a structure of $A$-module via $C$ and $\Delta_P={\rm Gal}(K_P/K)$ acts on this module. Since $\Delta_P$ is abelian of order prime to $p,$ one can study the isotypic components of $$H(C/A[\lambda_P])\otimes_{\FF} \frac{A}{PA}.$$ This is precisely what is done in \cite{ANG&TAE}. In particular, in loc. cit., the authors prove an ``equivariant class number formula"  and with the help of such a formula they are able to recover an analogue of the Herbrand-Ribet Theorem  which was originally obtained by L. Taelman by using different methods of proof (see \cite{TAE3}).

The basic idea in this section has its origins in \cite{PEL2} and  \cite{TAE3}. Let $s\geq 1$ be an integer. If $f\in \mathbb T_s(K_{\infty}),$ we can evaluate $f$ at the points in $(\FF^{ac})^s.$ Now let $\chi$ be a Dirichlet character of type $s$ (see \S \ref{somesettings}), to such a character we can associate a point $\underline{\zeta}_{\chi}\in  (\FF^{ac})^s,$ and we therefore have a morphism (see \S \ref{evaluationmapss}) of $K_{\infty}$-algebras $$\operatorname{ev}_{\chi}: \mathbb T_s(K_{\infty})\rightarrow K_{\infty}(\chi),\quad  f\mapsto f(\underline{\zeta}_{\chi}),$$ where $K_{\infty}(\chi)$ is the field obtained by adjoining to $K_{\infty}$ the values of the character $\chi.$ For example:
$$\operatorname{ev}_{\chi} (L(1, C_s))=L(1, \chi),$$
where we recall that $C_s$ is the Drinfeld $A[\undt{s}]$-module of rank one whose parameter is $(t_1-\theta)\cdots (t_s-\theta),$ and  $$L(1,\chi)=\sum_{a\in A_{+}}\frac{\chi(a)}{a}\in K_{\infty}(\chi)$$ is the value at one of the Goss abelian $L$-function associated to the character $\chi.$

Let $a$ be the conductor of the Dirichlet character $\chi$. Then we prove (see \S \ref{484}) the crucial fact that the map $\operatorname{ev}_{\chi}$ induces a surjective morphism of $A$-modules between $H_{C_s}$ and the $\chi$-isotypic component of  Taelman's class module $H(C/A[\lambda_a]),$  where $A[\lambda_a]$ is the integral closure of $A$ in the $a$-th cyclotomic function field. This enables us to study isotypic components of Taelman's class modules in families with the help of the results obtained in the previous sections and the recent results  in \cite{ANG&PEL,ANG&PEL2,ANG&TAE}.

\subsection{Setting.}\label{somesettings} Let $P$ be a prime of $A$ of degree $d\geq 1$. We recall that $\KPhat$ denotes the $P$-adic completion of $K$. Let $\mathbb C_P$ be the completion of a fixed algebraic closure ${\KPhat}^{ac}$ of $\KPhat.$ Since we are going to study certain congruences modulo $P,$ we choose once and for all a $K$-embedding 
$$\iota_P:K^{ac}\rightarrow \mathbb{C}_P.$$
We  normalize the valuation $v_P$ at the place $P$ by setting $v_P(P)=1$. We recall that $\Delta_P$ is the Galois group of the finite abelian extension $K(\lambda_P)/K$ where $\lambda_P=\exp_C\left(\frac{\widetilde{\pi}}{P}\right)$, and that there is an isomorphism $\left(A/PA\right)^\times \simeq \Delta_P$ given by $b\mapsto \sigma_b$ where $\sigma_b$ satisfies \eqref{sigma}.

The {\em Teichm\"uller character} is the unique morphism $\vartheta_P:\Delta_P\rightarrow(\FF^{\text{ac}})^{\times}$ such that, for all $a\in A\setminus AP$,
$$v_P(\iota_P(\vartheta_P(\sigma_a))-a)\geq 1.$$
Note that  there exists a unique element  $\zeta_P$  in $\FF^{\text{ac}}$ such that $v_P(\theta-\iota_P(\zeta_P))>0.$ Observe also that $P(\zeta_P)=0.$ Furthermore,  for  $\sigma_b\in\Delta_P$ with  $b\in A\setminus PA,$  we have $\vartheta_P(\sigma_b)=b(\zeta_P)$.

Then, every character $\chi\in\widehat{\Delta_P}=\operatorname{Hom}(\Delta_P,(\FF^{ac})^\times)$ is a power of $\vartheta_P$; we can write 
 \begin{equation}\label{primitivechar}
 \chi=\vartheta_P^{i},\quad 0\leq i\leq q^d-2.\end{equation} More generally, if $a\in A_+$, we recall (see \S \ref{Carlitztorsion}) that we have defined $K_a=K(\lambda_a)$ and $\Delta_a=\operatorname{Gal}(K_a/K)$.
 If $a$ is non-constant and squarefree, we  can write $a=P_1\cdots P_r$ with $P_1,\ldots,P_r$ distinct primes of respective degrees $d_1,\ldots,d_r$, and 
 we have that $\widehat{\Delta_a}\cong\widehat{\Delta_{P_1}}\times\cdots\times\widehat{\Delta_{P_r}}$, so that
 for every character $\chi\in\widehat{\Delta_a}$, 
 \begin{equation}\label{factdirichletfirstkind}
 \chi=\vartheta_{P_1}^{N_1}\cdots\vartheta_{P_r}^{N_r}\end{equation}
 where the integers $N_i$ are such that $0\leq N_i\leq q^{d_i}-2$ for all $i$. We call a character like $\chi$ in (\ref{factdirichletfirstkind}), a {\em Dirichlet character} (\footnote{In fact we should call such characters {\em tame Dirichlet characters} because the extensions of $K$ associated to such characters are tamely ramified. However, since this is the only  kind of characters to be considered in this  section,  we adopt the terminology {\em Dirichlet characters} for simplicity.}) and its conductor is the product  $\prod_{N_i\not =0}P_i$ (note that the trivial character has conductor $1$). We define $\FF_a$ as the subfield of $\FF^{ac}$ generated over $\FF$ by the roots of $a$ (and we set $\FF_1=\FF$).
 Observe that if $b\in A_+$ ($b$ need not be square-free), any homomorphism $ \Delta_b\rightarrow \mathbb (\FF^{\text{ac}})^\times$ comes from a Dirichlet character.
 
 \subsubsection{Gauss-Thakur sums.} If $P$ is a prime of $A$, the Gauss-Thakur sum associated to a character $\chi=\vartheta_P^{q^j}\in\widehat{\Delta_P}$
 is defined (see Thakur \cite[Section 9.8]{THA2}) by 
 $$g(\vartheta_P^{q^j})=-\sum_{\delta\in\Delta_P}\vartheta_P(\delta^{-1})^{q^j}\delta(\lambda_P)\in\FF_PK_P,$$
where $\FF_PK_P$ is the compositum of $\FF_P$ and $K_P=K(\lambda_P)$ in $K^{ac}$ (we have that $\FF_P\cap K_P=\FF$). Observe that the Gauss-Thakur sums $g(\vartheta_P^{q^j})$ do not depend on the choice of $\iota_P$ although they  depend on the choice  of a $(q-1)$th root of $\theta-\theta^q$ in order to choose $\widetilde{\pi}$ in (\ref{productpi}).
  Let $\chi$ be a character of $\widehat{\Delta_P}$. We define $g(\chi)$ by using (\ref{primitivechar}) in the following way. We expand $i=i_0+i_1q+\cdots+i_{d-1}q^{d-1}$ in base $q$ 
 ($i_0,\ldots,i_r\in\{0,\ldots,q-1\}$), and along with this expansion, we define, 
 $$g(\chi)=\prod_{j=0}^{d-1}g(\vartheta_P^{q^j})^{i_j}.$$
 Observe that in general $g(\chi)$ depends also on the choice of the embedding $\iota_P$ used to define $\vartheta_P.$
 We recall (see \S \ref{omega}) that
 $$\omega(t)=\lambda_{\theta}\prod_{i\geq 0}\left(1-\frac{t}{\theta^{q^i}}\right)^{-1}\in \lambda_{\theta}\mathbb T(K_{\infty}).$$
 By  \cite[Theorem 2.9]{ANG&PEL2},  we have:
 \begin{equation}\label{equagauss}
 g(\vartheta_P^{q^j})=P'(\zeta_P)^{-q^j}\omega(\zeta_P^{q^j}),\quad j=0,\ldots,d-1.
 \end{equation}
 
Now let $\chi$ be a general Dirichlet character whose factorization is of the form  $
 \chi=\vartheta_{P_1}^{N_1}\cdots\vartheta_{P_r}^{N_r}, $ where $P_1,\ldots,P_r$ are distinct primes of respective degrees $d_1,\ldots,d_r$,
 and where the integers $N_i$ are such that $0\leq N_i\leq q^{d_i}-2$ for all $i$. We copy below from \cite[Section 2.3]{ANG&PEL} the definition of the Gauss-Thakur sum associated to such a character. We set $a=P_1\cdots P_r.$ Let us expand in base $q$: 
 \begin{equation}\label{expansionni}
 N_i=\sum_{j=0}^{d_i-1}n_{i,j} q^j,\quad i=1,\ldots,r,\end{equation}
with $n_{i,j} \in \{0, \ldots, q-1\}.$ For a positive integer $N$ we denote by $\ell_q(N)$ the sum of the digits of the expansion in base $q$ of $N$ so that 
$\ell_q(N_i)=\sum_{j=0}^{d_i-1}n_{i,j}$. We also set $s=\sum_i\ell_q(N_i)$. The integer $s$ is called the {\em type of $\chi$}.  We point out that the type $s$ of $\chi$ does not depend of the embeddings $\iota_{P_i}.$
Note that the trivial character is the unique Dirichlet character of type zero.
 The Gauss-Thakur sum associated to $\chi$ is defined as follows:
 $$g(\chi)=\prod_{i=1}^r g(\vartheta_{P_i}^{N_i})\in \FF_aK_a,$$ where $\FF_aK_a$ denotes the compositum  of $\FF_a$ and $K_a=K(\lambda_a)$
 in $K^{ac}$.
\subsubsection{$\FF_a[\Delta_a]$-modules.}\label{modules} 
 For $v$ a place of $K_a$, let us denote by 
 $\widehat{K_{a,v}}$ the completion of $K_a$ at $v$. If $v$ divides $\infty$ and $a\neq 1$, then
 $$\widehat{K_{a,v}}\cong K_\infty(\widetilde{\pi}).$$
 We set $K_{a,\infty}=K_a\otimes_KK_\infty$. We have $K_{1,\infty}=K_\infty$ and if $a$ is non-constant, we have an isomorphism of  $K_\infty$-algebras: 
  $$K_{a,\infty}\cong \prod_{v\in S_\infty(K_a)}K_\infty(\widetilde{\pi}),$$ where 
 $S_\infty(K_a)$ is the set of places of $K_a$ dividing $\infty$.
There is an action of $\Delta_a$ on $K_{a,\infty}$; if $\sigma\in\Delta_a$ and 
 $x\otimes y\in K_a\otimes_KK_\infty$, then:
 $$\sigma(x\otimes y):=\sigma(x)\otimes y.$$ The operator $\tau$ acts on $K_{a,\infty}$ 
 by exponentiation by $q$ (explicitly, $\tau(x\otimes y)=x^q\otimes y^q$) and the actions of $\Delta_a$ and $\tau$
 commute). 
 We set
 $$\Omega_a=K_{a,\infty}\otimes_\FF\FF_a.$$
 We endow $\Omega_a$ with a structure of $\FF_a[\Delta_a]$-module by setting (as above), for $\sigma\in\Delta_a$
and $x\otimes y\in \Omega_a=K_{a,\infty}\otimes_\FF\FF_a$,
$$\sigma(x\otimes y)=\sigma(x)\otimes y.$$ This action commutes with the $\FF_a$-linear 
extension $\varphi$ of the operator $\tau$ on $\Omega_a$. If $x\otimes y\in \Omega_a$, we have
$\varphi(x\otimes y)=\tau(x)\otimes y$. If we identify $\FF_aK_{\infty}$ with $K_{\infty}\otimes _{\FF}\FF_a$ which is in a natural way a  $K_{\infty}$-subalgebra of $\Omega_a,$  if 
$x=\sum_{i}x_i\theta^{-i}\in\FF_a((\theta^{-1}))=\FF_aK_\infty$, then 
$$\varphi(x)=\sum_{i}x_i\theta^{-iq}.$$

 \subsubsection{Idempotents}\label{idempotents} Here we assume that $a$ is square-free. We identify $\FF_aK_a$ with $K_a\otimes_{\FF}\FF_a$ which is a $K$-subalgebra of $\Omega_a$.  We denote by $O_a$ the integral closure of $A$ in $K_a$; then $O_a[\FF_a]=O_a\otimes_{\FF}\FF_a$ is the integral closure of $A[\FF_a]=A\otimes_{\FF}\FF_a$ in $\FF_aK_a$.  
 To a character $\chi\in\widehat{\Delta_a}$ as in (\ref{factdirichletfirstkind}) we associate  an idempotent $e_\chi\in \FF_a[\Delta_a]$, defined as follows:
 $$e_{\chi} =|\Delta_a|^{-1}\sum_{\delta\in \Delta_a}\delta^{-1}\chi(\delta).$$
By  \cite[Lemma 16]{ANG&PEL}, we have:
 $$e_\chi(\FF_a K_a)=(\FF_aK)g(\chi),\quad \text{and}\;e_{\chi}(O_a[\FF_a]) = A[\FF_a]g(\chi).$$
Therefore:
 \begin{equation}\label{image}\Omega_a=\bigoplus_{\chi\in\widehat{\Delta_a}}e_\chi(\Omega_a),\quad
 \text{and }
 e_{\chi}(\Omega_a) =  \FF_aK_{\infty}g(\chi),\quad \chi\in\widehat{\Delta_a}.\end{equation}

\subsubsection{Evaluation map}\label{evaluationmapss} Let $\chi$ be a Dirichlet character of conductor $a=\prod_{i=1}^r P_i,$ $\chi=
\vartheta_{P_1}^{N_1}\cdots\vartheta_{P_r}^{N_r},$ $1\leq N_i\leq q^{d_i}-2,$ where $d_i$ is the degree of $P_i.$ We recall that the type of $\chi$ is $s=\sum_i\ell_q(N_i)$.
Consider an $s$-tuple of variables
 $$\undt{s}=(\underbrace{\underbrace{t_{1,0,1},\ldots,t_{1,0,n_{1,0}}}_{n_{1,0}},\ldots,\underbrace{t_{1,d_0-1,1},\ldots,t_{1,d_0-1,n_{1,d_0-1}}}_{n_{1,d_0-1}}}_{\ell_q(N_1)},\;\;\ldots\;\;\underbrace{\ldots\underbrace{\ldots,t_{r,d_{r}-1,n_{r,d_r-1}}}_{n_{r,d_{r}-1}}}_{\ell_q(N_r)}).$$
 We define the $\CC_\infty$-algebra homomorphism ``evaluation" at $\chi$: $$\operatorname{ev}_\chi:\mathbb{T}_s\rightarrow\CC_\infty$$
 by setting $$\operatorname{ev}_\chi(t_{i,j,k})=\zeta_{P_i}^{q^j}$$ for all $i,j,k$ (with $\zeta_{P_j}=\vartheta_{P_j}(\theta)$).
If we restrict 
 $\operatorname{ev}_\chi$ to the subring of $\mathbb{T}_s$ whose
 elements are symmetric in $t_1,\ldots,t_s$, then $\operatorname{ev}_\chi$ 
 only depends on $\chi$, that is, it does not depend of the choice of the embeddings $\iota_{P_i}$ and it does not depend of the order of indexation of the primes $P_i$.
 Observe that, by \eqref{equagauss}:
\begin{eqnarray}
g(\chi)&=&g(\vartheta_{P_1}^{N_1})\cdots g(\vartheta_{P_r}^{N_r})\nonumber\\
&=&P_1'(\zeta_{P_1})^{-N_1}\cdots P_r'(\zeta_{P_r})^{-N_r}\prod_{i=1}^r\prod_{j=0}^{d_i-1}\omega(\zeta_{P_i}^{q^j})^{n_{i,j}}\nonumber\\
&=&\operatorname{ev}_\chi\left(\prod_{i=1}^r\prod_{j=0}^{d_i-1}\prod_{k=1}^{n_{i,j}}P_{i}'(t_{i,j,k})^{-1}\omega(t_{i,j,k})\right).\label{gchi}
\end{eqnarray}

\subsubsection{An equivariant isomorphism.}\label{484} We recall that we have introduced, in \S \ref{modules},
 a $\FF_a$-linear endomorphism $\varphi$ of $\Omega_a$. We choose a Dirichlet character $$\chi=\vartheta_{P_1}^{N_1}\cdots\vartheta_{P_r}^{N_r}$$
as in (\ref{factdirichletfirstkind}) and we expand the integers $N_i$ in base $q$ as in (\ref{expansionni});
$$N_i=\sum_{j=0}^{d_i-1}n_{i,j}q^j.$$

By (\ref{gchi}) and by the functional equation $\tau(\omega)=(t-\theta)\omega$ of the function $\omega$ of Anderson and Thakur, we see that 
$$\varphi(g(\chi)) =\left( \prod_{i=1} ^{r} \prod_{j=0}^{d_i-1}(\vartheta_{P_i}(\theta)^{q^j}-\theta)^{n_{i,j}}\right)g(\chi)=\left( \prod_{i=1} ^{r} \prod_{j=0}^{d_i-1}(\zeta_{P_i}^{q^j}-\theta)^{n_{i,j}}\right)g(\chi).$$
We now come back to the identity (\ref{image}) and we consider the isomorphism 
$$\nu_\chi:e_\chi(\Omega_a)\rightarrow\FF_aK_\infty$$
defined by $\nu_\chi(y)=yg(\chi)^{-1}$. Then,
$$\nu_\chi(\varphi(x))=\widetilde{\varphi}(\nu_\chi(x)),$$
where 
$$\widetilde{\varphi}(x):=\left( \prod_{i=1} ^{r} \prod_{j=0}^{d_i-1}(\zeta_{P_i}^{q^j}-\theta)^{n_{i,j}}\right)\varphi(x).$$

The Taelman class module associated to the Carlitz module and relative to the extension $K_a/K$ (see \cite{TAE1} and \cite{ANG&TAE}) is defined by
 \begin{equation}\label{taelmanclassmodule}
 H_a=\frac{C(K_a\otimes_KK_{\infty})}{\exp_C(K_a\otimes_KK_{\infty})+C(O_a)}.\end{equation}
 Let $C^\varphi:\FF_a\otimes_\FF A\rightarrow\Omega_a[\varphi]$ be the homomorphism of $\FF_a$-algebras defined by
 $$C^\varphi_\theta=\theta+\varphi.$$ Let us additionally set 
 $$\exp_C^\varphi=\sum_{i\geq 0}D_i^{-1}\varphi^i,$$ which gives rise to a $\FF_a$-linear continuous function $\Omega_a\rightarrow\Omega_a.$

 We then have an isomorphism of $A[\FF_a]$-modules:
 $$ H_a\otimes_{\FF}\FF_a \cong \frac{C^\varphi(\Omega_a)}{\exp_C^\varphi(\Omega_a)+C^\varphi(O_a[\FF_a])}.$$

 Now, we consider, with an analogue meaning of the symbols, the  $\FF_a$-linear endomorphisms $C^{\widetilde{\varphi}}$ and $\exp_C^{\widetilde{\varphi}}$
 of $\FF_aK_\infty$,
and the $A[\FF_a]$-module:
 $$H_\chi=\frac{C^{\widetilde{\varphi}}(\FF_aK_\infty)}{\exp_C^{\widetilde{\varphi}}(\FF_aK_\infty)+C^{\widetilde{\varphi}}(A[\FF_a])}.$$
 The previous discussions  imply the next Lemma.
 \begin{lemma}\label{lemma10}
 The map $\nu_\chi$ induces an isomorphism of $A[\FF_a]$-modules
 $$e_\chi( H_a\otimes_{\FF}\FF_a)\cong H_\chi.$$ 
 \end{lemma}

 \subsubsection{Link between $H_{C_s}$ and $H_\chi$}  
 We recall that we are examining $\chi$ a Dirichlet character of conductor $a=\prod_{i=1}^r P_i,$ $\chi=
\vartheta_{P_1}^{N_1}\cdots\vartheta_{P_r}^{N_r},$ $1\leq N_i\leq q^{d_i}-2,$ where $d_i$ is the degree of $P_i.$ We write $N_i$  in base $q$: $ N_i=\sum_{j=0}^{d_i-1}n_{i,j} q^j,$ with $n_{i,j} \in \{0, \ldots, q-1\}.$

With $s=\sum_{i=1}^r\ell_q(N_i)$ the type of $\chi$  we consider the uniformizable Drinfeld module of rank one $C_s$,
 with parameter 
\begin{equation}\label{alphaparameter}\alpha=\prod_{i=1} ^{r} \prod_{j=0}^{d_i-1}\prod_{k=1}^{n_{i,j}}(t_{i,j,k}-\theta),\end{equation} of degree $s$.
We notice  that $\operatorname{ev}_\chi(\FF[\undt{s}])=\FF_a$ and that we have a field isomorphism
 $\FF[\undt{s}]/I_\chi\cong\FF_a$,
 where $$I_\chi=\operatorname{Ker}(\operatorname{ev}_\chi)\cap\FF[\undt{s}]$$
 is the ideal of $\FF[\undt{s}]$ generated by the polynomials $P_i(t_{i,j,k})$,
 which yields an isomorphism
\begin{equation}\label{isoIchi} 
\frac{A[\undt{s}]}{I_{\chi}A[\undt{s}]}\cong A[\FF_a]=\FF_a[\theta].\end{equation}

\begin{proposition}
 \label{proposition3}
The evaluation map $\operatorname{ev}_{\chi}$ induces an  isomorphism of  $A[\FF_a]$-modules
 $$\psi_{\chi}: \frac{H_{C_s}}{I_{\chi}H_{C_s}}\rightarrow H_{\chi}.$$
 \end{proposition}

 \begin{proof}
The evaluation map $\operatorname{ev}_\chi:\mathbb{T}_s(K_\infty)\rightarrow\FF_a K_\infty$
 satisfies:
 $$\operatorname{ev}_\chi(\tau_\alpha(f))=\widetilde{\varphi}(\operatorname{ev}_\chi(f)),\quad f\in\mathbb{T}_s(K_\infty).$$
In particular, for all $b\in A[\FF_a]$, $\widetilde{b}\in A[\undt{s}]$
such that $b=\operatorname{ev}_\chi(\widetilde{b})$ 
and $f$ in $\TT_s(K_\infty)$, we have (with $\phi={C_s}$):
\begin{equation}\label{firstrem}
C_b^{\widetilde{\varphi}}(\operatorname{ev}_\chi(f))=\operatorname{ev}_\chi(\phi_{\widetilde{b}}(f))\end{equation}
and 
\begin{equation}\label{secrem}\exp_C^{\widetilde{\varphi}}(\operatorname{ev}_\chi(f))=\operatorname{ev}_\chi(\exp_\phi(f)).\end{equation}

We consider the composition $w=\operatorname{pr}\circ\operatorname{ev}_\chi$ of two surjective $\FF$-linear maps:
$$\mathbb{T}_s(K_\infty)\rightarrow\FF_a K_\infty\rightarrow\frac{\FF_a K_\infty}{\exp_C^{\widetilde{\varphi}}(\FF_a K_\infty)+A[\FF_a]},$$
where the first map is $\operatorname{ev}_\chi$ and the second map $\operatorname{pr}$ is the projection.
We deduce from (\ref{firstrem}) that, with $b$ and $\widetilde{b}$ such that
$b=\operatorname{ev}_\chi(\widetilde{b})$, $$w(\phi_{\widetilde{b}}(f))=C_b^{\widetilde{\varphi}}(w(f))$$ for all $f\in\TT_s(K_\infty)$ and $b\in A[\undt{s}]$.

Let $f$ be an element of $\TT_s(K_\infty)$. We have $w(f)=0$ if and only if 
$$\operatorname{ev}_\chi(f)\in\exp_C^{\widetilde{\varphi}}(\FF_aK_\infty)+A[\FF_a].$$ By (\ref{secrem}),
we have $$\exp_C^{\widetilde{\varphi}}(\FF_aK_\infty)+A[\FF_a]=\operatorname{ev}_\chi(\exp_{C_s}(\TT_s(K_\infty))+A[\undt{s}])$$ which means that
$w(f)=0$ if and only if $$f\in I_\chi\TT_s(K_\infty)+\exp_{C_s}(\TT_s(K_\infty))+A[\undt{s}],$$
where 
$$I_\chi\mathbb{T}_s(K_\infty)=\left\{\sum_{i\geq m}x_i\theta^{-i};x_i\in I_\chi,m\in\ZZ\right\}.$$
Now, we notice the isomorphism of $\FF$-vector spaces
$$\frac{H_{C_s}}{I_\chi H_{C_s}}=
\frac{\TT_s(K_\infty)}{I_\chi\TT_s(K_\infty)+\exp_{C_s}(\TT_s)+A[\undt{s}]}$$
 which shows, with (\ref{isoIchi}), that the map of the Proposition is an isomorphism of $A[\FF_a]$-modules.
 \end{proof}

 \begin{corollary}
 \label{theorem6}
Let $\chi$ be a Dirichlet character of type $s\geq q,$ with $s\equiv 1\pmod{q-1}$ and with conductor $a$. Then:
 $$\operatorname{Fitt}_{A[\FF_a]}(H_{\chi} )= \operatorname{ev}_{\chi}(\mathbb{B}_{C_s}) A[\FF_a].$$
 \end{corollary}
\begin{proof} By Theorem \ref{theorem4}, we get:
 $$\operatorname{Fitt}_{A[\FF_a]} \left(\frac{H_{C_s}}{I_{\chi} H_{C_s}}\right)= \operatorname{ev}_{\chi}(\mathbb{B}_{C_s}) A[\FF_a].$$
 We conclude by applying Proposition \ref{proposition3}.
 \end{proof}
 The author of the appendix to the present paper,
Florent Demeslay, informed us that, using ideas similar to that developed in the appendix, he has  obtained an equivariant class number formula for the Carlitz module  similar to Theorem A of \cite{ANG&TAE} for the  extension $K_a/K$ when $a$ is square-free (the case where $a$ is a prime is treated in \cite{ANG&TAE}).  By similar arguments as those used in \cite{ANG&TAE},  he proved that the above  result  can also be  deduced from such an equivariant class number formula. Such an equivariant class number formula  is a special case of a more general result concerning $L$-series of Anderson's $t$-modules defined over Tate algebras  recently obtained by Demeslay and using  ideas developed in this article (this will appear in a forthcoming work of Demeslay).

\subsection{On the structure of $H_{C_s}$ and the isotypic components $H_\chi$}

We consider the set $\mathcal{E}_s$ of all the Dirichlet characters $\chi$ of type $s$, namely, the set of Dirichlet characters $\chi$ which can be written in the form
$$\chi=\vartheta_{P_1}^{N_1}\cdots\vartheta_{P_r}^{N_r},$$
for distinct primes $P_1,\ldots,P_r$ and for integers $N_1,\ldots,N_r$ such that 
$\ell_q(N_1)+\cdots+\ell_q(N_r)=s$ (with $r$ depending on $\chi$). There is a map (depending on the embeddings $\iota_{P_j}$):
$$\mathcal{E}_s\xrightarrow{\mu_s}(\FF^{ac})^s$$ defined by $\chi\mapsto\underline{\zeta}_{\chi}$ 
where $\underline{\zeta}_\chi$ is the $s$-tuple with entries $\zeta_{P_i}^{q^j}$ (see \S \ref{evaluationmapss}).

\begin{Definition}\label{almostall}{\em 
Let $\mathcal{P}$ be a property over $\mathcal{E}_s$.
We say that $\mathcal{P}$ holds for {\em almost all characters of type $s$}
if $$\mu_s(\{\chi;\mathcal{P}(\chi)\text{ holds}\})\supset\mathcal{O},$$ where $\mathcal{O}$ is a non-empty Zariski-open 
subset of $(\FF^{ac})^s$.}
\end{Definition}

\begin{Remark}{\em (1). Let us suppose that $s\geq 1.$ Observe that if a property $\mathcal P$ is true for almost all Dirichlet characters of type $s,$ then it is true for infinitely many Dirichlet characters of type $s.$ Indeed, there exist infinitely many $s$-tuples of  primes $(P_1, \ldots, P_s)$ with $P_i\not = P_j $ for $i\not = j,$ such that, given $F\in\FF[\undt{s}]$, $F(\zeta_{P_1}, \ldots, \zeta_{P_s})\not = 0,$ and   $\chi=\vartheta_{P_1}\cdots \vartheta_{P_s}$ is of type $s.$

(2). Definition \ref{almostall} does not depend on the choice of the embeddings $\iota_{P}$ for $P$ prime. Indeed, if $\mu_s(\{\chi;\mathcal{P}(\chi)\text{ holds}\})\supset\mathcal{O},$ we can always find  a non-empty Zariski-open  $\mathcal{O}' \subset \mathcal{O}$ such that $\mu_s(\{\chi;\mathcal{P}(\chi)\text{ holds}\})\supset\mathcal{O}'$ for any choice of the embeddings $\iota_{P}$ for $P$ prime.

(3). There are such properties $\mathcal{P}$ which hold for almost all Dirichlet characters of type $s$ but which fail for infinitely many such characters. This of course depends on the fact that there are strictly closed subsets of $(\FF^{ac})^s$ which are infinite. For example, one can consider the property over $\mathcal{E}_2$ determined by the non-vanishing on $(\FF^{ac})^2$ of the polynomial $t_1^q-t_2$, or the polynomial $(t_1^q-t_2)(t_2^q-t_1)$ (in the last case, the property does not depend on the embeddings $\iota_P$).}
\end{Remark}

Let $R$ be a commutative ring and $M$ an $R$-module. Let $f$ be an element of $R$. We denote  the $f$-torsion submodule of $M$ by $M[f]=\{ m\in M, f.m=0\}.$

\begin{proposition}
 \label{propo7}
 Let us suppose that $s\geq 1$ and let $f$ be in $A_+$. Then, for almost all Dirichlet characters of type $s$, we have
 $H_{\chi}[f]=\{0\}$.
\end{proposition}
\begin{proof} We can suppose that $s\geq 2q-1$.
 By Proposition \ref{propo7puntodue}, if $s\equiv1\pmod{q-1}$, $H_{C_s}[f]=\{0\}$. If $s\not\equiv1\pmod{q-1}$, then, by Proposition \ref{propM}, $H_{C_s}[f]\otimes_{\FF[\undt{s}]}\FF(\undt{s})=\{0\}$. Hence, in all cases, 
 $$H_{C_s}[f]\otimes_{\FF[\undt{s}]}\FF(\undt{s})=\{0\}.$$
 We now consider the exact sequence of $\FF[\undt{s}]$-modules of 
 finite type (the middle map is the multiplication by $f$):
 $$0\rightarrow H_{C_s}[f]\rightarrow H_{C_s}\rightarrow H_{C_s}\rightarrow \frac{H_{C_s}}{fH_{C_s}}\rightarrow0.$$ Taking the tensor product of the above exact sequence
 with $\FF(\undt{s}),$ we get:
   $$\frac{H_{C_s}}{fH_{C_s}}\otimes_{\FF[\undt{s}]}\FF(\undt{s})=\{0\}.$$
 In particular, the $\FF[\undt{s}]$-module of finite type $\frac{H_{C_s}}{fH_{C_s}}$ is 
 torsion and there exists $F_f\in\FF[\undt{s}]\setminus\{0\}$ such that
 $F_fH_{C_s}\subset fH_{C_s}$. It remains to apply Proposition \ref{proposition3}.
 \end{proof}
We now suppose that $s\geq 1$ is such that $s\not \equiv 1\pmod{q-1}$.  Recall that, by Remark \ref{floric}, the element $$m_s=[\calh_{C_s}]_{\FF(\undt{s})[\theta]}$$
belongs to $A[\undt{s}]$ and is monic as a polynomial in $\theta$. 

\begin{theorem}\label{theorem7}
For almost all characters $\chi$ of type $s$, we have that
$$\operatorname{Fitt}_{A[\FF_a]} (H_{\chi} )= \operatorname{ev}_{\chi}(m_s)A[\FF_a],$$
where $a$ is the conductor of $\chi$.
 \end{theorem}
 \begin{proof} 
 Let us denote by $M$ the torsion sub-$\FF[\undt{s}]$-module of $H_{C_s}$: $$M=\{m\in H_{C_s} ; \exists f\in  \FF[\undt{s}], f.m=0\}.$$
 We set:
 $$\widetilde{H_{C_s}}:=\frac{H_{C_s}}{M}$$
 Note that the $\FF(\undt{s})[\theta]$-module
 $\widetilde{ H_{C_s}}\otimes_{\FF[\undt{s}]}\FF(\undt{s})$ is isomorphic to $\calh_{C_s}$ and therefore its Fitting ideal over $\FF(\undt{s})[\theta]$ is  generated by $m_s$.
Since $M$ is a finitely generated and torsion $\FF[\undt{s}]$-module, we notice that for almost all Dirichlet characters $\chi$ of type $s$, we have that
 $M\subset I_{\chi} H_{C_s}.$ Thus,  for almost all the Dirichlet characters $\chi$ of type $s$, the $A[\undt{s}]/I\chi A[\undt{s}]$-module $$\frac{\widetilde{H_{C_s}}}{I_\chi\widetilde{H_{C_s}}}$$ is well defined.
For almost all Dirichlet characters $\chi$ of type $s$, we have that
 $$\operatorname{Fitt}_{\frac{A[\undt{s}]}{I_\chi A[\undt{s}]}}\left(\frac{\widetilde{H_{C_s}}}{I_{\chi}\widetilde{H_{C_s}}} \right)=
 \frac{\operatorname{Fitt}_{A[\undt{s}]} (\widetilde{ H_{C_s}}) +I_{\chi} A[\undt{s}]}{I_{\chi} A[\undt{s}]}.$$
But by Proposition \ref{proposition3}, for almost all characters $\chi$ of type $s,$ we also have an isomorphism of $A[\FF_a]$-modules
 $\frac{\widetilde{H_{C_s}}}{I_{\chi}\widetilde{H_{C_s}}}\cong H_{\chi}$. The Theorem follows easily.
 
 \end{proof}

\subsection{Pseudo-null and pseudo-cyclic modules}\label{pseudo}

\begin{Definition}{\em Let $M$ be a finitely generated $A[\undt{s}]$-module which also is finitely generated as a $\FF[\undt{s}]$-module. We  say that $M$ is {\em pseudo-null} if $$M\otimes_{\FF[\undt{s}]}\FF(\undt{s})=\{ 0\}.$$ We say that $M$ is {\em pseudo-cyclic} if there exists $m\in M$ such that $\frac{M}{mA[\undt{s}]}$ is pseudo-null. }
\end{Definition}

In this Section, we investigate the property of pseudo-cyclicity and pseudo-nullity for the modules $H_{C_s}$ (that is, in the case of the Drinfeld module ${C_s}$ of parameter $\alpha=(t_1-\theta)\cdots(t_s-\theta)$).
We are concerned by the following questions which we leave open.
\begin{question}
  \label{question1} Is $H_{C_s}$ pseudo-cyclic?
  \end{question}
  \begin{question}
  \label{question2} Assuming that $s\not \equiv 1\pmod{q-1},$ is $H_{C_s}$ pseudo-null?
  \end{question}

\subsubsection{The  torsion case}

 We are here in the case $s\equiv 1\pmod{q-1}.$ Recall that $H_{C_s}=\{0\}$ for $s=1,q.$ We can suppose without loss of generality, that $s\geq 2q-1$.
\begin{proposition}
 \label{proposition4even} The following conditions are equivalent.
\begin{enumerate}
\item $H_{C_s}$ is pseudo-cyclic,
\item For almost all Dirichlet characters $\chi$ of type $s$, $H_{\chi}$ is a cyclic $\FF_a[\theta]$-module, where $a$ is the conductor of $\chi$,
\item There exists a Dirichlet character $\chi$ of type $s$ such that $H_{\chi}$ is a cyclic $\FF_a[\theta]$-module, where $a$ is the conductor of $\chi$.
\end{enumerate}
\end{proposition}
\begin{proof}
The first condition implies the second by means of the equivariant isomorphism of the Proposition \ref{proposition3}. The second condition
obviously implies the third. It remains to show that the third condition implies the first. Write $\ring{}$ for $\FF(\undt{s})[\theta]$ recall that $\calh_{C_s}=H_{C_s}\otimes_{\FF[\undt{s}]}\FF(\undt{s})$. Since $\calh_{C_s}$ is finitely generated and torsion over $\ring{}$, it is isomorphic to $\prod_{i=1}^n\ring{}/r_i\ring{}$ for some $r_1, \dots, r_n$ monic in $\ring{}$. 
Since $\operatorname{Fitt}_{\ring{}}(\calh_{C_s})$ is generated by $\prod_{i=1}^n r_i$ and $\operatorname{Ann}_{\ring{}}(\calh_{C_s})$ by 
the least common multiple of $r_1, \dots, r_n$, the condition that $H_{C_s}$ is pseudo-cyclic is equivalent to
$$\operatorname{Fitt}_{\ring{}}(\calh_{C_s})=\operatorname{Ann}_{\ring{}}(\calh_{C_s}).$$
 By Theorem \ref{theorem4}, the polynomial $\mathbb{B}_{C_s}$ is the monic generator 
of $\operatorname{Fitt}_{A[\undt{s}]}(H_{C_s})$ and $\operatorname{Fitt}_{\ring{}}(\calh_{C_s})=\mathbb{B}_{C_s} \ring{}$. Let $m$ be the monic generator of $\operatorname{Ann}_{\ring{}}(\calh_{C_s}),$ observe that $m\in A[\undt{s}]$ and   
$m$ divides $\mathbb{B}_{C_s}$. By Corollary \ref{theorem6}, 
$$\operatorname{Ann}_{A[\FF_a]}(H_\chi)=\operatorname{ev}_\chi(\mathbb{B}_{C_s})A[\FF_a].$$
Thus, $\operatorname{ev}_\chi(\mathbb{B}_{C_s})$ divides $\operatorname{ev}_\chi(m)$ and 
therefore $\mathbb{B}_{C_s}=m$.
\end{proof}

\subsubsection{The non-torsion case}

Here we suppose that $s\not \equiv 1\pmod{q-1}$ and $s\geq 1$. We recall that
$\ring{}=\FF(\undt{s})[\theta]$.
\begin{proposition}
 \label{proposition4odd}
The following assertions are equivalent.
\begin{enumerate}
\item $H_{C_s}$ is pseudo-cyclic,
\item $\operatorname{Ann}_{\ring{}}\left(\frac{ U_{C_s}}{U^c_{C_s}}\otimes_{\FF[\undt{s}]}\FF(\undt{s})\right) =\operatorname{Ann}_{\ring{}} (\calh_{C_s})$,
\item For almost all Dirichlet characters $\chi$ of type $s$, the $A[\FF_a]$-module $H_{\chi}$ is a cyclic module, where $a$ is the conductor of $\chi$.
\end{enumerate}
\end{proposition}
\begin{proof}
We know that $U_{C_s}/U^c_{C_s}$ is a cyclic $A[\undt{s}]$-module and 
we already know that it is of finite rank over $\FF[\undt{s}]$ and free (see Remark \ref{floric}). By Corollary  \ref{theorem41},
$$\operatorname{Ann}_{\ring{}}\left(\frac{ U_{C_s}}{U^c_{C_s}}\otimes_{\FF[\undt{s}]}\FF(\undt{s})\right)=\operatorname{Fitt}_{\ring{}} \left(\frac{ U_{C_s}}{U^c_{C_s}}\otimes_{\FF[\undt{s}]}\FF(\undt{s})\right)=
\operatorname{Fitt}_{\ring{}}(\calh_{C_s}).$$ 
Then, $H_{C_s}$ is pseudo-cyclic if and only if
 $\operatorname{Fitt}_{\ring{}}(\calh_{C_s})=\operatorname{Ann}_{\ring{}}(\calh_{C_s})$.
 This implies the equivalence of the first condition and the second condition. That these conditions are also equivalent to the third condition follows in a way which is very similar to that used in the proof of Proposition 
\ref{proposition4even}.
\end{proof}
\begin{Remark}{\em 
We notice that $H_{C_s}$ is pseudo-null if and only if
 $$\operatorname{Fitt}_{\ring{}}(\calh_{C_s})=\ring{}.$$
 Thus $H_{C_s}$ is pseudo-null if and only if $ U_{C_s}=U^c_{C_s}.$ Moreover, $H_{C_s}$ is pseudo-null if and only if $m_s=[\calh_{C_s}]_{\ring{}}=1$ which is equivalent, by  Theorem \ref{theorem7}, to the fact that, for almost all Dirichlet characters $\chi$ of type $s$ we have $H_{\chi}=\{0\}.$}\end{Remark}

 \subsection{Evaluation of the polynomials $\mathbb{B}_{C_s}$ }\label{evaluationpolynomials}

For $s\geq 1$ with $s\equiv 1\pmod{q-1},$ to simplify our notation, we write $\mathbb{B}_s$ instead of $\mathbb{B}_{C_s}.$  Note that $\mathbb B_1=\frac{1}{\theta-t}.$

Let $\chi$ be a character of conductor $a$ and of type $s,$ write  $a=P_1\cdots P_r$ for distinct primes $P_1,\ldots,P_r\in A$, so that $\chi=\vartheta_{P_1}^{N_1}\cdots\vartheta_{P_r}^{N_r}$ with $N_i\leq q^{d_i}-2$, $d_i$ being the 
degree in $\theta$ of $P_i$ for all $i$.

We recall that the special value at $n\geq 1$ of Goss-Dirichlet $L$-series (see \cite{GOS}, chapter 8)  associated to $\chi$ is defined by
  $$ \forall n\in \mathbb Z,\, L(n,\chi) = \sum_{m\geq 0}\,\, \sideset{}{{}^{'}}\sum_{b\in A_{+,m}}\chi(\sigma_b)b^{-n} \in \mathbb C_{\infty},$$ where the sum runs over the elements $b$ which are relatively prime to $a$. 

In \cite{ANG&PEL} and \cite{PEL2} it is shown that  these $L$-series values can be 
  obtained from the evaluation of $L$-series values $L(n,C_s)$. More precisely, for all $b$ relatively prime to  $a$, 
  $$\chi(\sigma_b)=\operatorname{ev}_\chi(\rho_\alpha(b))$$ ($\alpha$ being the parameter of $C_s$), and therefore we get $$L(n,\chi)=\operatorname{ev}_\chi(L(n,C_s)).$$ 
  
We choose $N\in\mathbb N$ and we expand it in base $q$ as 
$$N=\sum_{j=0}^kn_jq^j$$ ($n_0,\ldots,n_k\in\{0,\ldots,q-1\}$). We further set $$s'=s+\ell_q(N).$$
We then have the evaluation map (we recall that $\mathbb{E}_s$ is the sub-algebra of $\mathbb{T}_s$
of entire functions \S \ref{entirefunctions}).
$$\operatorname{ev}_{N}:\mathbb{E}_{s'}\rightarrow\mathbb{E}_{s}$$
defined by replacing the family of variables $(t_1,\ldots,t_s,t_{s+1},\ldots,t_{s+\ell_q(N)})$ by 
$$(t_1,\ldots,t_s,\underbrace{\theta,\ldots,\theta}_{n_0},\underbrace{\theta^q,\ldots,\theta^q}_{n_1},\ldots,\underbrace{\theta^{q^k},\ldots,\theta^{q^k}}_{n_k}).$$ 
If $N=0$, this map is obviously the identity map of $\mathbb{E}_s$. 
We shall work, in this subsection, with the evaluation map $\operatorname{ev}_{\chi,N}:\mathbb{E}_{s'}\rightarrow\CC_\infty$
defined by
$$\operatorname{ev}_{N,\chi}=\operatorname{ev}_{\chi}\circ\operatorname{ev}_{N}.$$
In particular:
  $$\operatorname{ev}_{N, \chi}(A[\undt{s+\ell_q(N)}]) =A[\FF_a].$$
If ${C_{s'}}$ is the Drinfeld module of rank one of parameter $\alpha=(t_1-\theta)\cdots(t_{s'}-\theta)$, 
then this evaluation map allows to obtain the special values of the Dirichlet $L$-series of Goss from $L(1,C_{s'})\in \mathbb E_{s'}$ (\cite[Corollary 8]{ANG&PEL}). Indeed, for all $j\in \mathbb N$,
$$\tau^j (L(1,C_{s'}))=\sum_{n\geq 0}\sum_{a\in A_{+,n}} \frac{a(t_1)\cdots a(t_{s'})}{a^{q^j}},$$
thus:
$$\operatorname{ev}_{N}(\tau^j (L(1,C_{s'})))=\sum_{n\geq 0}\sum_{a\in A_{+,n}} \frac{a(t_1)\cdots a(t_{s})}{a^{q^j-N}},$$
and therefore:
$$\operatorname{ev}_{N,\chi}(\tau^j (L(1,C_{s'})))=L(q^j-N,\chi).$$
To $N$ as above, with its expansion $N=\sum_in_iq^i$ in base $q$, we associate the {\em Carlitz factorial} $\Pi(N),$ defined by
  $$\Pi(N)=\prod_{i\geq 0} D_i^{n_i}.$$ We apply the evaluations $\operatorname{ev}_{N,\chi}$ in two different ways.

\subsubsection{First way to apply $\operatorname{ev}_{N,\chi}$}  Recall that $s'=s+\ell_{q}(N)$ and we assume here that $s'\equiv 1\pmod{q-1},$ $s'\geq 1.$ Furthermore if $s=0$ and $\ell_{q}(N)=1,$ we assume that $N\geq 2.$ For a polynomial $a\in A$, we denote by $a'$ its derivative
in the indeterminate $\theta$.
Observe that $\operatorname{ev}_N(\mathbb{B}_{s'})$ is  well defined. By \cite[Corollary 8]{ANG&PEL}, the 
function $$L(1,C_{s'})=\sum_{d\geq 0}\sum_{b\in A_{+,d}}\chi_{t_1}(b)\cdots\chi_{t_{s'}}(b)b^{-1}$$ is 
in $\mathbb{E}_{s'}$, that is, entire in the set of variables $\undt{s'}$ (it is denoted by $L(\chi_{t_1}\cdots\chi_{t_{s'}},1)$ in \cite{ANG&PEL}). Write $N=\sum_{i=0}^k n_i q^i,$ $n_i\in \{ 0, \ldots, q-1\},$ $n_k\not = 0.$ We rename the variables $t_{s+1},\ldots,t_{s'}$
in a way which is compatible with the expansion of $N$ in base $q$ by writing:
$$(t_{s+1},\ldots,t_{s'})=(t_{0,1},\ldots,t_{0,n_0},\ldots,t_{k,1},\ldots,t_{k,n_k}).$$ 
Note that $\operatorname{ev}_N\left(\omega(t_1)\cdots\omega(t_{s'})\sum_{d\geq0}\sum_{b\in A_{+,d}}b^{-1}b(t_1)\cdots b(t_{s'})\right)$ is equal to:
\begin{eqnarray*}
\lefteqn{\operatorname{ev}_N\left(\omega(t_1)\cdots\omega(t_{s})(\prod_{i=0}^k\prod_{j=1}^{n_i}((t_{i,j}-\theta^{q^i})\omega(t_{i,j})))\times\right.}\\
& \times & \left. 
\frac{1}{\prod_{i=0}^k\prod_{j=1}^{n_i}(t_{i,j}-\theta^{q^i})}(\sum_{d\geq0}\sum_{b\in A_{+,d}}b^{-1}b(t_1)\cdots b(t_{s'}))\right).\end{eqnarray*}
By \cite[Lemma 5]{ANG&PEL}, $L(1,C_{s'})$  vanishes 
at any point of the form
$$(t_1,\ldots,t_s,t_{s+1},\ldots,t_{s+j-1},\theta^{q^l},t_{s+j+1},\ldots,t_{s'}),\quad l\geq 0,\quad t_i\in\CC_\infty.$$
Furthermore an easy computation shows that the function
$\omega(t_k)=\lambda_{\theta}\prod_{i\geq 0}(1-\frac{t_k}{\theta^{q^i}})^{-1}$ has a simple pole at $\theta^{q^l}$ of residue $-\widetilde{\pi}^{q^l}D_l^{-1}$ (for all 
$l\geq 0$). Indeed, in $\TT$, $\omega(t)=\exp_C\left(\frac{\widetilde{\pi}}{\theta-t}\right)=\sum_{i\geq 0}D_i^{-1}\frac{\widetilde{\pi}^{q^i}}{\theta^{q^i}-t}$. Thanks to this residue computation
we deduce,
with $\boldsymbol{\Delta}$ the differential operator
$$\boldsymbol{\Delta}=\frac{\partial}{\partial t_{0,1}}\cdots\frac{\partial}{\partial t_{0,n_0}}\cdots\frac{\partial}{\partial t_{k,1}}\cdots\frac{\partial}{\partial t_{k,n_k}},$$
\begin{eqnarray*}
 \lefteqn{\operatorname{ev}_N\left(\omega(t_1)\cdots\omega(t_{s'})\sum_{d\geq0}\sum_{b\in A_{+,d}}b^{-1}b(t_1)\cdots b(t_{s'})\right)}\\
 &=& \omega(t_1)\cdots\omega(t_s)(-\widetilde{\pi})^{n_0}\cdots(-\widetilde{\pi})^{n_kq^k}D_0^{-n_0}\cdots D_k^{-n_k}\times\\
& &\times\left[\boldsymbol{\Delta}\left(\sum_{d\geq0}\sum_{b\in A_{+,d}}b^{-1}b(t_1)\cdots b(t_{s})b(t_{0,1})\cdots b(t_{k,n_k})\right)\right]_{t_{i,j}=\theta^{q^i}}\\
\end{eqnarray*}
\begin{eqnarray*}
&=&(-1)^N\omega(t_1)\cdots\omega(t_s)\widetilde{\pi}^N\Pi(N)^{-1}\times\\
& &\times \sum_{d\geq0}\sum_{b\in A_{+,d}}b^{-1}b(t_1)\cdots b(t_s)[b'(t_{0,1})\cdots b'(t_{k,n_k})]_{t_{i,j}=\theta^{q^i}}\\
&=&(-1)^N\omega(t_1)\cdots\omega(t_s)\widetilde{\pi}^N\Pi(N)^{-1}\sum_{d\geq 0}\sum_{b\in A_{+,d}}b^{-1}b(t_1)\cdots b(t_s)
(b')^N.
\end{eqnarray*}
Recall that:
$$\mathbb B_{s'}=(-1)^{(s'-1)/(q-1)}L(1, C_{s'})\omega(t_1)\cdots \omega(t_{s'})\widetilde{\pi}^{-1}$$  (see \eqref{Bphi}).
Therefore, we get  the formula:
$$\operatorname{ev}_N(\mathbb{B}_{s'})=(-1)^{N+\frac{s'-1}{q-1}}\widetilde{\pi}^{N-1}\Pi(N)^{-1}\omega(t_1)\cdots\omega(t_s)\sum_{d\geq 0}\sum_{b\in A_{+,d}}\rho_{\alpha}(b)\frac{(b')^N}{b},$$
where we notice that the series in the right-hand side, with $\alpha$ the parameter of $C_s$, is convergent for the Gauss norm of $\mathbb{T}_s$.

We also have the formula  (use (\ref{gchi})):
$$\operatorname{ev}_\chi(\omega(t_1)\cdots\omega(t_s))=\vartheta_{P_1}(\sigma_{P'_1})^{N_1}\cdots\vartheta_{P_r}(\sigma_{P'_r})^{N_r}g(\chi)=
P_1'(\zeta_1)^{N_1}\cdots P_r'(\zeta_r)^{N_r}g(\chi)$$ where  $\zeta_i= \zeta_{P_i}$.
Therefore, 
composing with $\operatorname{ev}_\chi$ now gives the identity
\begin{eqnarray*}\lefteqn{
\operatorname{ev}_{N,\chi}(\mathbb{B}_{s'})=}\\ &=&(-1)^{N+\frac{s'-1}{q-1}}\widetilde{\pi}^{N-1}\Pi(N)^{-1}\vartheta_{P_1}(\sigma_{P'_1})^{N_1}\cdots\vartheta_{P_r}(\sigma_{P'_r})^{N_r}g(\chi)\times\\ & &\sum_{d\geq 0}\sum_{b\in A_{+,d}}\chi(b)\frac{(b')^N}{b}.\end{eqnarray*}

\subsubsection{Second way to apply $\operatorname{ev}_{N,\chi}$}
Here again $s'\equiv {1}\pmod{q-1},$ $s'\geq 1.$  Let us consider an integer $d\geq 1$ such that $q^d>N$. The functions $\tau^d(L(1,C_{s'}))=L(q^d,C_{s'})$ are also entire and we have
$$\tau^d(\mathbb{B}_{s'})=(-1)^{\frac{s'-1}{q-1}}\widetilde{\pi}^{-q^d}L(q^d,C_{s'})b_d(t_1)\cdots b_d(t_{s'})\omega(t_1)\cdots\omega(t_{s'}),$$
where  $b_i=(t-\theta)(t-\theta^q)\cdots(t-\theta^{q^{i-1}})$ for $i>0$ and $b_0=1$.
We observe, as in \cite[\S 3.2]{ANG&PEL}, that $$\operatorname{ev}_N(b_d(t_{s+1})\cdots b_d(t_{s'})\omega(t_{s+1})\cdots\omega(t_{s'}))=(-1)^N\widetilde{\pi}^{N}\prod_{i=0}^kl_{d-i-1}^{n_iq^i}.$$
Again by (\ref{gchi}) we have
\begin{eqnarray*}
\lefteqn{\operatorname{ev}_\chi(\omega(t_1)\cdots \omega(t_{s}))=}\\ &=&L(q^d-N,\chi)\vartheta_{P_1}(\sigma_{P'_1})^{N_1}\cdots\vartheta_{P_r}(\sigma_{P'_r})^{N_r}g(\chi).\end{eqnarray*}
Moreover,
$$\operatorname{ev}_{\chi,N}(L(q^d,C_{s'}))=L(q^d-N,\chi).$$
Hence, we obtain the formula
\begin{eqnarray*}\lefteqn{
\operatorname{ev}_{N,\chi}(\tau^d(\mathbb{B}_{s'}))=}\\ &=&(-1)^{N+\frac{s'-1}{q-1}}\widetilde{\pi}^{N-q^d}\vartheta_{P_1}(\sigma_{P'_1})^{N_1}\cdots\vartheta_{P_r}(\sigma_{P'_r})^{N_r}g(\chi)\times\\ & &L(q^d-N,\chi)\operatorname{ev}_{\chi}(b_d(t_1)\cdots b_d(t_s))\prod_{i=0}^kl_{d-i-1}^{n_iq^i}.\end{eqnarray*}

We  set
$$\rho_{N,\chi,d}:=\frac{\operatorname{ev}_{N,\chi}(\tau^d(\mathbb{B}_{s'}))}{\vartheta_{P_1}^{N_1}(\sigma_{P_1'})\cdots \vartheta_{P_r}^{N_r}(\sigma_{P_r'})}\in\FF_aK.$$

  \begin{proposition}
  \label{proposition5even}
Let  $s'=s+\ell_q(N)\equiv1\pmod{q-1}$, $s'\geq 1.$ The following properties hold.
\begin{enumerate}
\item If $s=0$ and $\ell_{q}(N)=1,$ we assume that $N\geq 2.$ We have:
$$\rho_{N,\chi,0}=(-1)^{N+\frac{s'-1}{q-1}}\widetilde{\pi}^{N-1}g(\chi)\Pi(N)^{-1}\sideset{}{{}^{'}}\sum_{d\geq 0}\sum_{b\in A_{+,d}}\chi(b)b'{}^Nb^{-1}.$$
\item Let $d\geq 1$ be an integer such that $q^d>N$. Then,
$$\rho_{N,\chi,d}=(-1)^{N+\frac{s'-1}{q-1}}g(\chi) L(q^d-N, \chi)\widetilde{\pi}^{N-q^d}\operatorname{ev}_{\chi}(b_d(t_1)\cdots b_d(t_s)) \prod_{i=0}^kl_{d-1-i}^{n_iq ^i}.$$
\end{enumerate}
\end{proposition}

\subsection{A refinement of Herbrand-Ribet-Taelman Theorem}\label{bernoullicarlitzgen}
As in the previous Section, let $\chi$ be a Dirichlet character of type $s\geq 0$ and conductor $a$. 
Following \cite[\S 2.4]{ANG&PEL}, we introduce the {\em generalized Bernoulli-Carlitz numbers} $\operatorname{BC}_{n, \chi^{-1}}$ by means of the following generating series:
  $$\frac{g(\chi)}{a}\sum_{b\in (A/aA)^\times}\frac{\chi(b) X}{\exp_C(\frac{X}{a})-\sigma_b(\lambda_a)} =
  \sum_{i\geq 0} \frac{\operatorname{BC}_{i, \chi^{-1}}}{\Pi(i)} X^i.$$
  If $s=0,$ when $a=1$, we set $\lambda_a=0$ in the above formula so that we get in this case
 $$\operatorname{BC}_{i, \chi^{-1}}= \operatorname{BC}_{i}$$ for $i\geq 0$, where for $n\in \mathbb N,$  $\operatorname{BC}_n$ denotes the $n$-th Bernoulli-Carlitz number (see \cite[Chapter 9, \S 9.2]{GOS}). 
 
 From \cite[Proposition 17]{ANG&PEL}, we deduce easily the following:
 \begin{lemma} The following properties hold:
 \begin{enumerate}
 \item For all $i\geq 0$, we have $\operatorname{BC}_{i, \chi^{-1}}\in \FF_aK$.
 \item If $i\not \equiv s\pmod{q-1}$, then $\operatorname{BC}_{i, \chi^{-1}}= 0$.
 \item We have $ \operatorname{BC}_{0, \chi^{-1}}=0$ if $s\geq 1$.
 \item If $i\geq 1$ is such that $i\equiv s\pmod{q-1}$, then
  $$L(i, \chi)g(\chi)=\widetilde{\pi}^i\operatorname{BC}_{i, \chi^{-1}}\Pi(i)^{-1}.$$
  \end{enumerate}
  \end{lemma}

  We now consider integers $s,s',N$ with $s+\ell_q(N)=s'$, $s\geq 1$, $s'\equiv1\pmod{q-1}$ and a character $\chi$ 
  of type $s'$ and conductor $a=Pb$ such that $$\chi=\vartheta_P^N\widetilde{\chi},$$ with $P$ a prime not dividing the conductor $b$ of 
  $\widetilde{\chi}$, and such that $N\leq q^d-2$, $d$ being the degree of $P$. 
 The valuation ring of the compositum $\FF_a\KPhat$ of $\FF_a$ and $\KPhat$  in $\CC_P$ is
the ring $\APhat[\FF_a]$ where  $\APhat$ is the valuation ring of $\KPhat$.

We highlight that  the congruences for the above generalized Bernoulli-Carlitz numbers that will be used in the proof of Theorem \ref{theorem8}  are well defined thanks to the choice of the embedding of $K^{ac}$ in $\mathbb{C}_P$ that we made at the 
beginning of \S \ref{somesettings}.
 
 \subsubsection{An example}\label{casesequalone} We consider the simplest non-trivial case of a character $\chi$ of type $s'=1$.
 Here, the factor $\widetilde{\chi}$ is  the trivial character  and therefore $s=0$ and $N=q^i$ so that $\chi=\vartheta_P^{q^i}$ with $d-1\geq i\geq 0$. The case $i=0$ is somewhat exceptional
 so that we assume that $i>0$. We have $H_{C_1}=\{0\}$ which implies, by Proposition \ref{proposition3}, the triviality of $H_{\chi^{q^j}},$ $j\in \mathbb N.$ By Lemma \ref{lemma10}, $e_{\chi^{q^j}}(H_a\otimes_{\FF}\FF_a)=0$ for all $j\in \mathbb N,$  and therefore   
 $e_\chi(H_a\otimes_{A}\APhat[\FF_a])=\{0\}$. Now we observe, by the second part of the Proposition \ref{proposition5even}, that
 $$\frac{\operatorname{BC}_{q^d-q^i}l_{d-1-i}^{q^i}}{\Pi(q^d -q^i)}=-\frac{1}{\theta^{q^d}-\theta ^{q^i}}.$$ 
 In particular, the Bernoulli-Carlitz number $\operatorname{BC}_{q^d-q^i}$ is in this case $P$-integral and 
 reduces to a unit modulo $P.$

  \begin{theorem}[Refinement of Herbrand-Ribet-Taelman Theorem \cite{TAE3}]
  \label{theorem8} Let $\chi$ be a Dirichlet character with conductor $a$ and type $s'\equiv 1\pmod{q-1}, s'\geq 1$. Let $P$ be a prime 
  dividing $a$, of degree $d$, and let us write $\chi=\vartheta_P^N\widetilde{\chi}$ with $1\leq N\leq q^d-2$ and
with  $\widetilde{\chi}$ a Dirichlet character of conductor prime to $P$. We further suppose that if $s'=1$, then $N$ is at least $2$.
  The generalized Bernoulli-Carlitz number $\operatorname{BC}_{q^d-N, \widetilde{\chi}^{-1}}$ is $P$-integral. Furthermore, 
  $$e_\chi(H_a\otimes_A\APhat[\FF_a])\not =\{0\}$$ if and only if  $$\operatorname{BC}_{q^d-N, \widetilde{\chi}^{-1}}\equiv 0\pmod {P}.$$
  \end{theorem}
  \begin{proof} We have already considered the case $s'=1$ in \S \ref{casesequalone} so we may now suppose that $s'\geq 2$.
  We work in $\FF_a\KPhat.$ Note that the Dirichlet  character $\widetilde {\chi}$  is of type $s\geq 0$ and that we have the congruence
  $\tau^d(\mathbb{B}_{s'})\equiv \mathbb{B}_{s'}\pmod{P}$.
  Since obviously,
  $\operatorname{ev}_{N,\widetilde{\chi}}(\mathbb{B}_{s'})\equiv \operatorname{ev}_{\chi} (\mathbb{B}_{s'})\pmod {P}$,
  we have that
  $$e_{N, \widetilde{\chi}}(\tau^d (\mathbb{B}_{s'}))\equiv \operatorname{ev}_{\chi} (\mathbb{B}_{s'})\pmod {P}.$$
  Let us write now $\widetilde{\chi}=\vartheta_{P_1}^{N_1}\cdots \vartheta_{P_r}^{N_r},$ where $b=P_1\cdots P_r$ is the conductor of $\widetilde{\chi}$ (we recall that $N=\sum_{i=0}^{d-1} n_i q ^i, $ $n_i\in \{ 0, \ldots, q-1\}$). By the second part of the Proposition \ref{proposition5even}, we have:
  $$\rho_{N,\widetilde{\chi},d}= (-1)^{N+\frac{s'-1}{q-1}}\Pi(q^d-N)^{-1}\operatorname{BC}_{q^d-N, \widetilde{\chi}^{-1}}\operatorname{ev}_{\widetilde{\chi}}(b_d(t_1)\cdots b_d(t_s)) \prod_{i=0}^{d-1} l_{d-1-i}^{n_iq ^i}.$$
  This implies that $\operatorname{BC}_{q^d-N, \widetilde{\chi}^{-1}}$ is $P$-integral.
  Moreover, $\operatorname{BC}_{q^d-N, \widetilde{\chi}^{-1}}\equiv 0\pmod{P}$ if and only if $\operatorname{ev}_{\chi} (\mathbb{B}_{s'})\equiv 0\pmod{P}$.
 We now set $$[\chi]=\{ \chi^{q^i}, i\geq 0\}$$ and we consider the element in $A[\FF_a][\Delta_a]$:
 $$F= \sum_{\psi \in [\chi]}\operatorname{ev}_{\psi}(\mathbb{B}_{s'})e_{\psi}.$$ In fact, 
 we have that $F\in A[\Delta_a]$.
 We also set:
 $$e_{[\chi]}= \sum_{\psi \in [\chi]} e_{\psi}\in \FF[\Delta_a].$$ 
 We deduce from Corollary \ref{theorem6} that
 $$\operatorname{Fitt}_{e_{[\chi]}A[\Delta_a]} e_{[\chi]} (H_a)= Fe_{[\chi]}A[\Delta_a].$$
 This implies that
 $$\operatorname{Fitt}_{e_{[\chi]}\APhat[\FF_a][\Delta_a]} e_{[\chi]} (H_a\otimes_A \APhat[\FF_a])= Fe_{[\chi]}\APhat[\FF_a][\Delta_a].$$
 Therefore,
 $$\operatorname{Fitt}_{\APhat[\FF_a]}e_\chi(H_a\otimes_A \APhat[\FF_a])= \operatorname{ev}_{\chi} (\mathbb{B}_s)\APhat[\FF_a].$$
  \end{proof}
\begin{remark}\label{Remark HR} {\em Our approach in the above proof finds its origins in \cite{ANG&TAE} where an alternative proof of the Herbrand-Ribet Theorem for function fields \cite{TAE3} is given, based on an equivariant class number formula and furnishing the description of the Fitting ideals of certain Taelman class modules in terms of generalized Bernoulli-Carlitz numbers.
Certain congruences modulo $P$ for these numbers are used to complete the proof. As we have seen in the previous sections, the "generic" class number formula gives, by specialization, the Fitting ideal of these class modules. In particular, we do not need an equivariant class number formula and furthermore, again by a specialization argument,  we additionally get  congruences  
for the generalized Bernoulli-Carlitz numbers.}
 \end{remark}
 
\section{Link with other types of $L$-series}\label{connection}
In this Section we shortly explain the link between our $L$-series values
and the global $L$-functions of Goss, Taguchi-Wan and B\"ockle-Pink
(see \cite{BOE, BOE&PIN, TAG&WAN}). We also refer the reader to the recent lecture notes written by Taelman \cite{TAE4}. The notation of this Section
slightly differs from the notation of the rest of this paper.

The Carlitz module is usually seen as a functor from $A$-algebras to 
$A$-modules. In fact, it can also be viewed more naturally as a functor defined over the 
larger category of {\em $\tau$-modules}. We present this construction.

\subsection{The Carlitz functor over $\tau$-modules}
We consider a $\FF$-algebra
$\boldsymbol{A}$ and the ring $R=A\otimes_{\FF}\boldsymbol{A}$, endowed with the $\boldsymbol{A}$-linear endomorphism $\tau$ defined by 
$\tau(a\otimes b)=a^q\otimes b$. 

\begin{Definition}{\em 
A {\em $\tau$-module} $M$ is an $R$-module $M$ of finite type together with 
an $\boldsymbol{A}$-linear endomorphism $\tau_M$ such that, for 
$a\in R$ and $m\in M$, $\tau_M(am)=\tau(a)\tau_M(m)$ (we say that
$\tau_M$ is {\em semi-linear}). A {\em morphism} $f:M_1\rightarrow M_2$ of $\tau$-modules $M_1,M_2$ is 
a morphism of $R$-modules which commutes with the $\tau$-module 
structures.}\end{Definition}

We define the {\em Carlitz functor $C$}
from the category of $\tau$-modules to the category of $R$-modules in the following way. Let $M$ be a $\tau$-module, with semi-linear endomorphism $\tau_M$. Then, $C(M)$ is the $R$-module having $M$ as underlying $\boldsymbol{A}$-module, and where the multiplication $C_\theta$ by $\theta$ is given by $C_\theta=\theta+\tau_M$ (depending on $M$). It is easy to show that this gives rise
to a functor. This functor is faithful, but it is not fully faithful.

\begin{Remark}{\em We can also define the Carlitz functor on certain {\em $\tau$-sheaves}.}\end{Remark}

Essentially, the basic case of this paper
is $\boldsymbol{A}=\FF[\undt{s}]$. But we also considered 
$\boldsymbol{A}=\FF(\undt{s})$ and $\boldsymbol{A}$
an algebraic extension of $\FF$. We have studied the case
of $M$ free of rank one, that is, $M=R=A\otimes_\FF\boldsymbol{A}$, with $\tau_M=\alpha\tau$,
with $\alpha\in R\setminus\{0\}$. Indeed, if $M=R=A[\undt{s}]$ and 
$\tau_M=\alpha\tau$ with $\alpha\in R\setminus\{ 0\}$, then $C(M)$ is  
the structure of $A[\undt{s}]$-module induced on $A[\undt{s}]$ by the Drinfeld module of rank 
one of parameter $\alpha$.

Note that this is in apparent conflict with Definition \ref{drinfelddef} where we have defined Drinfeld modules over Tate algebras (and the parameter varies
in Tate algebras).
In fact, in the settings of the present Section, Definition \ref{drinfelddef} corresponds, with $M$ as above, to an analytic realization (at the place infinity) of
$C(M)$, just as the Carlitz module structure on $\CC_\infty$ is usually considered as an analytic realization of the Carlitz module structure $C(A)$ over $A$. 

\subsection{Exponential functions} The ring $R$ is equipped with the norm $\|\cdot\|$ defined by 
$$\|\sum_ix_i\otimes y_i\|=\max_i\|x_i\|,$$ where the sum is finite and for all $i$, $x_i\in A$ and the $y_i$'s are linearly independent over $\FF$.
Let $M$ be a $\tau$-module. We suppose that it is endowed with a norm 
$\|\cdot\|_M$ such that $\|am\|_M=\|a\|\|m\|_M$, where $a\in R$ and $m\in M$. Then, the {\em exponential function} $\exp_M$ 
is the well defined, continuous, open map
$$\exp_M:\widehat{M}\rightarrow C(\widehat{M})$$ with $\widehat{M}=M\widehat{\otimes}_AK_\infty$,
defined by $\exp_M(m)=\sum_{i\geq 0}D_i^{-1}\tau_M^i(m)$. It satisfies
$$\exp_M(am)=a.\exp_M(m),$$ where the action of $a\in R$ in the right is that 
given by the module structure of $C(\widehat{M})$. This notably happens when
$R=A[\undt{s}]$ and 
$M=\TT_s\supset R$ with $\tau_M=\alpha\tau$, $\alpha\in
R\setminus\{0\}$ so we recover the exponential functions of \S \ref{explog}.

\subsection{$L$-series values revisited} We consider here a new variable $T$.
Let $M$ be a $\tau$-module which is free of finite rank over $R$ and $P$ 
a prime of $A$. Then, the module $M/PM$ is free of finite rank over $\boldsymbol{A}$. The {\em $L$-series value at one} of $M$ is:
$$L(M,1)=\prod_P\det_{\boldsymbol{A}[T]}(1-T\tau_M|M/PM)^{-1}|_{T^{\deg_\theta(P)}\mapsto P^{-1}}\in1+\theta^{-1}\boldsymbol{A}[[\theta^{-1}]].$$ 
By \cite[Lemma 8.2]{BOE2}
we note that for all $P$, $\det_{\boldsymbol{A}[T]}(1-T\tau_M|M/PM)^{-1}\in1+T^d\boldsymbol{A}[[T^d]]$, where $d$ is the degree of $P$. Hence,
we can replace $T^d$ by $P^{-1}$ in the above formal series and we 
get $\det_{\boldsymbol{A}[T]}(1-T\tau_M|M/PM)^{-1}|_{T^d=P^{-1}}\in1+P^{-1}\boldsymbol{A}[[P^{-1}]]$. Since for all $d$ there are finitely many
primes $P$ with degree $d$, the product defining $L(M,1)$ converges
in $K_\infty\widehat{\otimes}_{\FF}\boldsymbol{A}$ to an element 
which belongs to $1+\theta^{-1}\boldsymbol{A}[[\theta^{-1}]]$.
This is a variant of the 
value of the {\em global $L$-function of $M$ at one}, following Goss, Taguchi-Wan, and B\"ockle-Pink (see \cite{BOE&PIN} for the 
definition of the global $L$-function associated to a $\tau$-sheaf).

With such a module $M$, we have that for all $P$,
$C(M/PM)$ is also free of finite rank over $\boldsymbol{A}$ (observe that the identity map $M\rightarrow C(M)$ induces  an isomorphism of $\boldsymbol{A}$-modules between $M/PM$ and $C(M/PM)$).
The {\em $\mathcal{L}$-value at one} of $C(M)$ is 
$$\mathcal{L}(C(M),1)=\prod_{P}\frac{\left[\frac{M}{PM}\right]_R}{\left[C(\frac{M}{PM})\right]_R}\in1+\theta^{-1}\boldsymbol{A}[[\theta^{-1}]].$$
Here, $[N]_R$ denotes the monic generator (in $\theta$) of the 
Fitting ideal of an $R$-module $N$ which is free and finitely generated over
$\boldsymbol{A}$. This is essentially the $L$-series value of Taelman in \cite{TAE2};
the product converges as a consequence of the next Lemma:
\begin{lemma}
Let $M$ be a $\tau$-module which is free of finite rank over $R$.
Then, $L(M,1)=\mathcal{L}(C(M),1)$.
\end{lemma}
We omit the proof as this follows essentially the same lines of 
various other proofs in our paper. In particular, if $\phi$ is a Drinfeld
$A[\undt{s}]$-module of rank one of parameter $\alpha\in A[\undt{s}]\setminus \{ 0\}$, $L(1,\phi)=\mathcal{L}(C(A[\undt{s}]),1)$
can also be constructed starting from the Euler factors of the $L$-series $L(M,1)$
of the $\tau$-module $M=A[\undt{s}]$ with $\tau_M=\alpha\tau$. 
In this respect, the $L$-series values that we introduce in this paper
can be viewed as first examples of an alternative way of defining 
global $L$-functions, which moreover are rigid analytic, in contrast with the Goss' $L$-functions.

\subsection{The log-algebraic Theorem again} As a final remark,
we point out that also Theorem \ref{logalg} can be viewed as 
a statement of integrality of the value of an exponential function
associated to a certain $\tau$-module.

We consider again
the algebra $\mathcal{A}_r$ of \S \ref{anderson}; with the $\FF$-endomorphism $\tau_{\mathbb{M}}:=\tau$ and with the structure of $R$-module defined there (with $R=A[\undt{r+1}]$), it becomes a $\tau$-module 
that we denote by $\mathbb{M}$. But this is not sufficient 
to interpret Theorem \ref{logalg}.

Now, as we have seen, the structure of $A[\undt{r+1}]$-module 
extends to a structure of $\TT_{r+1}$-module. We then have the exponential function
$\exp_{\mathbb{M}}:\mathbb{M}\rightarrow C(\mathbb{M})$ and we have proved that $$\exp_{\mathbb{M}}(L(1,\phi).[ZX_1\cdots X_r])=\exp_\phi(L(1,\phi)).[ZX_1\cdots X_r],$$
where $\phi$ is the Drinfeld module of rank one of parameter $\alpha=t_{r+1}(t_1-\theta)\cdots(t_r-\theta)$, $\exp_\phi$ its exponential function, and $L(1,\phi)$
its $L$-series value at one. Of course, this is just a way to rephrase Theorem \ref{logalg}.

\section*{Acknowledgement}
The authors sincerely thank David Goss, Matthew Papanikolas, Rudolph Perkins and Lenny Taelman for interesting discussions, hints and useful remarks on earlier versions of this text. The authors thank the referee for his corrections and suggestions that helped us to improve the original text.

\section{Appendix by Florent Demeslay.}\label{appendix}

We shall work with Drinfeld $\FF(\undt{s})[\theta]$-modules rather than with $A[\undt{s}]$-modules. 
We set, from now on, $\ring{}=\FF(\undt{s})[\theta]$.
As we have already seen, the benefit 
of this assumption comes from the fact that $\ring{}$ is a principal ideal domain.
We keep using the same notation adopted in the previous sections.

Let us choose $\alpha \in \ring{}\setminus \{ 0\}$ and let us consider the Drinfeld $\ring{}$-module of rank one and parameter $\alpha$, that is,
the injective homomorphism of $\FF(\undt{s})$-algebras
$$\phi: \ring{}\rightarrow \operatorname{End}_{\FF(\undt{s})-\text{lin.}}(K(\undt{s})_{\infty})$$ given by $\phi_{\theta}= \theta +\alpha \tau,$ where we recall that $\tau:K(\undt{s})_{\infty}\rightarrow K(\undt{s})_{\infty}$ is the continuous (for the $\frac{1}{\theta}$-adic topology) morphism of $\FF(\undt{s})$-algebras given by $\tau(\theta)=\theta^q.$
Let $M$ be an $\ring{}$-algebra together with a $\FF(\undt{s})$-endomorphism $\tau_M:M\rightarrow M$ such that $$\tau_M(fm)= \tau(f) \tau_M(m),\quad f\in \ring{},\quad m\in M.$$ We denote by $\phi(M)$ the $\FF(\undt{s})$-vector space $M$ equipped with the structure of $\ring{}$-module induced by $\phi,$ e.g. $\forall m\in  \phi(M):$
$$\theta.m= \alpha \tau_M(m)+\theta m.$$ 
We recall that we have the exponential function associated to $\phi$, which is a $\FF(\undt{s})$-linear endomorphism of $K(\undt{s})_{\infty}$ defined by
$$\exp_{\phi} =\sum_{i\geq 0}\frac{1}{D_i}\tau_{\alpha}^i,$$
where $\tau_{\alpha}=\alpha \tau.$  Note that $\exp_{\phi}$ is a morphism of $\ring{}$-modules from $K(\undt{s})_{\infty}$ to $\phi (K(\undt{s})_{\infty}). $ 
We recall that if $u(\alpha)$ is the maximum of  the integer part of $\frac{r-q}{q-1}$  and $0$, with $r=-v_\infty(\alpha)$, then $\exp_{\phi}$
induces an isometric $\FF(\undt{s})$-linear automorphism $\mathfrak{m}_{K(\undt{s})_{\infty}}^{u(\alpha)+1} \rightarrow\mathfrak{m}_{K(\undt{s})_{\infty}}^{u(\alpha)+1}$. Notice that $\mathfrak{m}_{K(\undt{s})_{\infty}}=\frac{1}{\theta}\FF(\undt{s})[[\frac{1}{\theta}]]$.

Observe that a sub-$\FF(\undt{s})$-vector space $M$ of $K_{s, \infty}$ is discrete (for the $\frac{1}{\theta}$-adic topology) if and only  if there exists an integer $n\geq 1$ such that $M\cap \mathfrak{m}_{K(\undt{s})_{ \infty}}^n=\{0\}$ (note that this is equivalent to the fact that  the intersection $M\cap\mathfrak{m}_{K(\undt{s})_{\infty}}$ is a finite dimensional $\FF(\undt{s})$-vector space).

\begin{lemma}
\label{lemmaA1}
Let $M\not =\{0\}$ be a sub-$\ring{}$-module of $K(\undt{s})_{\infty}.$ The following assertions are equivalent:
\begin{enumerate}
\item $M$ is discrete,
\item $M$ is a free $\ring{}$-module of rank one.
\end{enumerate}
\end{lemma} 
\begin{proof}
The fact that the property (2) implies the property  (1) is clear. Let us prove that the property (1) implies the property (2). Let $f$ be a non-zero element of $M$. Then, $\ring{}\subset f^{-1} M$ and $f^{-1}M$ is discrete in $K(\undt{s})_{ \infty}$. Thus, we can assume that $\ring{}\subset M.$ We now observe that we have a direct sum of 
$\FF(\undt{s})$-vector spaces:
$$K(\undt{s})_{\infty} =\ring{}\oplus \mathfrak{m}_{K(\undt{s})_{,\infty}}=\ring{}\oplus\frac{1}{\theta}\FF(\undt{s})\left[\left[\frac{1}{\theta}\right]\right].$$
Since $M$ is discrete we deduce from the above decomposition that $M/\ring{}$ is a finite $\FF(\undt{s})$-vector space. But $M/\ring{}$ is also a $\ring{}$-module, hence a torsion $\ring{}$-module. Therefore there exists $r\in \ring{}\setminus\{0\}$ such that $rM\subset \ring{}.$ Since $\ring{}$ is a noetherian ring, we obtain that $M$ is a finitely generated $\ring{}$-module of rank $1$. Since $\ring{}$ is a principal ideal domain we deduce that $M$, as an $\ring{}$-module, is free of rank one.
\end{proof}

\begin{Remark}{\em We recall from \S \ref{calv} the following $\ring{}$-module:
$$\calv_{\phi}= \frac{\phi(K(\undt{s})_{\infty})}{\phi(\ring{})+\exp_{\phi}(K(\undt{s})_{ \infty})},$$
which is a $\FF(\undt{s})$-vector space of dimension $\leq u(\alpha)$.
We notice that $\ring{}+\mathfrak{m}_{K(\undt{s})_{\infty}}^{u(\alpha)+1}\subset \ring{}+\exp_{\phi}(K(\undt{s})_{ \infty})$.  
We observe that $\ring{}$ and $\operatorname{Ker}(\exp_{\phi})$ are discrete sub-$\ring{}$-modules of $K(\undt{s})_{\infty}$ which implies that $\exp_{\phi}^{-1}(\ring{})$ is a discrete sub-$\ring{}$-module of $K(\undt{s})_{\infty}$.
The exponential $\exp_{\phi}$ then produces an exact sequence of $\FF(\undt{s})$-vector spaces
$$0\rightarrow \frac{K(\undt{s})_{\infty}}{\exp_{\phi}^{-1}(\ring{})+\mathfrak{m}_{K(\undt{s})_{\infty}}^{u(\alpha)+1}} \rightarrow \frac{K(\undt{s})_{\infty}}{\ring{}+\mathfrak{m}_{K(\undt{s})_{\infty}}^{u(\alpha)+1}}\rightarrow \calv_{\phi} \rightarrow 0.$$
In particular, $\exp_{\phi}^{-1}(\ring{})\not =\{0\}$ and we obtain that $\exp_\phi^{-1}(\ring{})$ is free of rank one
by using Lemma \ref{lemmaA1}.}
\end{Remark}
\subsection{$L$-series} 
Let $P$ be a prime of $A$. Then the $\ring{}$-module $\phi(\frac{\ring{}}{P\ring{}})$ is finitely generated and torsion. One can show that the  product over the primes of $A$
$$\mathcal{L}(\phi/\ring{}) =\prod_{P}\frac{[\frac{\ring{}}{P\ring{}}]_{\ring{}}}{[\phi(\frac{\ring{}}{P\ring{}})]_{\ring{}}}$$
converges in $K(\undt{s})_{\infty}.$ We will only give a sketch of 
proof of the next Theorem, as the proof is very close to ideas developed by Taelman in \cite{TAE2}. 
\begin{theorem}[class number formula] The following identity holds in $K(\undt{s})_{\infty}$:
\label{theoremA1}
$$\mathcal{L}(\phi/\ring{}) =[\calv_{\phi}]_{\ring{}}[\ring{}:\exp_{\phi}^{-1}(\ring{})]_{\ring{}}.$$
\end{theorem}

\subsection{Nuclear operators and determinants }
This Section is inspired by \cite[Section 2]{TAE2}. Let $(V,\|.\|)$ be a $\FF(\undt{s})$-vector space equipped with a non-archimedean absolute value which is trivial on $\FF(\undt{s})$ and such that every open $\FF(\undt{s})$-subspace $U\subset V$ is of finite $\FF(\undt{s})$-co-dimension. Let's give a typical example of such objects: let $M$ be a  non-trivial, discrete $\ring{}$-submodule of $K(\undt{s})_{ \infty},$ then $V=\frac{K(\undt{s})_{ \infty}}{M}$ satisfies the above hypothesis.

Let $f$ be a continuous endomorphism of $V,$ we say that $f$ is {\em locally contracting} if there exists an open subspace  $U\subset V$ and a real number $0<c<1$ such that, for all $v\in U$, $$\|f(v)\| \leq c\|v\|.$$ Any such open subspace $U\subset V$ which moreover satisfies $f(U)\subset U$ is called a nucleus for $f.$ Observe that any locally contracting continuous endomorphism of $V$ has a nucleus.  Let's give an example: the map $$\tau_{\alpha}: \frac{K(\undt{s})_{ \infty}}{\ring{}} \rightarrow \frac{K(\undt{s})_{\infty}}{\ring{}} $$ is locally contracting and the image of $\mathfrak{m}_{K(\undt{s})_{\infty}}^{u(\alpha)+2}$ in $\frac{K(\undt{s})_{\infty}}{\ring{}} $ is a nucleus (just observe that for $f\in K(\undt{s})_{\infty},$  $v_{\infty}(\tau_{\alpha}(f))\geq v_{\infty}(f)+1$ if and only if $v_{\infty}(f) \geq \frac{r-1}{q-1}$).

Observe that any finite collection of locally contracting endomorphisms of $V$ has a common nucleus (see for example \cite{TAE2}, Proposition 6). Furthermore if $f$ and $g$ are locally contracting, then so are the sum $f+g$ and the composition $fg.$

For any integer $N\geq 0,$ we set:
$$\frac{V[[Z]]}{Z^N}= V\otimes_{\FF(\undt{s})}\frac{\FF(\undt{s})[[Z]]}{Z^N},$$
and we denote by $V[[Z]]$ the inverse limit of the $V[[Z]]/Z^N$ equipped with the limit topology. Observe that any continuous $\FF(\undt{s})[[Z]]$-endomorphism $$F: V[[Z]]\rightarrow V[[Z]]$$ is of the form:
$$F=\sum_{n\geq 0} f_n Z^n,$$
where $f_n$ is a continuous $\FF(\undt{s})$-endomorphism of $V.$ 

Similarly, any continuous $\frac{\FF(\undt{s})[[Z]]}{Z^N}$-linear endomorphism of $\frac{V[[Z]]}{Z^N}$ is of the form $\sum_{n=0}^{N-1} f_n Z^n.$ We say that a continuous $\FF(\undt{s})[[Z]]$-linear endomorphism $F$ of $V[[Z]]$ (resp. of $\frac{V[[Z]]}{Z^N}$) is {\em nuclear} if for all $n$ (resp. for all $n<N$), the $\FF(\undt{s})$-endomorphism $f_n$ of $V$ is locally contracting.
\noindent Let $F$ be a nuclear endomorphism of $\frac{V[[Z]]}{Z^N}.$ Let $U_1$ and $U_2$ be common nuclei for the $f_n,$ $n<N.$ Since  Proposition 8 in \cite{TAE2} is valid in our context,  $$\det{}_{\frac{\FF(\undt{s})[[Z]]}{Z^N}}\left(1+F|_{\frac{V}{U_i}[[Z]]/Z^N}\right)\in \frac{\FF(\undt{s})[[Z]]}{Z^N}$$ is independent of $i\in \{ 1,2\}.$ We denote this determinant by $$\det{}\;_{\frac{\FF(\undt{s})[[Z]]}{Z^N}}\left(1+F|_V\right).$$ If $F$ is a nuclear endomorphism of $V[[Z]]$, then we denote by $\det_{\FF(\undt{s})[[Z]]} (1+F|_V)$ the unique power series that reduces to $\det_{\frac{\FF(\undt{s})[[Z]]}{Z^N}}(1+F|_V)$ modulo $Z^N$ for any $N.$
\noindent Note that Proposition 9, Proposition 10, Theorem 2 and Corollary 1 of \cite{TAE2} are also valid in our context.

\subsection{Taelman's trace formula}
Observe that any element in $\ring{}[\tau] \tau$ induces a $\FF(\undt{s})$-linear continuous endomorphism of $\frac{K(\undt{s})_{ \infty}}{\ring{}}$ which is locally contracting. Denote by $\ring{}[\tau][[Z]]$ the ring of formal power series in $Z$ with coefficients in $\ring{}[\tau],$ the variable $Z$ being central (i.e. $\tau Z=Z\tau$). 
 Let $P_1, \cdots, P_n$ be $n$ distinct primes of $A. $ Set $S=\{ P_1, \ldots,P_n\}.$  Let us set $$R=\ring{}\left[\frac{1}{P_1}, \ldots, \frac{1}{P_n}\right].$$
 Let $P$ be a monic prime of $A.$ Let $K(\undt{s})_{P}$ be the $P$-adic completion of $\FF(\undt{s})(\theta)$ (with respect to the $P$-adic valuation on $\FF(\undt{s})(\theta)$ which is trivial on $\FF(\undt{s})$ and the usual one on $K$). Observe that every element of $K(\undt{s})_{P}$ can be written in a unique way:
$$\sum_{i\geq m} x_i P^i,$$
where $m\in \mathbb Z,$ $ x_i \in \ring{}$ of degree in $\theta$ strictly less than $\deg_{\theta} P.$

\noindent We also define:
$$K(\undt{s})_{S}=K(\undt{s})_{\infty}\times K(\undt{s})_{ P_1}\times\cdots \times  K(\undt{s})_{P_n}.$$
Observe that $R$ is discrete in $K(\undt{s})_{ S}.$ Let $P$ be a prime of $A,$ $P\not = P_1, \cdots, P_n.$ 
 Let $R_{P}$ be the valuation ring of $K(\undt{s})_{ P}.$ Then:
$$K(\undt{s})_{P} = R_{P}+ R[1/P].$$
Furthermore, the inclusion $\ring{}\subset R$ induces an isomorphism:
$$\frac{\ring{}}{P\ring{}}\simeq \frac{R}{PR}.$$
Let $F \in R[\tau][[Z]] \tau Z.$ Then $F$ defines $\FF(\undt{s})$-endomorphisms of $\frac{K(\undt{s})_{S}}{R}[[Z]],$ $\frac{R}{PR}[[Z]],$ $\frac{K(\undt{s})_{ S}\times K(\undt{s})_{ P}}{R[1/P]}[[Z]].$ Now Taelman's  localization Lemma (\cite{TAE2} Lemma 1) remains valid in our case: 
\begin{lemma}
\label{lemmaA6} Let us choose $F \in R[\tau][[Z]] \tau Z.$ Then:
$$\det{}_{\FF(\undt{s})[[Z]]}\left(1+F|_{\frac{R}{PR}}\right) =\frac{\det_{\FF(\undt{s})[[Z]]}\left(1+F|_{\frac{K(\undt{s})_{S}\times K(\undt{s})_{P}}{R[1/P]}}\right)}{\det_{\FF(\undt{s})[[Z]]}\left(1+F|_{\frac{K(\undt{s})_{S}}{R}}\right)}.$$
\end{lemma}
\noindent We also have  in our case:
\begin{theorem}
\label{theoremA2}
Let $F \in \ring{}[\tau][[Z]] \tau Z.$ Then we have an equality in $\FF(\undt{s})[[Z]]:$

$$\prod_{ P\, {\rm monic \, prime\,  of }\, A}\det{}_{\FF(\undt{s})[[Z]]}\left(1+F|_{\frac{\ring{}}{P\ring{}}}\right) = \det{}_{\FF(\undt{s})[[Z]]}\left(1+F|_{\frac{K(\undt{s})_{ \infty}}{\ring{}}}\right)^{-1}.$$
\end{theorem}
\begin{proof} This is a consequence of Lemma \ref{lemmaA6} and the proof of Theorem 3 in \cite{TAE2}. Note that in our case  in \cite {TAE2} page 383 line $-5$ we replace the original assumption of Taelman by the assumption that $R$ has no maximal  ideal of the form  $PR,$ $P$ monic prime of $A,$ such that $\dim_{\FF(\undt{s})}\left(\frac{R}{PR}\right)< D$.
\end{proof}

\subsection{Fitting ideals}
Let $f:\, K(\undt{s})_{\infty}\rightarrow K(\undt{s})_{ \infty}$ be a continuous $\FF(\undt{s})$-linear map. Let $M_1$ and $M_2$ be two  free  $\ring{}$-modules of rank one in $K(\undt{s})_{\infty}$ such that $f(M_1)\subset M_2.$ 
Then $f$ induces a $\FF(\undt{s})$-continuous linear map $$f: \frac{K(\undt{s})_{\infty} }{M_1} \rightarrow 
\frac{K(\undt{s})_{\infty} }{M_2}.$$ We say that $f$ is {\em infinitely tangent to the identity} on $K(\undt{s})_{ \infty}$ if for any $N\geq 0$ there exists an open $\FF(\undt{s})$-subspace $U_N\subset K(\undt{s})_{ \infty}$ such that the following 
properties hold:
\begin{enumerate}
\item $U_N \cap M_1=U_N\cap M_2=\{ 0\},$
\item $f$ restricts to an isometry between the images of $U_N,$
\item For all $u\in U_N, \, v_{\infty} (f(u)-u)\geq N+ v_{\infty}(u).$
\end{enumerate}

If $f \in K(\undt{s})_{\infty}[[\tau]]$ is such that  this power series is convergent on $K(\undt{s})_{\infty}$ and satisfies  that $f(M_1)\subset M_2,$ for some free  $\ring{}$-modules of rank one $M_1$ and $M_2,$ then, by the proof of Proposition 12 in \cite{TAE2}, $f$ is infinitely tangent to the identity on $K(\undt{s})_{\infty}.$ A typical example is given by: $M_1=\exp_{\phi}^{-1} (\ring{}), $ $M_2=\ring{}$ and $f=\exp_{\phi}.$  
\noindent Now let $M_1, M_2$ be two free  $\ring{}$-modules of rank one in $K(\undt{s})_{\infty}.$ Let $H_1, H_2$ be two finite dimensional $\FF(\undt{s})$-vector spaces  that are also $\ring{}$-modules. Set:
$$L_i=\frac{K(\undt{s})_{\infty}}{M_i}\times H_i.$$
Let $f: L_1\rightarrow L_2$ be a $\FF(\undt{s})$-linear map which is bijective and continuous. We shall write:
$$\Delta_f= \frac{1-f^{-1}\theta f Z}{1-\theta Z}-1= \sum_{n\geq 1} (\theta - f^{-1}\theta f )\theta^{n-1} Z^{n}.$$
We observe that $\Delta_f$ defines a $\FF(\undt{s})$-endomorphism of $L_1[[Z]].$ 
Let's assume that $f$ induces (see \cite{TAE2} page 385 line $-6$) a continuous $\FF(\undt{s})$-linear map  $$\frac{K(\undt{s})_{ \infty}}{M_1}\rightarrow \frac{K(\undt{s})_{ \infty}}{M_2}$$ which is infinitely tangent to the identity on $K(\undt{s})_{\infty}.$  Then, by the proof of Theorem 4 in  \cite{TAE2}, we get that $\Delta_f$ is nuclear, and
$$\det{}_{\FF(\undt{s})[[Z]]}(1+\Delta_f|_{L_1})_{Z=\theta^{-1}}=[M_1:M_2]_{\ring{}} \frac{[H_2]_{\ring{}}}{[H_1]_{\ring{}}}.$$

\subsection{Proof of Theorem \ref{theoremA1}}
We set, as in \cite{TAE2} \S 5:
$$F= \frac{1-\phi_{\theta} Z}{1-\theta Z}-1= \sum_{n\geq 1} (\theta-\phi_{\theta})\theta^{n-1}Z^n\in \ring{}[\tau][[Z]] \tau Z.$$
We have:
$$\mathcal{L}(\phi/\ring{})= \prod_{P{\rm monic\, prime \, in \,} A} \left(\det{}_{\FF(\undt{s})[[Z]]}(1+F|_{\frac{\ring{}}{P\ring{}}})\right)_{Z=\theta^{-1}}^{-1}.$$
By Theorem \ref{theoremA2}, we get:
$$\mathcal{L}(\phi/\ring{}) = \det{}_{\FF(\undt{s})[[Z]]}\left( 1+F|_{\frac{K(\undt{s})_{\infty}}{\ring{}}}\right)|_{Z=\theta^{-1}}.$$
We consider the short exact sequence of $\ring{}$-modules induced by $\exp_{\phi}:$
$$0\rightarrow \frac{K(\undt{s})_{\infty}}{\exp_{\phi}^{-1}(\ring{})}\rightarrow \frac{\phi( K(\undt{s})_{\infty})}{\phi (\ring{})} \rightarrow \calv_{\phi} \rightarrow 0\,.$$
Since $\ring{}$ is a principal ideal domain and the $\ring{}$-module  $\frac{K(\undt{s})_{\infty}}{\exp_{\phi}^{-1}(\ring{})}$ is divisible, this sequence splits. The choice of a section gives an $\ring{}$-isomorphism:
$$\frac{K(\undt{s})_{\infty}}{\exp_{\phi}^{-1}(\ring{})}\times \calv_{\phi} \simeq  \frac{\phi( K(\undt{s})_{\infty})}{\phi (\ring{})}.$$
This isomorphism gives rise to an isomorphism of $\FF(\undt{s})$-vector space:
$$\frac{K(\undt{s})_{\infty}}{\exp_{\phi}^{-1}(\ring{})}\times \calv_{\phi} \simeq  \frac{K(\undt{s})_{\infty}}{ \ring{}}.$$
We denote this map by $f.$ Then, by the proof of Lemma 6 in \cite{TAE2}, $f$ is infinitely tangent to the identity on $K(\undt{s})_{\infty}.$ But observe that  on $\frac{K(\undt{s})_{\infty}}{\ring{}}[[Z]]:$
$$1+F = \frac{1-f\theta f^{-1} Z}{1-\theta Z}.$$
Thus:
$$\det{}_{\FF(\undt{s})[[Z]]}\left(1+F |_{\frac{K(\undt{s})_{\infty}}{\ring{}}}\right)|_{Z=\theta^{-1}}= [\calv_{\phi}]_{\ring{}}[\ring{}:\exp_{\phi}^{-1} (\ring{})]_{\ring{}}.$$ The proof of our Theorem follows.

  \end{document}